\documentclass[12pt]{amsart}

\usepackage[draft]{say}

\usepackage{graphicx}
\usepackage{amssymb}
\usepackage{amsmath, amscd}
\usepackage{dbnsymb}
\usepackage[hidelinks]{hyperref}

\usepackage{physics}
\usepackage{amsmath}
\usepackage{tikz, tikz-cd}
\usepackage{mathdots}
\usepackage{yhmath}
\usepackage{cancel}
\usepackage{color}
\usepackage{siunitx}
\usepackage{array}
\usepackage{multirow}
\usepackage{amssymb}
\usepackage{gensymb}
\usepackage{tabularx}
\usepackage{extarrows}
\usepackage{booktabs}
\usetikzlibrary{fadings}
\usetikzlibrary{patterns}
\usetikzlibrary{shadows.blur}
\usetikzlibrary{shapes}

\newtheorem{thm}{Theorem}[section]

\newtheorem{theorem}[thm]{Theorem}
\newtheorem{proposition}[thm]{Proposition}
\newtheorem{corollary}[thm]{Corollary}
\newtheorem{lemma}[thm]{Lemma}

\theoremstyle{definition}
\newtheorem{definition}[thm]{Definition}

\newtheorem{construction}[thm]{Construction}

\theoremstyle{remark}
\newtheorem{remark}[thm]{Remark}

\newcommand{\R}{{\mathbb{R}}}
\newcommand{\C}{{\mathbb{C}}}
\newcommand{\Q}{{\mathbb{Q}}}
\newcommand{\Z}{{\mathbb{Z}}}

\newcommand{\T}{\mathbb{T}}

\usepackage[margin=1in,marginparwidth=0.8in, marginparsep=0.1in]{geometry}

\usepackage{epigraph}

\usepackage{wrapfig, framed, caption}

\begin{document}

\title[Skein valued cluster transformations from Legendrian mutation]{Skein valued cluster transformation in  \\ enumerative geometry of Legendrian mutation} 
\author{Matthias Scharitzer}
\address{Centre for Quantum Mathematics, SDU, Campusvej 55, 5230 Odense M, Denmark}
\email{matthias.scharitzer@gmail.com} 
\author{Vivek Shende}
\address{Centre for Quantum Mathematics, SDU, Campusvej 55, 5230 Odense M, Denmark 
$\qquad \qquad$ \& Department of Mathematics, UC Berkeley, Evans Hall, Berkeley CA 94720, USA}
\email{vivek.vijay.shende@gmail.com}

\begin{abstract} 
Under certain hypotheses, we show
that Legendrian surfaces related by
disk surgery will have 
q-deformed augmentation spaces that are 
related by q-deformed 
cluster transformation.  The proof is 
geometric, via considerations of moduli of holomorphic curves.  In fact, our results naturally 
give a more general HOMFLYPT `skein-valued cluster
transformation', of which the q-cluster transformation
is the U(1) specialization. 

We apply our methods to the Legendrians 
associated to cubic planar graphs, where  mutation of graphs lifts 
to Legendrian disk surgery. We show
that their skein-valued mirrors
transform by skein-valued cluster transformation,
and give a formula for the 
skein-valued curve counts on their fillings.  
\end{abstract}

\maketitle

\thispagestyle{empty} 

\section{Introduction}

Let $(V, \xi)$ be a 5-dimensional contact manifold, and $\Lambda \subset V$ be a 2-dimensional Legendrian.   We will always assume that $c_1(\xi) = 0$, that $\Lambda$ has vanishing Maslov class, 
and that all Reeb orbits of $V$ and all Reeb chords of $\Lambda$ have positive index.   
Then, for any index one Reeb chord $\rho$ of $\Lambda$, the skein-valued curve counting of \cite{SOB, bare, ghost} can be applied to the symplectization $(V \times \R, \Lambda \times \R)$, and produces an element 
\begin{equation}\mathbf{A}(\Lambda, \rho) \in Sk(\Lambda \times \R).\end{equation}    
Here, for a $3$-manifold $M$, we write $Sk(M)$ for the HOMFLYPT skein of $M$, which, by definition, is generated over $\Z[a^{\pm}, q^{\pm 1/2}]$ by framed links in $M$, subject to the relations 
in Figure \ref{HOMFLYPT skein}. If $(W, L)$ is a filling of $(V, \Lambda)$, and $\Psi_L \in Sk(L)$ is the count of holomorphic curves ending on $L$, then, under the natural action $Sk(\Lambda \times \R)$ on $Sk(L)$, we have \cite{unknot, Scharitzer-Shende}:
\begin{equation} \label{recursion}
 \mathbf{A}(\Lambda, \rho) \Psi_L = 0.
\end{equation}

Given a Lagrangian disk $D$ ending on $\Lambda$, there is a Legendrian disk surgery operation yielding a new Lagrangian $\Lambda_D$.  (The operation is a Legendrian lift of the relationship between the two Polterovich surgeries of a Lagrangian double point.)  

\vspace{2mm}
The purpose of the present article is to establish and explore the following identity: 
\begin{equation} \label{main formula} \mathbf{A}(\Lambda_D,\rho_+) \mathbf{E}(\partial D) = \mathbf{E}(\partial D)  \mathbf{A}(\Lambda,\rho_-) \end{equation}
We state a precise result in Theorem \ref{cluster transformation from perfect cobordism} below; there are various nontrivial topological and geometric hypothesis required. 
For now let us explain what the terms mean.  So long as there are no Reeb chords between $\Lambda$ and $D$ (as we will assume),  there is a natural bijection between Reeb chords of $\Lambda$ and those of $\Lambda_D$; here $\rho_+$ and $\rho_-$ are two matching such chords.  We have already explained $\mathbf{A}$.  Multiplication in the skein is concatenation in the $\R$ direction of $\Lambda \times \R$.  Finally, $\mathbf{E}$ is an element living in (a certain completion of) the skein of the solid torus, and $\mathbf{E}(\partial D)$ means we insert this element along the curve $\partial D$.  

The element $\mathbf{E}$ is, geometrically, the skein-valued count of holomorphic curves in $\mathbb{C}^3$, ending on the Harvey-Lawson/Aganagic-Vafa brane.  In \cite{unknot}, it was shown that for the standard Legendrian Clifford torus $\T_{Cl} \subset S^5$, one has:
\begin{equation} \label{clifford torus operator}
    \mathbf{A}(\T_{Cl}) = \bigcirc - m + a \ell 
\in Sk(T^2 \times \R),
\end{equation}
It was also shown that the corresponding specialization of Equation  \eqref{recursion} has a unique solution up to scalar multiple in a certain appropriate completion of $Sk(\mathbf{T})$; said solution must, therefore, be $\mathbf{E}$.  An explicit formula for $\mathbf{E}$ can be found in \cite{unknot}, which agrees with the appropriate 1-leg specialization of the topological vertex of \cite{AKMV}.  
In this article we only use the characterization of  $\mathbf{E}$ as the unique solution of \eqref{clifford torus operator}. 

\vspace{2mm}

Equation \eqref{main formula} is a skein-valued lift of a q-cluster transformation in the sense of \cite{Fock-Goncharov-quantum}.  One can see this as follows.  Consider the $U(1)$ or `linking' 
skein, $Lk(M)$, which is by definition
the linear span of framed links over $\Z[q^{\pm1/2}]$ modulo the following relations: 
$$\overcrossing = q^{1/2} \,\, \smoothing \qquad \qquad \qquad \bigcirc = 1$$
Note this quotient factors through $Sk(M)/(a=q^{1/2})$.  

When $S$ is a surface, 
$Lk(S \times \mathbb{R})$ is naturally identified with the quantum torus associated to the symplectic lattice $H_1(S, \Z)$.  Moreover, there is a natural identification of the skein action
$Lk(T^2 \times \R) \circlearrowright Lk(\mathbf{T})$ with the polynomial representation of the quantum torus.  In fact $\mathbf{A}(\T_{Cl})|_{Lk}$ is the $q$-difference equation for the exponential of the $q$-dilogarithm; so $\mathbf{E}|_{Lk}$ is said exponentiated dilogarithm. 
(We give more details around skeins and $q$-dilogarithms in Section \ref{skein dilogarithm}.)  Thus, if we specialize 
equation \eqref{main formula}  to the linking skein, we see a conjugation by a $q$-dilogarithm, which is by definition
the $q$-cluster transformation of \cite{Fock-Goncharov-quantum}.

\begin{figure}
    \includegraphics[scale=0.3]{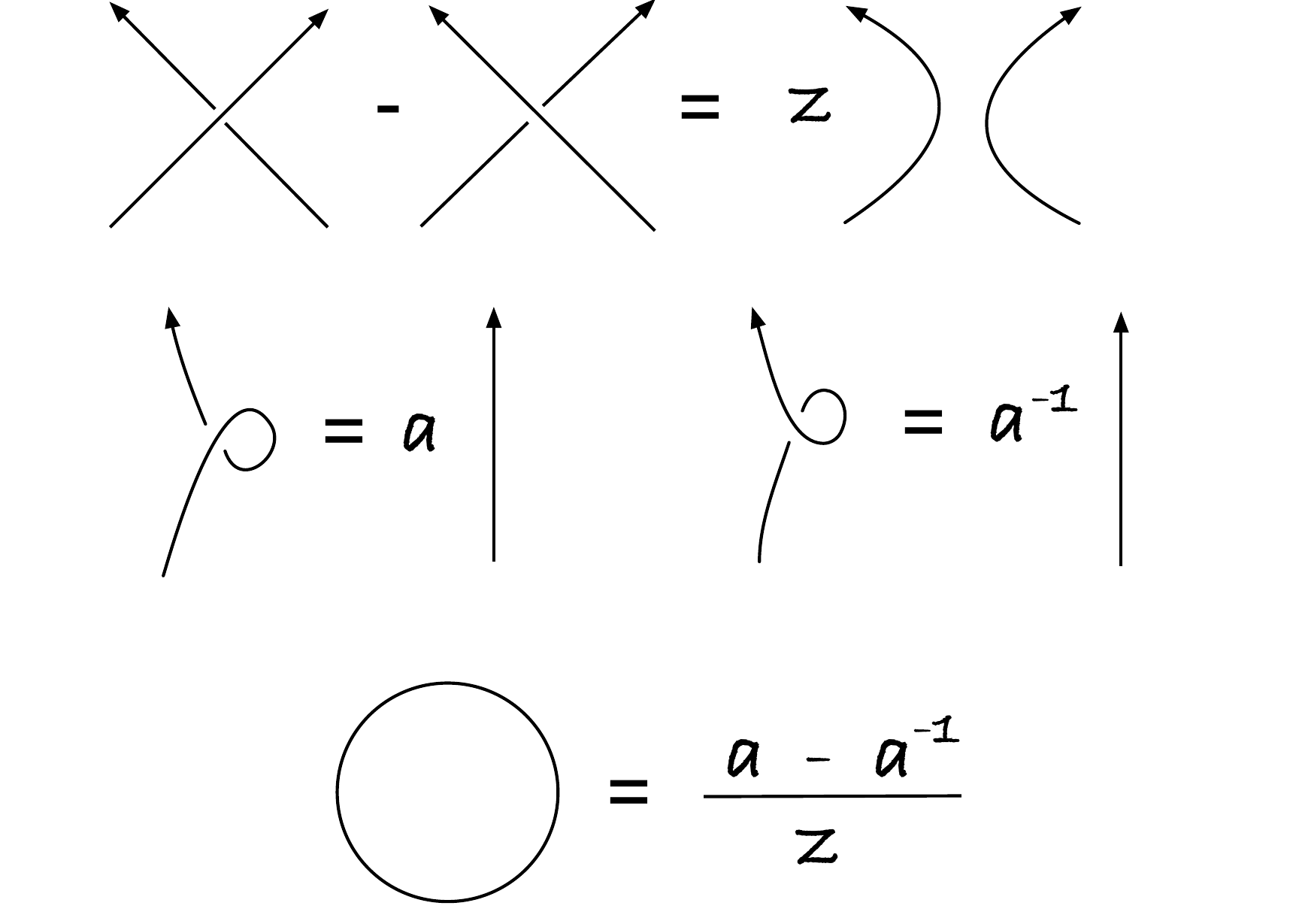}
    \caption{HOMFLYPT skein relations in the variable $z = q^{1/2} - q^{-1/2}$.} 
    \label{HOMFLYPT skein}
\end{figure}

\vspace{2mm}
A class of examples to which our work applies, and by which we were largely motivated, are the cubic planar graph Legendrian surfaces of Treumann and Zaslow in \cite{Treumann-Zaslow}.
The Legendrians
are constructed as follows.  First, take a cubic graph $\Gamma \subset S^2$; then form the (essentially unique) Legendrian 
$L_\Gamma \subset J^1(S^2)$ characterized by the fact that $L_\Gamma \to S^2$ is a double cover branched at the vertices of $\Gamma$, 
such that the crossings of the front projection $L_\Gamma \to S^2 \times \R$  sit over $\Gamma$.   Now implant this into a standard
neighborhood of $S^2 =  \partial \R^3 \subset \partial \C^3 = S^5$.

The various possible trivalent graphs  of given genus are related by edge flips,
as depicted in Figure \ref{Fig:Edge-Flip}. 
For an edge $e$ of a cubic planar graph $\Gamma$, we write $\Gamma_e$ for the result of flipping $e$. 
This system is known to organize an interesting cluster algebra \cite{Fock-Goncharov-Teichmuller}. 
The q-deformation of this cluster algebra was studied by Schrader, Shen, and Zaslow in \cite{SSZ}.  
They consider $\mathcal{T}_q(\Lambda_\Gamma)$, 
a (twist of the) quantum torus associated to $H^1(\Lambda_\Gamma)$, and define by explicit formula 
some element $\mathcal{A}_q^{SSZ} (L_\Gamma) \subset \mathcal{T}_q(\Lambda_\Gamma)$. 
The definition of $\mathcal{A}_q^{SSZ}(\Lambda_\Gamma)$ is entirely algebraic, making no mention of symplectic geometry or holomorphic curves.  However \cite{SSZ}
further conjectured that $\mathcal{A}_q^{SSZ}$ should annihilate
an open Gromov-Witten curve count for certain Lagrangian fillings of $L_\Gamma$, and showed that for a subset of graphs (those obtained by an ``admissible sequence of mutations'' from a ``necklace graph'', such as $g = 2$ example in Figure \ref{Fig:Edge-Flip}), the prescription  
that $\mathcal{A}_q^{SSZ}$ annihilates the count would determine the count uniquely. 

% (In fact, their formula for $\mathcal{A}_q^{SSZ}$ is basically determined by 
% making the simplest guess compatible with their conjecture when $L_\Gamma$ admits a filling which is exact (and should therefore 
% have trivial Gromov-Witten invariant) and demanding that $\mathcal{A}_q^{SSZ} (\Lambda_\Gamma)$  transform by quantum cluster transformation.)  

\begin{figure}
    \begin{picture}(50,75)
    \put(-150,0){\includegraphics[width=12cm, trim=0cm 5cm 0cm 0cm, clip]{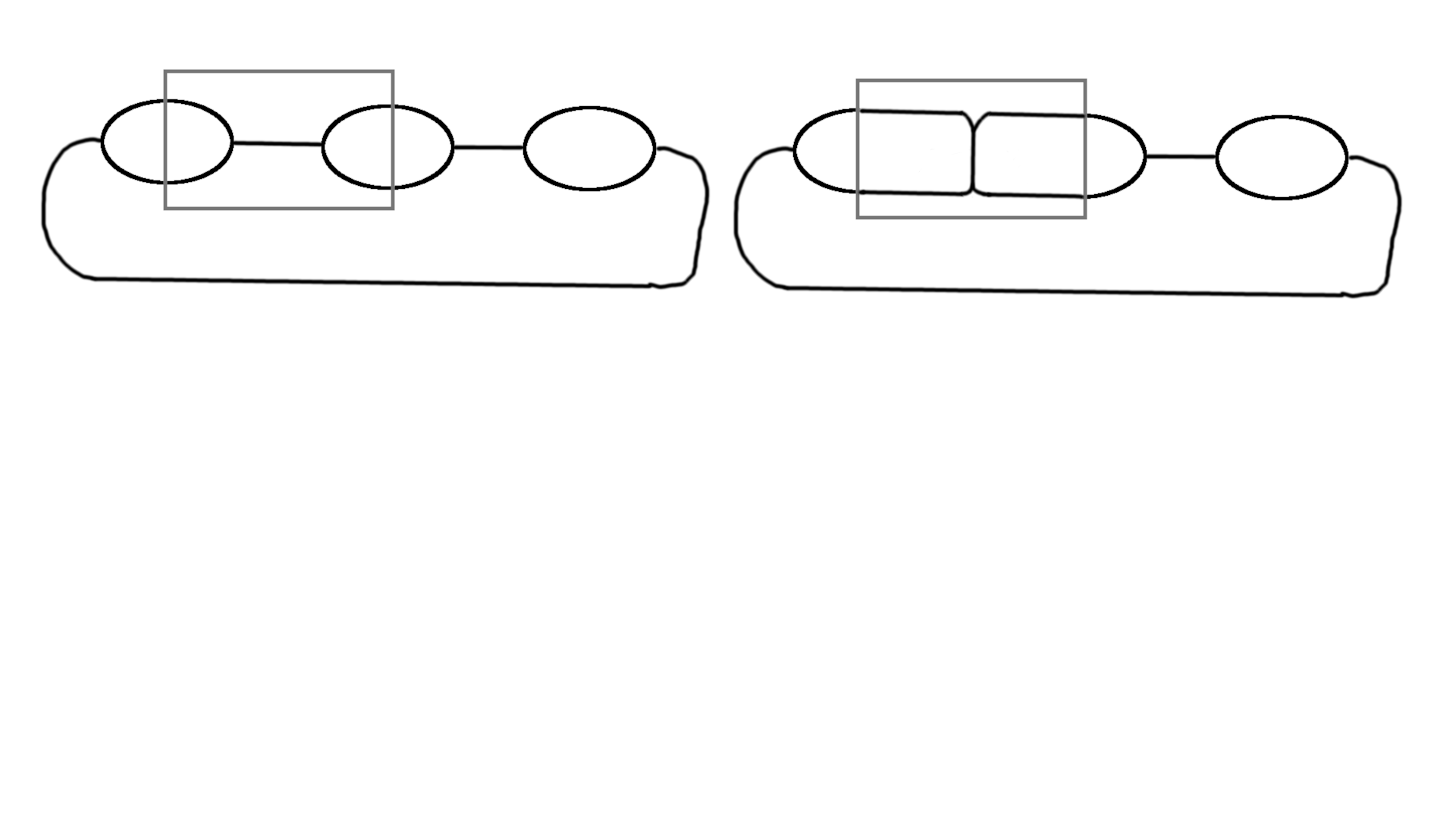}}
    
    \end{picture}
    \caption{On the left the ``necklace'' graph of genus $2$ and on the right a graph obtained by flipping the edge in the grey box.}  
    \label{Fig:Edge-Flip}
\end{figure}
 
  In \cite{Scharitzer-Shende}, we proved the geometric conjectures of \cite{SSZ}\footnote{Schrader, Shen, and Zaslow formulated their conjecture
ambiguously in terms of an unspecified higher genus open Gromov-Witten theory \cite{SSZ}.  It 
is not entirely accurate to call the skein-valued curve counting theory of \cite{SOB} a `Gromov-Witten
theory'; in particular, the closed-string count in the \cite{SOB} sense is related to the disconnected Gromov-Witten count by the change of variable $z = q^{1/2} - q^{-1/2} = e^{u/2} - e^{-u/2}$, where $u$ is the Gromov-Witten
genus counting parameter \cite{bare}.  On the other hand, no other all genus open curve counting theory is available in the literature.}, by establishing equation \eqref{recursion} in the necessary context, determining $\mathbf{A}(\Lambda_\Gamma)$, and checking $\mathbf{A}(\Lambda_\Gamma)|_{Lk}=\mathcal{A}^{SSZ}_q(\Lambda_\Gamma)$.  (Said differently, one could a-historically read \cite{SSZ} as showing that the geometric elements $\mathbf{A}(\Lambda_\Gamma)|_{Lk}$ are related by $q$-cluster transformation.)

In this article, we show that the cubic planar graph Legendrians and their disk surgeries satisfy the hypotheses needed to apply equation \eqref{main formula}, and hence that the $\mathbf{A}(\Lambda_\Gamma)$ are related by skein-valued cluster transformation.  The precise statement is given below as Theorem \ref{cubic planar graph cluster}.  We also prove correspondingly that the operator equation \eqref{recursion} determines the curve count on a filling associated to an admissible sequence of mutations (Corollary \ref{admissible sequences determined}), and in fact for a somewhat larger class (see Theorem \ref{uniqueness}, Lemma \ref{admissible implies composable}, Remark \ref{composable more general}).  

As a final illustration,  we study in Section \ref{sec: 5 terms} one example for $g=2$, in which appears a skein-valued incarnation of the pentagon relation for the quantum dilogarithm.

\begin{remark}
    Let us mention the article \cite{HSZ}, contemporary with the present work, which shows by purely algebraic methods that the  formulas for the  $\mathbf{A}(\Lambda_\Gamma)$ given in  \cite{Scharitzer-Shende} satisfy the skein valued cluster transformation \eqref{main formula}.  By contrast, here we show that the $\mathbf{A}(\Lambda_\Gamma)$ satisfy \eqref{main formula} without reference to the formulas for the  $\mathbf{A}(\Lambda_\Gamma)$.  
    Logically speaking, for results about $\mathbf{A}(\Lambda_\Gamma)$, any two of \cite{Scharitzer-Shende}, \cite{HSZ}, and the present work, determine the third. 
    However, only in the present article can be found a geometric explanation of why the q- or skein-valued cluster transformations should appear. 
\end{remark}

\vspace{2mm}

{\bf Acknowledgements.}  We thank Tobias Ekholm, Anton Mellit, Georgios Dimitroglou Rizell, Peter Samuelson, Gus Schrader, Eric Zaslow and Peng Zhou for helpful discussions. We  thank Lukas Nakamura for pointing out a mistake in Lemma \ref{skein meridians cut} in an earlier version of this article. The work presented in this article is supported by Novo Nordisk Foundation grant NNF20OC0066298, Villum Fonden Villum Investigator grant 37814, and Danish National Research Foundation grant DNRF157. 

\section{The skein lift of the q-dilogarithm}  \label{skein dilogarithm}

The q-cluster transformations of \cite{Fock-Goncharov-quantum} are expressed in terms of conjugation by the exponential of the q-dilogarithm.  Here we recall this series and its
properties, explain their interpretation  in the linking skein, and propose a  lift to the HOMFLYPT skein.

We use the q-numbers $[n]_q = 1-q^{-n}$
and the corresponding q-derivative $\partial^{(q)}_x$ acting by 
$$\partial^{(q)}_x (x^n) = [n]_q \cdot x^{n-1}$$
We consider the power series
solution to $\partial^{(q)}_x \mathcal{E}_q(x) = \mathcal{E}_q(x)$ with leading term $1$; explicitly 
$$\mathcal{E}_q(x) = \sum_{k=0}^{\infty} \frac{x^n}{[n]_q!}$$ 
where the q-factorial is as usual $[n]_q! = [n]_q [n-1]_q \cdots [1]_q$. 
We write $\Q_q$ for $\Z[q]$ with all $[n]_q$ inverted; then $\mathcal{E}_q(x) \in \Q_q[[x]]$.

We introduce
\begin{eqnarray*} 
\ell : \Q_q[[x]] &  \to &  \Q_q[[x]] \\
f(x) & \mapsto & xf(x) 
\end{eqnarray*} 
and
\begin{eqnarray*} 
m: \Q_q[[x]] &  \to &  \Q_q[[x]] \\
f(x) & \mapsto & f(qx)
\end{eqnarray*} 
Note $\ell m = q m \ell$; we also have $x \partial^{(q)}_x = 1 - m$.  From the second identity 
we rewrite the defining  equation: 
\begin{equation}  \label{q-exponential equation} 
(1 - m - \ell) \cdot \mathcal{E}_q(x) = 0
\end{equation}

There is a celebrated `pentagon identity': 
If $x, y$ are $q$-commuting variables, $xy = q yx$, then 
\cite{Schutzenberger-interpretation, Fadeev-Volkov-virasoro, Fadeev-Kashaev-dilogarithm}: \begin{equation}
\mathcal{E}_q(x)\mathcal{E}_q (y) = \mathcal{E}_q (y) \mathcal{E}_q(- yx) \mathcal{E}_q (x)
\end{equation}

In q-cluster algebra, the cluster transformation associated to a variable $x$ is given by conjugation by $\mathcal{E}_q(x)$ \cite{Fock-Goncharov-quantum}.  The pentagon identity reflects a corresponding structure of the $A_2$ cluster algebra 
\cite{Goncharov-pentagon}. 

\vspace{2mm}

The operators $\ell, m$ have a natural interpretation in terms of the linking skein.  Let $\mathbf{T}$ be a solid torus, 
and $\mathbb{T}$ its boundary.  
Choose a framed, oriented longitude $\ell$ of  $\mathbb{T}$.  We choose the framing and orientation of the 
meridian $m$ such that  $\ell m = q m \ell \in Lk(\mathbb{T})$ and
$m \to 1$ under the natural map $Lk(\mathbb{T}) \to Lk(\mathbf{T})$.  We write $x \in \mathbb{T}$ for the image of $L$.  Then
(after appropriate extension of scalars) the action of these $\ell, m \in Lk(\mathbb{T})$ on $Lk(\mathbf{T})$ is identified with the action
on $\Q_q[[x]]$ described above.  This identifies $\mathcal{E}_q(x)$ with the element of an (appropriately completed) $Lk(\mathbf{T})$ 
solving equation \eqref{q-exponential equation}, where now $\ell$ and $m$ mean `longitude' and `meridian'.

We lift to the HOMFLYPT skein.  Fix an orientation of the longitude of $\mathbf{T}$.
We have a splitting 
$Sk(\mathbf{T}) = Sk^+(\mathbf{T}) \otimes Sk^-(\mathbf{T})$ into positively and negatively 
winding links.   
We recall: 

\begin{lemma} \label{unique solution for clifford} \cite{unknot}
    For any scalar $\gamma$, there is a unique element 
    $$\mathbf{E}^\gamma = 1 + \cdots\in \widehat{Sk}^+(\mathbf{T})$$ 
    which solves the equation
\begin{equation} \label{skein dilogarithm equation}
(\bigcirc - m - \gamma \ell) \cdot \mathbf{E}^\gamma = 0
\end{equation}  
Here, $\bigcirc$ denotes the unknot, which
in the skein is $\bigcirc = (a-a^{-1})/z$.
\end{lemma}

The proof of the lemma is not difficult: 
the homology degree gives an $\mathbb{N}$-grading on $Sk^+(\mathbf{T})$, 
and there is a basis of $Sk^+(\mathbf{T})$ in which
$(\bigcirc - m)$ is diagonal, and $\ell$ is upper triangular \cite{Aiston-Morton}.  Thus there is a solution in the completion $\widehat{Sk}^+(\mathbf{T})$. 
In \cite{unknot}, the explicit
formula for $\mathbf{E}$ in this basis is derived.  We do not reproduce this formula here, as we  will never need to use it.

\begin{corollary}
$\mathbf{E}|_{Lk(\mathbf{t})} = \mathcal{E}_q$.     
\end{corollary}
\begin{proof}
As $\bigcirc = 1 \in Lk$, the defining relation of
$\mathbf{E}$ specializes to that of $\mathcal{E}_q$, and
both relations have a unique solution.
\end{proof}

The definition of $\mathbf{E}$ depended on a choice of framed, oriented longitude of $\mathbb{T}$.  
Correspondigly, if $M$ is an oriented 3-manifold and $k \subset M$ is a framed, oriented knot, then we define 
$\mathbf{E}(k)$ by cutting out a neighborhood of $k$ and gluing in $\mathbf{E}$; similarly $\mathbf{E}^\gamma(k)$.

\section{A disk surgery cobordism} \label{sec: disk surgery cobordism} 

In Figure \ref{Fig:Disk-surgery}, we describe Legendrian disk surgery in a Darboux chart.  In this section we provide an exact Lagrangian cobordism whose negative end is the union of the pre-surgery Legendrian and a Clifford torus, and whose positive end is the post-surgery Legendrian.

\begin{figure}
    \begin{picture}(0,180)
    \put(-140,0){\includegraphics[width=12cm, trim=0cm 1cm 1.5cm 0cm, clip]{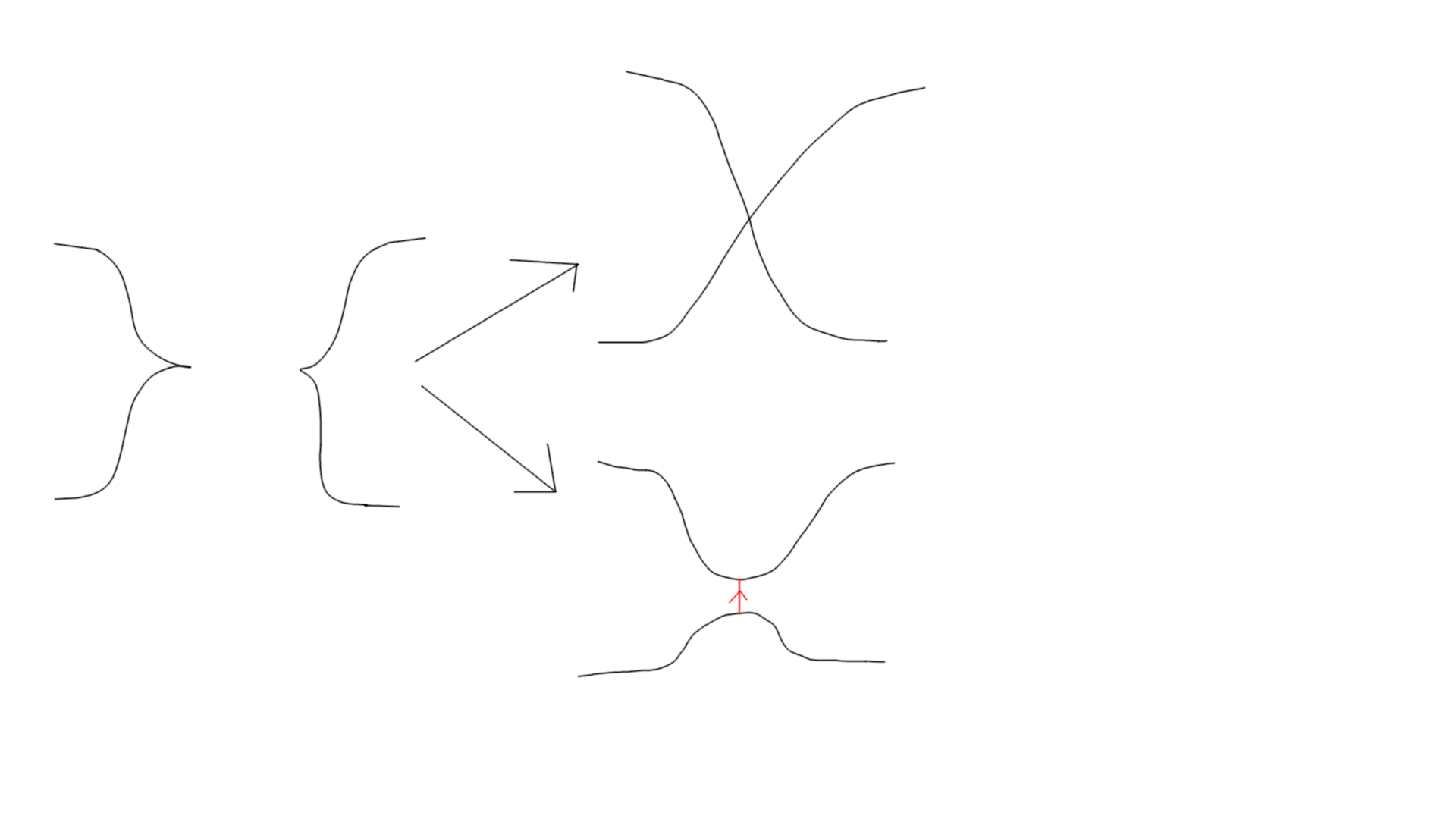}}
    \end{picture}
    \caption{Legendrian disk surgery is the spinning up of the upper front transition.  (The front of the disk is the interval connecting the two cusps on the LHS.) We contrast this with Legendrian ambient surgery of \cite{Rizell-Surgery}, which is the spinning up of the lower front transition. Indicated in red is an index-$0$ Reeb chord.} 
    \label{Fig:Disk-surgery}
\end{figure}

In a symplectization $V \times \R$, we consider only Lagrangian cobordisms which are eventually cylindrical on both ends.  By the {\em length} of the cobordism, we mean the length of the complement in $\R$ of the cylindrical region.  

As building blocks, we will use the 
`Legendrian ambient surgery cobordisms' of Dimitroglou-Rizell \cite[Definition 4.6]{Rizell-Surgery}.  These also depend on the choice of an isotropic disk, but let us be clear that the ambient surgery in the critical case is not the same as what we call Legendrian disk surgery here.  The difference is illustrated in  Figure \ref{Fig:Disk-surgery}.

\begin{figure}
    \begin{picture}(0,180)
    \put(-150,0){\includegraphics[width=12cm, trim=0cm 0cm 0cm 0cm, clip]{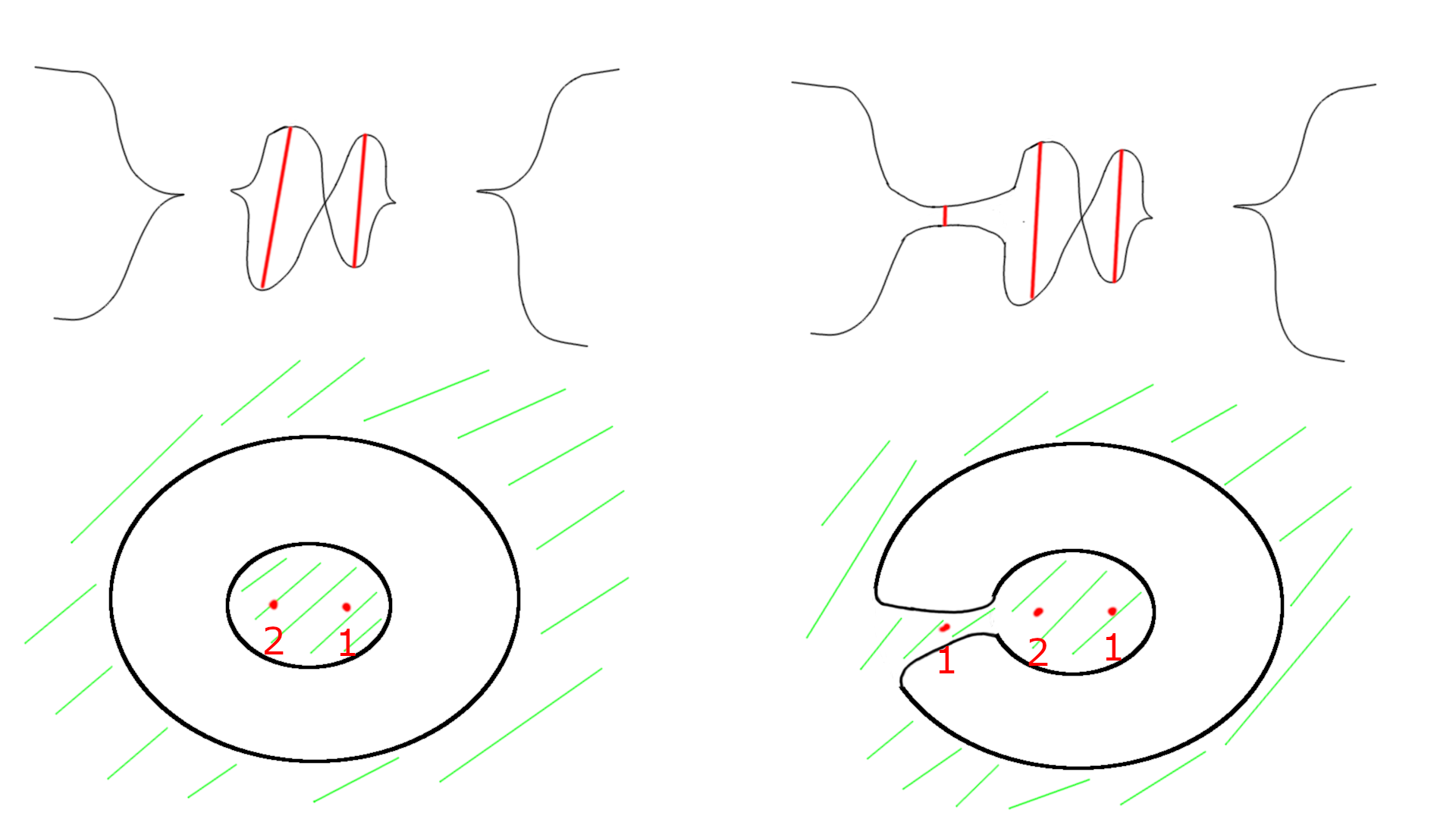}}
    
    \end{picture}
    \caption{On the left $\Lambda \sqcup \T$ and on the right $\Lambda \#_{p=q} \T$ (joining them by a Legendrian $1$-handle). 
    The upper pictures are a slice of the front 
    projections.  The lower pictures are the view from above (in dimension $n=2$): 
    the green-shaded area is the double-covered area and the non-shaded area has no Legendrians above them, and the black lines are the cusp edges. 
    %So it needs a maximum and a minimum. 
%     On the left, we have the cusp-edges $k$ and $k'$. Afterwards we have them joined by a $1$-handle which creates a new connected cusp-edge. 
    Indicated in red are the Reeb chords together with their respective indices.
    Note that we have perturbed $\T$ away from being rotation-symmetric, to make it front generic. 
    } 
    \label{Fig:Surgery-1}
\end{figure}

\begin{construction} \label{Con:Basic cobordism}
    Let $Y$ be a contact manifold, $\Lambda \subset Y$ a Legendrian of dimension $n$, and $D$ a Legendrian disk ending on $\Lambda$.  Assume given a disk surgery setup as in Figure \ref{Fig:Disk-surgery}. 

    Then for any $\epsilon > 0$, there is a choice of  $\T=S^1 \times S^{n-1}$ as in 
    Figure \ref{Fig:Surgery-1} and a Lagrangian cobordism $\tilde{L}$  of length $ < \epsilon$ with negative end $\Lambda \sqcup \T$ and positive end $\Lambda_D$.  
    When $n = 2$, then $\T$ is the Clifford torus. 

    The cobordism is trivial away from a neighborhood of $\R \times D$, and is obtained from a composition of  
    \cite[Definition 4.6]{Rizell-Surgery} ambient surgeries of index $0$ and $n-1$, and a ``Legendrian isotopy cobordism".
\end{construction}
\begin{proof}
    The cobordism is the concatenation of the following three pieces, described from the negative end to the positive end. 

    \begin{enumerate}
        \item     We start by joining $\Lambda$ and $\T$ with ambient index-$0$ surgery.  The result will be 
        an exact Lagrangian cobordism with negative end $\Lambda \sqcup \T$ and positive end $\Lambda \#_{p=q} \mathbb{T}$ which is cylindrical outside the Darboux chart. Topologically this cobordism is just the trivial cylinder above $\Lambda \sqcup \T$ with a 1-handle attached.
        
        To do so, we pick a point $p$ along the cusp-edge of $\Lambda$ and a point $q$ along the cusp-edge of $\T$ and join it along a $1$-disc inside $D$. The resulting Legendrian $\Lambda \#_{p=q} \mathbb{T}$ is depicted on the right of Figure \ref{Fig:Surgery-1}. 
        As computed by Rizell \cite[Subsection 4.2.3]{Rizell-Surgery}, the index $\le 1$ 
        Reeb chords at the positive end are the same as those at the negative end, together, when $n=2$, with a single index-$(n-1)$ Reeb chord which sits at the minimum of the handle we just attached. (The restriction to index $\le 1$ here is because we only demanded that index $\le 1$ Reeb chords avoid the surgery disk.)

        \item    The next piece of the cobordism is an index-$(n-1)$ Legendrian ambient surgery. This time the ambient surgery disk has boundary the newly connected cusp edge $\Lambda \#_{p=q} \mathbb{T}$ (the disk projects to the bounded white region in Figure \ref{Fig:Surgery-1} bottom-right).  The positive end of the cobordism is the left picture of Figure \ref{Fig:Surgery-2}. There is a new Reeb chord of index-$0$, see \cite[Subsection 4.2.3.]{Rizell-Surgery}. The cobordism is topologically  an $n$-handle attachment to the trivial cylinder  $(\Lambda \#_{p=q} \mathbb{T} \times \R)$.

        \item    Finally, a standard construction  shows that inside a Darboux chart, a Legendrian isotopy (relative to the boundary of $B$) gives rise to an exact Lagrangian cobordism, see for instance \cite[Section 6.1]{Ekholm-Honda-Kalman}.  In the Darboux chart $B$ the Legendrian $\Lambda'_D$ is symmetric w.r.t. to the surgery locus $D$. As such (away from the front cone singularity) we can write $\Lambda'_D$ as the image of $j^1(f)$ and $j^1(-f)$. We take the Legendrian isotopy which is obtained by interpolating between $j^1(f)$ and $j^1(g)$ where $g$ is a linear function of non-zero incline except at the boundary where we round the corners so we obtain a smooth Legendrian. This is depicted on the right of Figure \ref{Fig:Surgery-2}.

    To obtain a Lagrangian cobordism between $\Lambda'_D$ and the end of this Legendrian isotopy $\Lambda_D$ we have to take a deformation of its trace. So let $\Lambda_D(s,t)$ be the Legendrian isotopy where $s$ is the coordinate on $\Lambda_D$ and $t$ the coordinate on $[0,1]$. Furthermore, let $h:\R \rightarrow [0,1]$ be a smooth non-decreasing function. Then we define an exact Lagrangian through the image of 
    $$(s,t) \mapsto (\Pi_\C(\Lambda_D(s,h(t)),t,z(\Lambda_D(s,h(t)) + \alpha(\frac{d}{dt}(\Lambda_D(s,t))) $$
    which is  a deformation of the $h$-scaled trace of the isotopy in the Reeb direction. In particular, $h$ controls the size of the Lagrangian deformation by its support. 
    
    Note, however that if the derivatives of $h$ are too large then this Lagrangian is immersed as double points may appear.  Thus, for fixed initial and final Legendrian, the isotopy cobordism has length bounded below. 
\end{enumerate}

    Now that we have described the basic pieces of the cobordism we need to argue that this composition can be made arbitrarily small. 
    
    The length of the surgery pieces can be made arbitrarily small.  Indeed, in \cite[Section 4.2.2.]{Rizell-Surgery} there is a  quantity $\delta>0$ which controls both the length of the surgery cobordism (by $[1-\delta^{1/3},1+\delta^{1/3}]$) and the length of the new Reeb chords by $2 \delta^{3/2} (\frac{1}{2}))^{3/2}$. 
    %where the factor $2$ is given by symmetry around the $0$-section, for more details see the original construction 
    
    It remains to consider the isotopy cobordism. As explained above, to make this cobordism arbitrarily short, we must make the ratio between the Reeb chords of $\T$ and the Reeb chords introduced by the Legendrian surgery arbitrarily close to $1$. 
    
    Rizell's construction depends on a function $\rho_\delta: \R \to \R$, 
    which roughly will control how close to the cusp edges one begins the modification, and how flat the front projection of the positive end Legendrian will be.  

    We illustrate how to modify his construction as we require in the case when the dimension of the Legendrian is $1$ and we are performing an index $0$ surgery; the general case is similar. 

    In \cite{Rizell-Surgery}, the coordinates on the Darboux chart around the surgery disk are chosen such that the Legendrian $\Lambda$ has a cusp edge on $\{-1,1\}$ and is thus locally given by both branches of $\pm (x^2-1)^{3/2}$ where the function inside is non-negative.  We require this description holds on a chart containing some fixed neighborhood of $[-1, 1]$, let us say $[-1-\delta(\frac{1}{2} + \delta'), 1+\delta(\frac{1}{2} + \delta')]$ where $\delta'$ is some small number (in \cite{Rizell-Surgery},  $\delta'=\frac{1}{6}$). The positive end of the cobordism $\Lambda'$ is given by $\pm (x^2+\rho(x^2)-1)^{3/2}$ where $\rho$ is a non-negative function fulfilling:

    \begin{enumerate}
        \item $\rho(0)=(1+ \frac{1}{2}\delta )$ and $\rho'(0)=0$
        \item $\rho(x^2)=0$ if $x^2 \geq 1+\delta(\frac{1}{2} + \delta')$
        \item %for all $1\geq\lambda\geq 0$: 
        $x^2 +  \rho(x^2)$ is strictly convex.
    \end{enumerate}

Given such a function $\rho$, the 
Reeb chords for the lift of $\pm (x^2+\lambda\rho(x^2)-1)^{3/2}$ are only possible above $0$: from the $\pm$ we see the front is symmetric about the zero section, so Reeb chords can only appear if the first derivative of the function vanishes. By the third condition, for $\lambda$ close enough to $1$, this happens exactly once, namely when $x=0$. The Reeb chord length is dictated by $\rho(0)$.

We will now argue that such $\rho$ exists for any choice of $\delta' > 0$.  The essential constraint imposed by the above conditions on $\rho$ is that 
$x^2 + \rho(x^2)$ at $1+ \frac{1}{2}\delta $ 
must be greater than $\rho(0)^2$.  This means 
that $\rho((1+\frac{\delta}{2})^2) > 0$. However, its value at this point may be arbitrarily small and its support may be any interval strictly including $[-1-\frac{\delta}{2},1+\frac{\delta}{2}]$. In particular, $\delta'$ can be chosen to be any positive number. 

Thus by choosing $\delta_0$ and $\delta_{n-1}$ small enough we obtain a bound on the length of the surgery cobordisms and by adjusting the corresponding functions $\rho_0, \rho_{n-1}$ we can make our initial choice of $\T$ such that the pairs of Reeb chords have length which has ratio arbitrarily close to $1$ thus allowing us to make the Legendrian isotopy arbitrarily short.
\end{proof}

\begin{figure}
    \begin{picture}(100,185)
    \put(-150,0){\includegraphics[width=12cm, trim=0cm 0cm 0cm 0cm, clip]{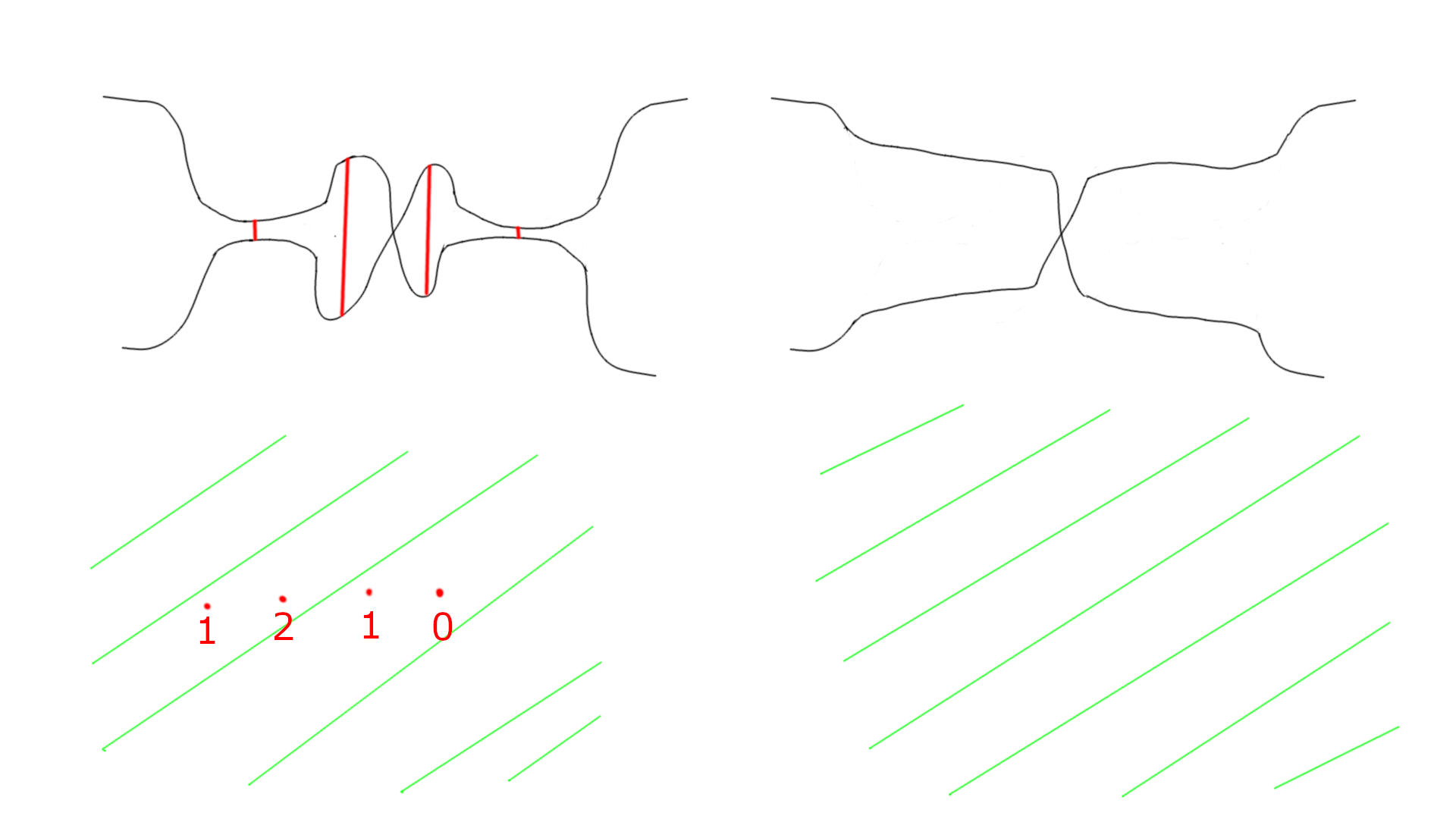}}
    
    \end{picture}
    \caption{On the left the Legendrian $\Lambda$ after index $0$ and index $1$ surgery. In the picture one sees that now there are two annuli $A_1$ and $A_2$ with small height difference join $T_{Cl}$ and $\Lambda$. To obtain the final position, depicted on the right, one increases the height difference of the annuli until both Reeb chord pairs vanish.}  
    \label{Fig:Surgery-2}
\end{figure}

We characterize how the cobordism relates, topologically speaking, to a trivial cylinder. 

\begin{lemma} \label{Topological filling possible}
    Let $m$ be any embedding of $S^{n-1} \hookrightarrow \T$ which intersects
    the cusp edge of the front projection transversely and exactly once.  
    Let $\mathbf{T}$ be a filling of $\T$ with topology $B^2\times S^{n-1}$ such that $m$ becomes contractible.  Then $\tilde{L} \cup_{\T_{Cl}} \mathbf{T}$ is topologically $\Lambda \times \R$.
\end{lemma}

\begin{figure}
    \begin{picture}(100,150)
    \put(-50,0){\includegraphics[width=8cm, trim=0cm 5cm 10cm 0cm, clip]{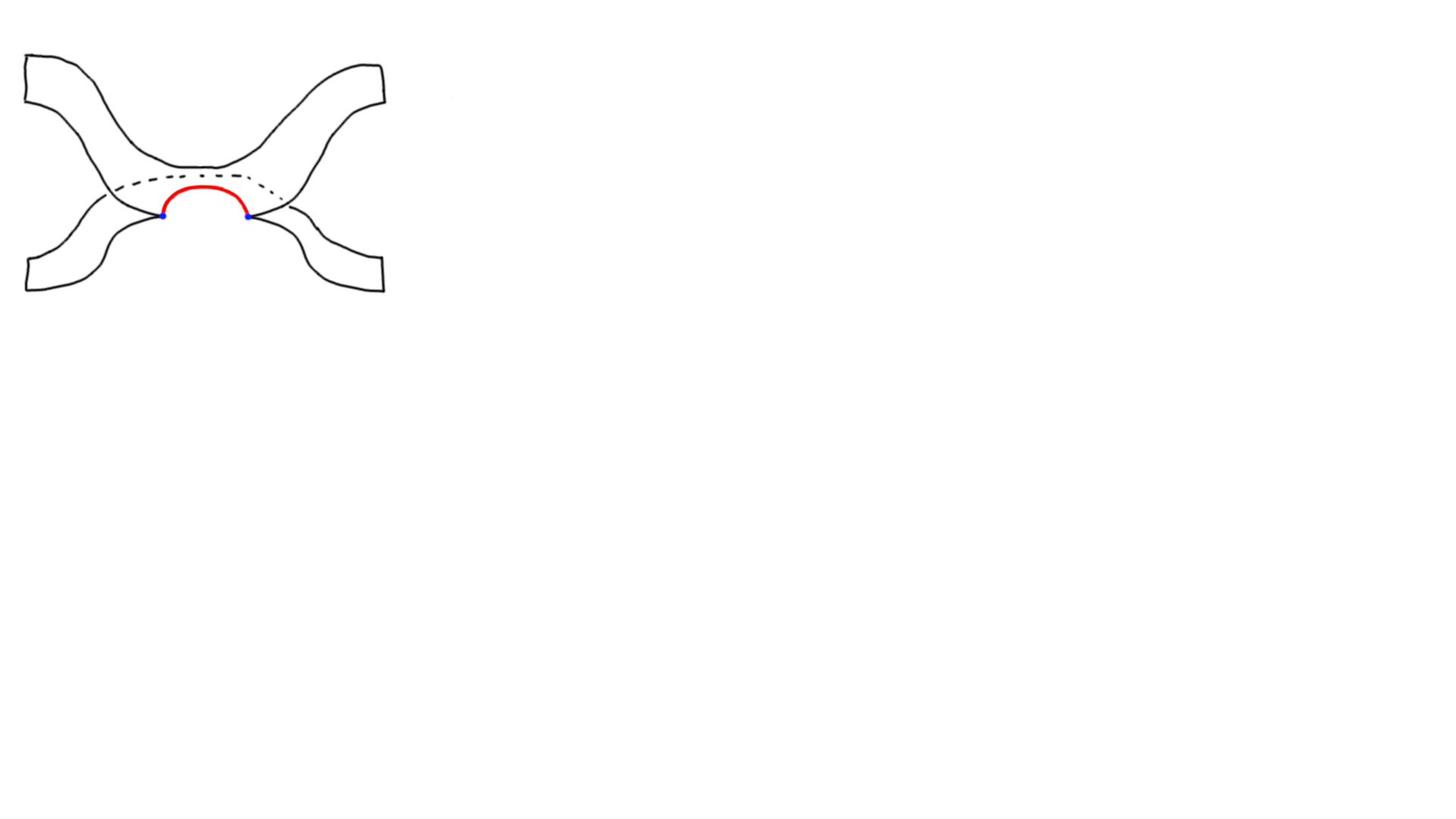}}
    
    \end{picture}
    \caption{A small piece of the index $0$ ambient surgery cobordism for $n=1$. Indicated in red is the core and in blue the attaching sphere of the corresponding $1$-handle. Copied from \cite[Figure 2]{Rizell-Surgery}.}  
    \label{Fig:Attaching-sphere}
\end{figure}

\begin{proof}
    We will describe Morse functions and cancel handles. First, we will choose an appropriate Morse function on $\tilde{L}$. Consider the function $f$ induced by the projection onto the symplectisation variable $t$. As was shown by Rizell \cite[Remark 4.7.5]{Rizell-Surgery} this function when restricted to stages $1$ and $2$ has a singularity only at the base point of the handle attachments which are automatically Morse whose index is $k+1$ where $k$ is the index of the Legendrian surgery. So it only remains to note that stage $3$ of the cobordism which comes from a Lagrangian of the the form:
    \begin{align*}
        (\Pi_\C(\Lambda(s,h(t))),t,z(\Lambda(s,h(t))) + \alpha(\frac{d}{dt}\Lambda(s,h(t))))
    \end{align*}
    has non-degenerate projection onto the $3$rd slot and thus $-f$ has no singularities. 

    We illustrate in Figure \ref{Fig:Attaching-sphere}  the core disc of a handle attachment for an index $0$ ambient surgery on a Legendrian knot. 
    In any dimension, the front of the core disk is a cusp edge. 
    % \sayVS{these next sentences are confusing} We note that the core disc of the index $0$ surgery can be viewed as a copy of $\R$ whose both ends limit towards $-\infty$ with one end in $\T$. Similarly, we can define the core disc of the index-$(n-1)$ surgery which is an $n$ handle attachment. The problem is that in dimensions $>1$ it will eventually hit the $1$-handle attachment. However, we can consider a ``continuation of the attaching sphere" of the index $n$ handle which has two components which are the union of the two cusp-edge spheres.

    On a filling $\mathbf{T}_m$ of $\T$ with topology of $B^2 \times S^{n-1}$, there's a Morse function which has exactly one Morse singularity of index $(n-1)$ and exactly one Morse singularity of index $0$ and no others, and whose gradient points strictly outward at the boundary. The belt sphere of the $(n-1)$-handle is homotopic to $m$.
    
    Now, we wish to glue in the filling such that the glued Morse functions have cancellable singularities. Recall, that this is possible if the belt sphere of an index $i$ and the attaching sphere the index $i+1$ have one transverse intersection. For dimension reasons, this is automatic in the case of the index $0$ singularity of $\mathbf{T}_m$ and the index $1$ singularity of $\tilde{L}$. For the index $(n-1)$ and index $n$ singularities, this was guaranteed by hypothesis. 
    % for the continued attaching sphere and the belt sphere of $m$ as it was chosen to have exactly one intersection with the cusp-edge of $\T$.
\end{proof}

\section{Skein valued cluster transformation from surgery cobordism} \label{what perfect cobordisms do}

In this section, we identify properties of the holomorphic curve theory of a disk surgery cobordism which will ensure that the skein-valued augmentation ideals of the Legendrians at the ends are related by skein-valued cluster transformation. 

First let us recall some details of the setup of the skein-valued curve counting. 
The general setting is that of a pair $(X, L)$, where $X$ is a symplectic Calabi-Yau of real dimension 6, and $L$ is a Maslov zero Lagrangian.  In the noncompact case, we demand that $X$ has convex end given by a contact manifold  
In addition, we require $L$ to carry a vector field a 4-chain $K \subset X$ with $\partial K = 2L$.  In particular, $K$ is a cycle in 
$H_4(X, \Z/2\Z)$, which we will use as a background class for defining signs; correspondingly $L$ should carry a spin structure twisted by $K$, as in \cite{Scharitzer-Shende}.  (Using $K$ as a background class is not necessary to set up the theory, but gives in the present context formulas which are more symmetric and match those of \cite{SSZ}.) 

Following \cite{SOB, ghost, bare}, we count holomorphic curves in $X$ with boundary on $L$ by the class of their boundary in $Sk(L, K^\circ \cap L) \otimes \Z[[H_2(X, L)]]$, where the power series go in the direction of positive symplectic area.  Here, $K^\circ$ is the interior of $K$, and the decoration $K^\circ \cap L$ indicates that we use a slightly twisted version of the skein, where crossing this 1-manifold multiplies skein elements by $(-a)^{\pm 1}$, 
as explained in \cite[Sec. 2, Sec. 6]{Scharitzer-Shende}. Henceforth we omit the decoration $K \cap L$; it should always be understood. 

In all examples of interest in the present article, the symplectic form is exact, $\omega = d\lambda$.  Thus the symplectic area of a holomorphic curve can be read off the $H_1(L)$ class of its boundary, hence already off the skein itself, and we may work instead in the completion of the skein along $\lambda$-positive classes, which we denote $\widehat{Sk}_\lambda(L)$.  We use the same notation for the corresponding completion of the skein of any 3-manifold equipped with a 1-form $\lambda$. 

We will allow noncompact $X$ which at infinity have ends which are cylindrical on contact manifolds; throughout we work in a setting where the concave ends have no Reeb orbits of index $\le 0$, so we do not have to worry about escape of curves. 
We correspondingly allow noncompact Lagrangians $L$ which are asymptotic to a Legendrian $\partial_\infty L \subset \partial_\infty X$.  In such cases we require the various brane data to be cylindrical at infinity: in a trivialization near infinity
$L \sim \partial_\infty L \times \R$, 
the vector field $v$ should point along the $\R$ factor, and the $4$-chain should be $\R$-invariant.  In particular, $\partial_\infty K$ gives a 3-chain with $\partial(\partial_\infty K) = 2 \partial_\infty L$. 

Let us consider the situation where $X$ is a symplectization $\R \times V$ and $\Lambda \subset V$ is a Legendrian.  As always, moduli spaces of curves
ending on $\R \times \Lambda $ have an $\R$-action; 
we take the quotient by it and study the zero-dimensional moduli spaces.  
The boundaries of such curves live most naturally in a skein of $\Lambda \times \R$ on tangles which, near $\pm\infty$ in $\R$, are straight lines going to the endpoints of the Reeb chords, with the appropriate orientation.  We write 
$\mathbf{A}(\Lambda)$ for the resulting skein element. 
The situation for counting curves in cobordisms is analogous, save that now there is no $\R$-action.

Recall that we say $(V, \Lambda)$ is Reeb-positive when $V$ has no index $\le 1$ Reeb orbits, and $\Lambda$ has no index $\le 1$ Reeb chords.  In this case we write $R_1(\Lambda)$ for the set of index-1 Reeb chords. 
In the Reeb positive case, the only rigid (up to translation) curves for $(\R \times V, \R \times \Lambda)$ are curves with one positive puncture at some chord in $R_1(\Lambda)$.  

Suppose now $(X, L)$ is a filling of $(V, \Lambda)$.  Under the Reeb-positivity hypothesis.  Then: 

\begin{lemma} \label{wave function} (\cite{unknot}, \cite[Lemma 3.4]{Scharitzer-Shende})
$\mathbf{A}(\Lambda) \Psi(L) = 0$.      
\end{lemma} 

We recall the consequence for exact fillings: 

\begin{corollary}
     \label{necklace annihilates}
    Let $L$ be an exact filling for $\Lambda$.  
    Fix an index one Reeb chord $\rho$ of $\Lambda$. 
    Under the natural map 
    $Sk(\Lambda) \to Sk(L_g)$, the element $A(\rho)$  is sent to zero.  
\end{corollary}
\begin{proof}
    An exact filling bounds no compact curves, so the curve count in the filling is $\Psi=1$.  On the other hand, 
    $A(\rho) \Psi = 0$. 
\end{proof}

\begin{remark}
    When writing the equation $\mathbf{A}(\Lambda) \Psi(L) = 0$, it is relatively harmless to choose capping paths for $\mathbf{A}(\Lambda)$, since they effectively appear on the {\em left} of the expression
$\mathbf{A}(\Lambda) \Psi(L)$, and so do not interfere with the equation.  We often do this without further
comment, and regard $\mathbf{A}(\Lambda)$ as lying in the usual the usual skein of 
$\R \times \Lambda$, where curves may not go to infinity. Indeed, we did this already in writing the expression for $\mathbf{A}(\T_{Cl})$ in Equation \eqref{clifford torus operator}.

However, when we want to compose cobordisms, it is not appropriate to choose capping paths, hence below we leave our elements in skeins of tangles.
\end{remark}

The corresponding statement for cobordisms is the following: 

\begin{lemma} \label{Lem:General-breaking} 
Assume $(X, L)$ is a Reeb-positive exact Liouville cobordism, equipped with appropriate brane data.  
Given $\rho \in R_1(\partial_+ L)$ and $\nu \in R_1 (\partial_- L)$,
denote by $\Psi_{X, L}(\rho, \nu)$ the skein-valued count of curves with one positive puncture asymptotic to $\rho$ and one negative
puncture asymptotic to $\nu$.  Then

$$\mathbf{A}(\partial_+ L, {\rho}) = \sum_{\nu \in  R_1 (\partial_- L)} \Psi_{W, L}(\rho, \nu)  \cdot  \mathbf{A}(\partial_- L, {\nu})$$
\end{lemma}
\begin{proof}
    Same as the proof of Lemma \ref{wave function} \cite{unknot}, \cite[Lemma 3.4]{Scharitzer-Shende}: 
    study the moduli of curves in the cobordism with one positive puncture at $\rho$.  This is a one dimensional moduli space, whose boundary is on the one hand numerically zero, and on the other hand consists of the SFT breakings -- giving the terms in the equation -- and boundary breaking, which is zero modulo the skein relations. 
\end{proof}

We turn to the disk surgery. 
Consider a 
contact manifold $V$, a Legendrian $
\Lambda \subset V$, and a Legendrian surgery disk $D$ with $\partial D \subset \Lambda$.  We fix a Darboux chart is locally as in the chart as in Figures \ref{Fig:Disk-surgery} and 
\ref{Fig:Surgery-1}
We write $\Lambda_D$ for the result of disk surgery. 
As depicted in Figure \ref{Fig:Surgery-1}, we fix a Clifford torus $\T_{Cl}$ 
whose front may be drawn alongside 
$\Lambda$ in the surgery region.

We abstract the properties we require of the cobordism constructed in Section  \ref{sec: disk surgery cobordism}.

\begin{definition} \label{disk surgery cobordism}
A {\em disk surgery cobordism} is 
an exact cobordism $C$ with negative end $\Lambda \sqcup \T_D$, and positive end $\Lambda_D$.  
A {\em topological trivialization} of a disk surgery cobordism is the data of a solid torus $\mathbf{T}$, 
an identification $\partial \mathbf{T} = \T_{Cl}$, 
and a smooth trivialization 
\begin{equation} \label{trivialization} 
\R \times \Lambda = 
C \cup_{\T_{Cl}} \mathbf{T} = \R \times \Lambda_D
\end{equation}
carrying the longitude of 
$\mathbf{T}_D$ to the 
$0 \times \partial D \subset 0 \times \Lambda$.  
We also require the data of an extension of the topological brane structures on $C$ (vector field, linking lines, spin structure) into $\mathbf{T}$.  

Note there is an evident identification of $\T_{Cl}$
with the standard Clifford torus in $S^5$ (they have the same front projection).  We say that a topological trivialization is {\em standard} if this identification extends to an isomorphism of pairs of $(\mathbf{T}, \T_{Cl})$ with the topological space underlying one of the (three) Lagrangian smoothings of the Harvey-Lawson cone in $\R^6$ which fills the standard Clifford torus in $S^5$.   In this case the symplectic primitive on $\R^6$ determines a positive cone in $H^1(\mathbf{T})$, and we orient $\partial D$ positively in this sense.   
\end{definition}

\begin{remark}
    We fill the cobordisms only topologically, rather than geometrically (i.e. actually gluing in the smoothed Harvey-Lawson Lagrangian), because we want to avoid discussing the composition of non-exact cobordisms.  
\end{remark}

\begin{figure}
\includegraphics{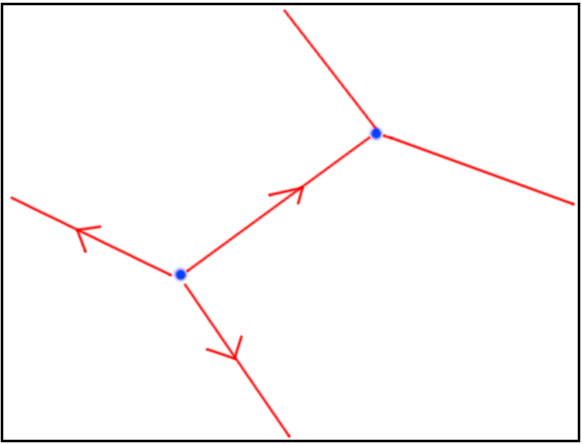}
    \caption{Paths from the boundaries of the curves along the Clifford torus}  
    \label{Fig:Clifford paths}
\end{figure}

\begin{remark}
    Let us explain one way to characterize which trivialization fillings are standard.  Consider the disks in the symplectization of $(S^5, \T_{Cl})$ with one positive puncture at the (unique in our presentation) index-1 Reeb chord.  Their boundaries trace out paths along $\R \times \T_{Cl}$.  These were determined (with varying degrees of explicitness) in e.g. \cite{Rizell-Surgery, unknot, Scharitzer-Shende}; we depict them in Figure \ref{Fig:Clifford paths}. 

    The characteristic topological property of the Harvey-Lawson fillings is that two of the three paths should become isotopic in the filling.  This means contracting (in the figure) either the vertical, horizontal, or anti-diagonal circle.  One of these is in the same homotopy class as the cusp edge, and the others have intersection numbers $\pm 1$ with it.  From Lemma \ref{Topological filling possible}, we see the latter two give fillings compatible with trivialization.  One checks in terms of the rules above that these provide the two different orientations of $\partial D$.  
    
    We will give another version of this discussion in terms of a different presentation of the Clifford torus later in Lemma \ref{Lem: CPG topological filling}. 
\end{remark}

The disk surgery cobordism we created above was trivial outside the surgery region.  One might hope that curves which begin and end outside the surgery region behave as if the cobordism was in fact a trivial cobordism.  We axiomatize the desired property: 

\begin{definition} \label{def: weakly perfect}
    We say that a disk surgery cobordism is {\em weakly perfect} if: 
    \begin{enumerate}
        \item if there is a nontrivial rigid curve with positive end at Reeb chord $\rho_+ \in R_1(\Lambda_D)$ and negative end at Reeb chord $\rho_- \in R_1(\Lambda)$, then 
    $\mathrm{length}(\rho_+) \ge \mathrm{length}(\rho_-)$.  That is, we ask that these curves respect the action filtration, as they would inside a trivial cobordism.  Note we put no condition on curves with negative end at the Clifford torus.  
        \item \label{wpdef trivial strips}
        There is a unique rigid curve from a each index-1 Reeb chord at the positive end to its counterpart at the negative end, which is a strip whose boundaries become isotopic to vertical lines under the trivialization \eqref{trivialization}. 
    \end{enumerate}    
\end{definition}
 
By definition, in a weakly perfect disk surgery cobordism 
for which all chords in $R_1(\Lambda)$ and 
$R_1(\Lambda_D)$ have the same length, 
the only curves between them which may appear
are the trivial strips of \eqref{wpdef trivial strips} above.  We axiomatize this: 

\begin{definition} \label{def: perfect} 
We say that a disk surgery cobordism is 
{\em perfect} if the only rigid curves between 
elements of $R_1(\Lambda)$ and  $R_1(\Lambda_D)$
are one topologically trivial strip between 
each pair of corresponding chords. 
\end{definition}

\begin{theorem} \label{cluster transformation from perfect cobordism} 
Suppose that $\Lambda$ and $\Lambda_D$ are related by disk surgery, and there exists
a perfect disk surgery cobordism $L$ (admitting compatible brane structures), and a standard topological trivialization of $L$. 
Assume $\partial D$ is not homologically trivial. 
Fix a completion $\widehat{Sk}(L)$ along a cone containing $\partial D$. 
Then 
$$\mathbf{A}(\Lambda_D,\rho_+) \mathbf{E}(\partial D) = \gamma \mathbf{E}(\partial D)  \mathbf{A}(\Lambda,\rho_-)$$
for  $\gamma$ some signed monomial in the framing variable.  

If we weaken the hypothesis to `weakly perfect', we may still conclude the result for the shortest Reeb chords. 
\end{theorem}

\begin{proof}
Lemma \ref{Lem:General-breaking} 
plus the perfectness hypothesis
gives the following equation:
$$ \mathbf{A}(\Lambda_D,\rho_+) = \Psi_{Y\times \R,L}(\rho_+,\rho_-) \mathbf{A}(\Lambda,\rho_-) + \sum_{\nu \in R_1(\T_{Cl})} \Psi_{Y\times \R,L}(\rho_+,\nu)\mathbf{A}(\T_{D}, \nu) $$
where $\rho_+$ is a Reeb chord of $\Lambda_D$ and $\nu$ runs over the the index $1$ Reeb chords of $\T_{Cl}$.
Here,  $\Psi_{Y\times \R,L}(\rho_+,\rho_-) $ is the contribution of a trivial strip, hence some signed monomial in `$a$'.  

We now forget symplectic geometry, and consider the above equation just in the skein of the cobordism.  
We glue in $\mathbf{T}$ containing $\mathbf{E}$.  By the `standard' hypotheses on $\mathbf{T}$, the fact that $\mathbf{E}$ is the skein-valued curve count for a Harvey-Lawson brane \cite{unknot}, and Lemma \ref{wave function}, we deduce
$\mathbf{A}(\mathbb{T}_{Cl}, \nu) \mathbf{E}(\partial D) = 0$.  We are left with the stated result.

(It may be confusing why the $\mathbf{E}(\partial D)$ appears to the left, rather than right, of $\mathbf{A}(\Lambda, \rho_-)$, since $\mathbb{T}_{Cl}$ and $\Lambda$ are symplectically both negative ends of the cobordism.  The point is that the equation is written in the skein of the $\R \times \Lambda$ of Def. \ref{disk surgery cobordism}.  Here, $\Lambda_D$ is at one end, $\Lambda$ at the other, and $\mathbf{T}$ in the middle.) 
% In the weakly perfect case, one works inductively in the action filtration: Consider the action filtration $\ell(\rho_1)< \dots \ell(\rho_n)$ of the index $1$ Reeb chords of $\Lambda$ then for $\rho_1$ we have the equation:
%     $$ \mathbf{A}(\Lambda_D,\rho_{1,+}) = \Psi_{Y\times \R,L}(\rho_{1,+},\rho_-) \mathbf{A}(\Lambda,\rho_{1,-}) + \Psi_{Y\times \R,L}(\rho_{1,+},\nu)\mathbf{A}(\T_{D}, \nu) $$
%     which follows as in the perfect case and implies:
%     \begin{align*}
%         \mathbf{A}(\Lambda_D,\rho_{1,+}) \mathbf{E}(\partial D) = \gamma \mathbf{E}(\partial D)  \mathbf{A}(\Lambda,\rho_{1,-})
%     \end{align*}
%     Now, for $\rho_2$ we get a slightly different equation:
%     \begin{align*}
%         \mathbf{A}(\Lambda_D,\rho_{2,+}) = \Psi_{Y\times \R,L}(\rho_{2,+},\rho_{2,-}) \mathbf{A}(\Lambda,\rho_{2,-}) + \Psi_{Y\times \R,L}(\rho_{2,+},\rho_{1,-}) \mathbf{A}(\Lambda,\rho_{1,-}) + \Psi_{Y\times \R,L}(\rho_{1,+},\nu)\mathbf{A}(\T_{D}, \nu)
%     \end{align*}
%     Again, we insert the topological trivialisation and obtain:
%     \begin{align*}
%         \mathbf{A}(\Lambda_D,\rho_{2,+}) \mathbf{E}(\partial D) = \gamma \mathbf{E}(\partial D)  \mathbf{A}(\Lambda,\rho_{2,-}) + \Psi_{Y\times \R,L}(\rho_{2,+},\rho_{1,-}) \mathbf{A}(\Lambda,\rho_{1,-})
%     \end{align*}
\end{proof}

When we wish to compose cobordisms, we must ensure compatibility of the skein completions.  
One way to do this is to ask our Legendrians $\Lambda$
to come equipped with an element $\lambda \in H^1(\Lambda)$, e.g. given by a 1-form, which match under the topological trivializations of the cobordism, and so that $\lambda(\partial D) > 0$ for the various surgeries which appear.  We term such cobordisms {\em composable}. 

\begin{corollary} \label{composable cobordism formula}
Let $S = (C_1, C_2, \ldots, C_n)$ be a collection of perfect disk surgery cobordisms with standard topological trivializations.  Let $C_i$ have positive end $\Lambda_i$ and negative end $\Lambda_{i-1}$ (union a Clifford torus); assume that the brane structures are also compatible.  We write $D_i$ and $\partial D_i$ for the corresponding surgery disk and its boundary.  Suppose given $\lambda_i \in H_1(\Lambda_i)$ which agree under the cobordisms and so that $\lambda_i(\partial D_i) > 0$. 

Then the following expression is well defined
$$\Psi_S := \mathbf{E}(\partial D_n) \cdots \mathbf{E}(\partial D_1) \in \widehat{Sk}_\lambda(\R \times \Lambda)$$
Here, $\lambda$ is the common value of the $\lambda_i$ and $\Lambda$ the common topological type of the $\Lambda_i$.  

Moreover,
$$\mathbf{A}(\Lambda_n)
\Psi_S   = \Psi_S  \mathbf{A}(\Lambda_0)$$
\end{corollary}

In the above situation, we say $\lambda$ {\em controls} the composable cobordisms.

\begin{corollary} \label{relation for admissible filling}
    In the situation of Corollary \ref{composable cobordism formula}
    suppose the final negative end is $\Lambda_0$ has an exact filling $L$.  
    Then the image of 
    $\mathbf{A}(\Lambda) 
    \Psi_S$ in 
    $Sk(L)$ is zero. 
\end{corollary}
\begin{proof}
    Immediate from Proposition \ref{necklace annihilates} and
    Corollary \ref{composable cobordism formula}.   
\end{proof}

\begin{remark}
    In practice,  the various $\Lambda_i$ will come with their own parameterizations, which match nontrivially (some Dehn twist) under the trivialization of the cobordism.   This may lead to the above expression being nontrivial to write in terms of some fixed parameterization of some $\Lambda_i$.  See e.g. Corollary \ref{dehn twist} for how this works
    for cubic planar graph Legendrians. 
\end{remark}

In the following sections we take on the task of finding conditions which ensure perfectness of disk surgery cobordisms.

\begin{remark} \label{why morse flow trees}
    Establishing perfectness amounts to excluding certain curves.  
    We will do this by estimating energy.  For Lagrangian cobordisms in symplectizations, cylindrical outside some fixed region $M \times [a,b]$, 
    there are (at least) 
    three natural energies to consider, see e.g.
    \cite{Ekholm-rational, Rizell-Surgery}. 

    For a curve $u$  with positive Reeb chords $\rho_1,\dots \rho_k$ and negative Reeb chords $\rho'_1,\dots \rho'_l$,  these $3$ energies are:

    \begin{align*}
        E_\omega(u) &= e^b\sum \ell(\rho_i) - e^a \sum \ell(\rho'_j) \\
        E_\lambda(u) &= e^b \sum \ell(\rho_i) \\
        E(u) &= E_\omega(u) + E_\lambda(u)
    \end{align*}
    All of these values are necessarily non-negative.  To exclude a holomorphic curve, it therefore suffices to show that one of these would be negative; of course this is in some sense most likely for $E_\omega$.  However, the most straightforward to estimate is $E$, and it is nontrivial to separate out the $E_\omega$ contribution.   Compare for instance the two results \cite[Lemma 5.1./5.3.]{Rizell-Surgery} where in one he establishes only an estimate for $E$ and in the other one an estimate for $E_\omega$ to rule out some disks. These methods are not sufficient for this text.

    However, in the Morse flow tree limit, $E_\omega$ has already been separated out.  This is why we have been able to prove our results when Morse trees are applicable.
\end{remark}

\section{Morse flow graphs for cobordisms} 
Our basic tool for counting curves in the cobordisms will be the adaptations in \cite{EK, Ekholm-Honda-Kalman} of the Morse flow tree technology of \cite{E}.  We review this below, and then develop some additional results in this context. 

\subsection{Review of results from \cite{EK, Ekholm-Honda-Kalman}}

\begin{construction}(Morse cobordism, \cite[Section 2.3./Remark 3.2.]{Ekholm-Honda-Kalman}) \label{Con: Morse Cobordism}
    Let $\tilde{L}$ be an eventually cylindrical exact Lagrangian cobordism in $(J^1(M) \times \R)$ and denote by $[a,b]$ an interval such that $\tilde{L} \cap (J^1(M) \times [a+\epsilon,b-\epsilon])^c$ is cylindrical. Then there is an associated Legendrian $L^{Mo}$ in $J^1(M \times [e^a,e^b])$ called the Morse cobordism of $\tilde{L}$: Let $\Lambda_+= \partial_{\infty} \tilde{L}$ and $\Lambda_-= \partial_{-\infty} \tilde{L}$. Furthermore let $h_+:[e^{b-\epsilon},e^b] \rightarrow \R$ be a non-decreasing function fulfilling the following properties:
    \begin{enumerate}
        \item $h_+(e^b)=e^b$
        \item $h_+(e^{b-\epsilon})=e^{b-\epsilon}$
        \item $h_+$ is given locally around $e^b$ by $e^b-t^2$ and locally around $e^{b-\epsilon}$ by $t+C$ (where $C$ is some positive number).
        \item The derivative of $h_+$ is $0$ only at $e^b$.
    \end{enumerate}
    We define a similar function $h_-:[e^a,e^{a+\epsilon}]\rightarrow \R$. Consider splitting $\Lambda_+=(\Pi_\C(\Lambda_+),z_+)$ into its Lagrangian projection and Reeb coordinate. Similarly for $\Lambda_-$. Then $L^{Mo}$ has the following properties:

    \begin{enumerate}
        \item In a neighborhood of $J^1(M \times \{e^b\})$ the Morse cobordism is given by:
        \begin{align*}
            (h_+(t) \Pi_\C(\Lambda_+), t, h'_+(t)z_+, h_+(t)z_+) 
        \end{align*}
        where we identify $J^1(M\times [e^a,e^b])=T^*M \times T^*[e^a,e^b] \times \R$. A similar statement is true at the negative end.
        \item All Reeb chords of $L^{Mo}$ are above $M \times \{e^a\}$ or $M\times \{e^b\}$. The Reeb chords above the positive end are in $1:1$ correspondence with the Reeb chords of $\Lambda_+$ with the index shifted up by $1$. The Reeb chords above the negative end are in $1:1$ correspondence with the Reeb chords of $\Lambda_-$ with the same index as in $\Lambda_-$.        
    \end{enumerate}
\end{construction}

\begin{proof}
    We will need some details of this construction, so we review it here.   The main idea is to first identify $J^1(M) \times \R=T^*M \times \R \times \R$ with $T^*(\Lambda) \times T^*\R_{>0}$ where we choose coordinates $(\lambda,z,t)$ on $J^1(M) \times \R$ and thus define the following map:

    \begin{align*}
        \Phi(\lambda,z,t)=(e^t\lambda,t,z)
    \end{align*}
    One easily verifies that this is a (non-exact) symplectomorphism. We use this symplectomorphism to transport $\tilde{L}$ into $T^*M \times T^*\R_{>0}$. By assumption outside of the interval $[a+\epsilon,b-\epsilon]$ $\tilde{L}$ is cylindrical which means it is given by $\Lambda_+ \times [b-\epsilon, \infty)$ close to $J^1(M) \times\{+\infty\}$. Let $(\Pi_\C(\Lambda),z_+)$ be a parametrisation of $\Lambda_+$ then $\Lambda_+\times \R$ is given by:

    \begin{align*}
        (\Pi_C(\Lambda_+),z_+,t)
    \end{align*}
    So the image of $L$ under this symplectomorphism over $T^*M \times T^*[e^{b-\epsilon},\infty)$ is given by:

    \begin{align*}
        (t\Pi_\C(\Lambda_+),t,z_+)
    \end{align*}
    We obtain a similar description over $T^*M \times T^*(0,e^{a+\epsilon}]$. Now we take the truncation $L^{con} = \Phi(\tilde{L}) \cap (T^*M \times T^*[e^a,e^b])$ and choose a function $h_+$ as in the statement of the Construction. One then defines $L^{Mo}$ to coincide with $L^{con}$ except above $M \times [e^{b-\epsilon},e^{b}]$ one deforms $L^{con}$ using $h_+$ (and similarly close to the negative end):
    \begin{align*}
        (h_+(t)(\Pi_\C(\Lambda_+),t,h_+'(t)z_+)
    \end{align*}
    As $\tilde{L}$ was initially exact, so is $L^{con}$ and one verifies that $h_+(t)z$ is an integral of $\alpha_{std}$  of $L^{Mo}$ close to the positive end which can be glued to a lift of $\alpha_{std}$ of $L^{con}$. Thus $L^{Mo}$ lifts to a Legendrian of $J^1(M \times [a,b])$ and by abuse of notation we refer to $L^{Mo}$ by this Legendrian.

    Reeb chords of $L^{Mo}$ are in bijection with self-intersections of its Lagrangian projection $\Pi_\C(L^{Mo})$, and $\tilde{L}$ had no such self-intersections as it was assumed to be an embedded Lagrangian. So away from the neighborhoods where we deformed using $h_+$, we have no Reeb chords. In these neighborhoods, we observe that we have Reeb chord above a point $(p,t)$ iff $\Lambda_+$ has a Reeb chord above $p$ and $h_+'(t)(z_+(p_2)-z_+(p_1))=0$ where $p_2$ and $p_1$ are the endpoints of the Reeb chord above $p$. For a Reeb chord of $\Lambda_+$ we can't have that $z_+(p_2)=z_+(p_1)$ since this would be an immersed double point. So the above can only be true if $h_+'(t)=0$ which happens exactly at $e^b$ (or $e^a$ for $h_-$).

    The statement about the indices follows directly by observing that the Morse index of a Reeb chord above $\Lambda_+$ is increased by $1$: Let $(f_2-f_1)(p') = x_1^2+\dots +x_k^2 - x_{k+1}^2 \dots - x_n^2 + C$ be a local Morse chart around a Reeb chord of $\Lambda_+$. Then in this chart $(f_2-f_1)(x_1,\dots,x_n)h(t)= (x_1^2+ \dots x_k^2-x^2_{k+1}-\dots -x_n^2+C)(e^b-t^2)$ which raises the index by $1$. At the negative end there is a $+t^2$ term added so the Morse index remains the same. In addition, if $\gamma$ is a capping path of any such Reeb chord in $\Lambda_{\pm}$ then we can include it into $L^{Mo}$ and obtain a corresponding capping path for the Reeb chord there. So the full index is either shifted up by $1$ or agrees with the initial one.
\end{proof}

Our curve counting strategy will  be to first count certain curves in $L^{Mo}$ and then use the following result to retranslate it into a statement about curves in the cobordism:

\begin{definition}
    Let $M$ be a any manifold and $g$ a metric on $M \times [e^a,e^b]$ and $L^{Mo}$ be a Morse cobordism in $J^1(M\times [e^a,e^b])$ obtained through Construction \ref{Con: Morse Cobordism}. We say that $g$ is pseudo-flat if it fulfills the following properties:

    \begin{enumerate}
        \item Close to $e^a$ and $e^b$: $g$ decomposes into products $g_- + dt^2$ and $g_+ + dt^2$.
        \item For each Reeb chord $\rho$ of $L^{Mo}$ at the positive end there is a neighborhood $U_\rho$ in $M$ such that $g_+$ is the Euclidean metric on this neighborhood. Similarly for Reeb chords above the negative end.
        \item $g$ has bounded geometry.
    \end{enumerate}
\end{definition}

\begin{theorem}(Disk counting, \cite[Lemma 1.4./Theorem 1.5.]{Ekholm-Honda-Kalman}) \label{thm: ehk flow tree count}
Let $\tilde{L} \subset J^1(M) \times \R$ be a Lagrangian cobordism cylindrical outside $[a+\epsilon,b-\epsilon]$ and $L^{Mo}$ the associated Morse cobordism from Construction \ref{Con: Morse Cobordism}. For $\gamma \in \R$, we write $\gamma L$ for the rescaling of $L \subset J^1(M) \times \R$ in the fibers of $J^1(M)$. Furthermore let $g$ be a pseudo-flat metric. Then there exists:

\begin{enumerate}
    \item A family $L_{\delta,\sigma} \subset J^1(M) \times \R$ of eventually cylindrical Lagrangian such that $L_{\delta,\sigma}$ is exact Lagrangian isotopic to $\sigma L$.
    \item A family of eventually cylindrical almost complex structures $J_{\delta,\sigma}$.
\end{enumerate}

Then for sufficiently small $\delta,\sigma>0$ there is a $1:1$-correspondence between rigid flow trees with one positive puncture at $L^{Mo}$ defined by $g$ and rigid holomorphic discs with one positive puncture tangent to $L_{\delta,\sigma}$ defined by $J_{\delta,\sigma}$, assuming that the space of rigid flow trees is transversely cut-out.
\end{theorem}

The condition that $g$ is pseudo-flat is inherited from the study of solutions near Reeb chords of $L_{\delta,\sigma}$ in \cite[Section 5]{Ekholm-Honda-Kalman}.  (In that article,  only the case $\dim M = 1$ was presented, but, as explained to us by Ekholm, the same argument works in general using pseudo-flatness.) 

%This restriction (except bounded geometry) can presumably be dropped if one does a more careful study of the tree to disk gluing process near the Reeb chords at infinity. 

\begin{remark}
For our applications,  we will need to exclude the possibility of higher genus curves.  In fact the arguments of \cite{E} already suffice to show that such curves will in this situation limit to flow graphs (see e.g. discussion is \cite[Sec. 4]{Scharitzer-Shende}), so showing that there are no flow graphs implies there were no curves.
(Note we do not require any  transversality or gluing results regarding higher genus flow graphs.) 
\end{remark}

\subsection{Some additional lemmas}

In \cite{E} (and so also \cite{EK, Ekholm-Honda-Kalman}), it is assumed that all front projections are either two dimensional with generic singularities, or, if higher dimensional, have only cusp-edge singularities.   The fronts of (the $L^{Mo}$ of) our cobordisms will be three dimensional, and do not have only cusp edge singularities.  
So, we cannot directly appeal to \cite[Theorem 1.1]{E} to ensure the existence of some deformation of $g$ for which rigid flow trees are transversely cut-out.  

Instead, we will have to first directly show that our trees avoid all the singularities other than cusp edges, and are transversely cut out; after this, we can apply the results of \cite{E}.  The following lemmas will help us to do this.  They hold for Morse flow trees in $L^{Mo}$, without any further assumptions on the geometry of the front projection.  (The arguments are essentially trivial, but carried out in sufficient detail as may obscure this fact.)

$L^{Mo}$ close to the boundary decomposes into a product both on the level of the Legendrian and on the level of the metric.  This constrains the possibilites for flow graphs considerably. The following Lemma is a sharpening of a statement from \cite[Lemma 4.11.]{Ekholm-Honda-Kalman}:

\begin{lemma}(\cite[Lemma 4.11.]{Ekholm-Honda-Kalman}) \label{Lem:Flow graph restrictions}
    Let $L^{Mo}$ be a Morse cobordism and $g$ a metric which decomposes into $g_{\pm} + dt^2$ close to the positive/negative end of $L^{Mo}$. Then there are $2$ possible types for connected components of a Morse flow graph:
    \begin{enumerate}
        \item Boundary flow graphs: Flow graphs which are completely contained above $M \times \{e^b\}$ or $M \times \{e^a\}$. These are in $1:1$ correspondence with the flow graphs of $\Lambda_+$ and $\Lambda_-$.
        \item Interior flow graphs: All edges of the flow graph are above $M \times (e^a,e^b)$.
    \end{enumerate}
    Furthermore, an interior flow graph component has only the following allowed dynamics close to the ends.
    \begin{enumerate}
        \item There are no negative/positive punctures above the positive/negative end.
        \item A flow graph has only $1$-valent vertices above punctures.
    \end{enumerate}
\end{lemma}

\begin{proof}
    Let $\mathcal{G}' \subset \mathcal{G}$ be a connected component of a flow graph and let $v$ be a vertex of $\mathcal{G}'$ which is above the negative and positive ends and has an adjacent edge which is not contained in that end. Assume that it is above the negative end. Then by Construction \ref{Con: Morse Cobordism} this cobordism close to the negative end has the form:

    \begin{align*}
        (h(t)\Pi_\C(\Lambda_-),t,h'(t)z(\Lambda_-), h(t)z(\Lambda_-))
    \end{align*}
    where $\Lambda_-$ is an embedding of the negative end into $J^1(M)$, $\Pi_\C$ is the Lagrangian projection and $z$ is the Reeb chord projection and $h$ is a function which locally looks like $t^2$ around $e^a$. As the metric decomposes into a product in such a neighborhood the gradient between two sheets looks like:
    \begin{align*}
        (\grad(f_1-f_2)(s)h(t), (f_1-f_2)(s)h'(t))
    \end{align*}
    where $f_2$ and $f_1$ are locally defined functions (not necessarily on an open subset) around the edge such that $\Lambda$ coincides with the $1$-jet lift of $f_2$ and $f_1$. 

%    Note that, in general, on a manifold $M$ with a submanifold $N \subset M$, if a vector field is tangent to $N$, then the only way a flow line from $M \setminus N$ can limit to a point on $N$ is if said point is a zero of the vector field. 
    
    If $t=e^a$ then the last coordinate of the differential vanishes; flow lines through any such point are thus contained in the $t = e^a$ slice. 
    Consequently, a flow line from the complement of this slice can only limit to the slice if it limits to a critical point.  Thus a flow graph vertex on the slice to which a flow line from the slice complement arrives must sit above a Reeb chord, and the flow line from the complement must be carried by the corresponding sheets.  Such vertices are always $1$-valent vertices. 

    The last statement to show is that there are no negative/positive punctures above the positive/negative end for an interior flow graph component. We will argue that there are no positive punctures above the negative end. The other statement is proven similarly. Consider the local model of a Reeb chord which is given by $(f_2-f_1)(x)h(t)=(x_1^2+\dots x_k^2-x_{k+1}^2-\dots -x_n^2 + C)(e^a+ t^2)$ where $(x_1,\dots,x_n)$ are local coordinates on $M$. Then as the metric decomposes into a product, we immediately observe that the unstable manifold of $\grad(f_1-f_2)h$ is completely contained in the $t=e^a$ slice. So any edge limiting away from this puncture will be completely contained in the negative slice.
\end{proof}

%To conclude this section, we will show that if the Lagrangian cobordism $\tilde{L}$ was constant in a neighborhood of Reeb chords and these neighborhoods were sufficiently generic then no non-trivial graphs can live in these slices. \sayVS{what?}

We now give conditions which ensure that Reeb chords have only trivial flow lines. 

\begin{definition} \label{legendrian basic}
    Let $\Lambda \subset J^1 M$ be a Legendrian, 
    and $\rho$ a Reeb chord of $\Lambda$.  We say 
    $\Lambda$ is basic around $\rho$ if there is a neighborhood
    $U \subset M$ beneath $\rho$ such that the following hold over $U$:  
    \begin{enumerate}
        \item $\Lambda \to M$ is a covering,
        \item The front projection of $\Lambda$ is injective,
        \item $\rho$ is the only Reeb chord of $\Lambda$.
    \end{enumerate}
\end{definition}

    Being basic at a Reeb chord is a generic condition on $\Lambda$, which can be expected to fail in codimension one in families. 

\begin{definition}
    Let $\tilde{L}$ be an eventually cylindrical Lagrangian cobordism between $\Lambda_-$ and $\Lambda_+$. Let $\rho$ be a Reeb chord around which $\Lambda_-$ is basic for neighborhood $U$. 
     If we have in addition $\tilde{L} \cap (J^1(U) \times \R) = (\Lambda_- \cap J^1(U)) \times \R$,
    then we say that $\tilde{L}$ is basic around $\rho$.
\end{definition}

\begin{lemma}(Trivial tree) \label{Lem:Trivial tree} 
    Let $\rho$ be a Reeb chord of $\Lambda_-$ and $\tilde{L}$ an eventually cylindrical Lagrangian cobordism which is basic around $\rho$. Furthermore, let $g$ be a metric on $M \times [e^a,e^b]$ such that $g|_{U\times [e^a,e^b]} = g_- +dt^2$ where $U\times\R$ is the cylindrical neighborhood of $\rho$ and $g_-$ is a metric on $U$. Then for the associated Morse cobordism $L^{Mo}$ there is a unique flow graph with $1$ positive puncture completely contained in $J^1(U \times [e^a,e^b])$ which is a transversally cut-out flow tree and lives in the slice $\rho \times [e^a,e^b]$.
\end{lemma}
\begin{proof} 
        Let $\rho$ be a Reeb chord with a neighborhood in $\tilde{L}$ as above. Then there is a neighborhood of $\rho \times [e^a,e^b]$ such that $L^{Mo}$ is the graph of $f_{i}(s)h(t)$ where we chose some local parametrisation of the two sheets of $\Lambda_-$ between which the Reeb chord lives. If $f_2>f_1$ then set $f=f_2-f_1$. In addition, $h(t)$ is a non-decreasing positive function which locally looks like $t^2$ around $e^a$ and like $-t^2$ around $e^b$ whose derivative is $0$ iff $t=e^a,e^b$. As we are using a product metric the gradient of $-f(s)h(t)$ obeys this product structure. As such the gradient in the $M$ direction is dictated by $-h(t)\grad f(s)$ of $\Lambda_-$ and the gradient in the $[e^a,e^b]$-direction is given by $-f(s)h'(t)$ which is negative if $t \neq e^a,e^b$ and $0$ if $t=e^a,e^b$. So if $\Lambda_-$ has a Reeb chord above the point $s'$. Then the gradient given by $(\Lambda_{\Gamma})_{1}-(\Lambda_{\Gamma})_2(s')$ vanishes. The gradient of $fh$ vanishes on $s \times[e^a,e^b]$ exactly above $s \times \{e^a\}$ and $s \times \{e^b\}$ and points from the positive to the negative end. Thus, we have a basic Morse flow tree above this line with a positive puncture at $s \times \{e^b\}$ and a negative puncture at $s \times \{e^a\}$ at the Reeb chords $\rho_+$ and $\rho_-$ respectively. By the geometric and formal dimension formula from \cite[Definition 3.4./3.5.]{E} it follows directly that this flow tree is transversally cut-out.

        To finish the proof, we only need to eliminate the possibility that there are any other interior flow graphs. Let $\mathcal{G}$ be any connected flow graph completely contained within this region and $e$ any edge. Then $e$ is oriented in such a way that it points towards the $e^a$-slice and no edge of $e$ is ever tangent to a $t$-slice: This can only be possible if $h'=0$ or the difference of the functions supporting $e$ vanishes. The first is excluded as we consider an interior flow graph and the other as we consider a neighborhood which has no intersections of function sheets. So the orientation of any edge is always pointing strictly downwards.

        Now let $e^a=t_0 < \dots t_{n-1} < t_n=e^b$ be a decomposition such that all vertices of the flow graph appear on the $t_i$-slices. Let $(\tilde{f}_2(s)h(t))$ and $(\tilde{f}_1(s)h(t))$ be the functions defining those sheets close to an edge living in a slice $U \times [t_i,t_{i+1}]$. As a shorthand for the function difference we will simply write $e_H(s)=(\tilde{f}_2(s)-\tilde{f}_1(s))$. Now, we define the horizontal energy loss by $H(e)=e_H(s_{i+1})-e_H(s_i)$. We immediately, see that the gradient equation of $L^{Mo}$ along each $t$-slice coincides with the respective gradient equation of $\Lambda_-$ for these function differences up to multiplication by $h(t)$. So $H(e)$ is a non-negative number as the projection of $e$ is a reparametrisation of the gradient of $e_H$ and this number is thus $0$ iff the projection is constant which for a gradient flow is only possible at a Reeb chord. So if $e$ is any edge of the flow graph not projecting to the Reeb chord this will give a positive $H(e)$. Now, if we sum over all $e_H$ for any edge which has a contribution above $t_{i+1}$ we obtain:

        \begin{align*}
            \sum e_H(s_{i+1}) = \sum e_H(s_i) + \sum H(e)
        \end{align*}
        so inductively, we obtain:
        \begin{align*}
            E(\rho) = KE(\rho) + \sum H(e)
        \end{align*}
        which is only possible if $K=1$ and all $H(e)=0$. Meaning that no edge is supported on a sheet different from the ones which support the Reeb chord. Since, we do not consider flow graphs with marked points interior $2$-valent vertices are ruled out and the only possible graph remains the trivial one. 
\end{proof}

\begin{proposition}(Energy bound) \label{Prop:Energy bound}
    Let $L^{Mo}$ be a Morse cobordism above $J^1(M \times [e^a,e^b])$ and $B_\rho$ be a basic neighborhood of a Reeb chord $\rho$ such that $\hat{L}$ is cylindrical above $B_\rho$ and $g+dt^2$ a pseudo-flat metric on $M\times [e^a,e^b]$. Then any Morse flow graph G with the condition that (i) $\rho$ is the unique positive puncture and (ii) it leaves $B_\rho$ has a minimal energy loss $e^aE_\rho$.
\end{proposition}

\begin{proof}

    This is almost obvious: Take a slightly smaller neighborhood $B'_\rho$ and consider a Morse flow graph $G$ as above and restrict it to $(B_\rho \setminus B'_\rho) \times [e^a,e^b]$. By $g$-invariance and the cylindrical property. The projection of $G|_{\{B_\rho \setminus B'_\rho\}}$ to $J^1(B_\rho\setminus B'_\rho)$ consists of a collection of gradient flow lines. As $G$ originally left $B_\rho\times [e^a,e^b]$ there must be some sequence of flow lines $(e_1,\dots e_n)$ which lift to the upper sheet of the Reeb chord and intersect $\partial B_\rho\times [e^a,e^b]$. The length of the sum of the projected flow lines is bounded below by a geodesic and the energy loss is bounded below by the minimum of all gradient differences of all sheets. As we removed a neighborhood of the Reeb chord the norm of the gradient differences is bounded below. Thus, we arrive at an estimate $E$ which is the minimum norm of all gradient differences and the minimal length of a geodesic in the base. To obtain an estimate for the original graph, we must adapt the original gradient equation by multiplying it by $e^a$. So $E_\rho = \text{min} ||\grad(f_2-f_1)]]d(\partial B_\rho, \partial B'_\rho)$.
\end{proof}

\section{Weakly perfect surgery cobordisms for satellites}

We will now identify hypotheses which ensure that we can deform $\tilde{L}$ from Construction \ref{Con:Basic cobordism} into a 
Lagrangian cobordism $L$ which is weakly perfect in the sense
of Definition  \ref{def: weakly perfect}.

To count curves in the cobordism, we use Theorem \ref{thm: ehk flow tree count}. As a consequence, we are limited to Legendrian cobordisms which are contained in jet bundles, and which moreover have the property that all index 1 Reeb chords of the ends are contained in the jet bundle as well.

\begin{definition}
    Let $\Lambda, \Lambda' \subset Y$ be Legendrian surfaces. We say that $\Lambda$ is a \textbf{satellite} of $\Lambda'$ if $\Lambda$ is contained in a Darboux neighborhood $U$ of $\Lambda'$. 

    In this case we may consider the fiber rescalings
    $\Lambda_\sigma \subset U$.      
    We say that $\Lambda$ is \textbf{dominated} by $\Lambda'$ if $\Lambda'$ has no index $1$ Reeb chords and $\Lambda_\sigma \subset U$ for all $0<\sigma\leq1$ has no index $1$ Reeb chords escaping $U$.  
\end{definition}

\begin{theorem} \label{Thm: Basic Cobordism}
    Suppose given some $\Lambda' \subset Y$ and some $\Lambda$ dominated by $\Lambda'$.  Suppose given an admissible surgery disk $D$ for $\Lambda$.  Assume in addition: 
        \begin{enumerate}
        \item $\Lambda$ is basic (Def. \ref{legendrian basic}) around all its index $1$ Reeb chords.
        \item The Darboux chart $J^1  \Lambda'$ is compatible with the surgery-defining chart of Figure \ref{Fig:Disk-surgery}. 
        \item There are no index-$1$ Reeb chords of $\Lambda$ above the image of $D \subset J^1(\Lambda') \to \Lambda'$.  
    \end{enumerate}
    Then the disk surgery cobordism of \ref{Con:Basic cobordism} is isotopic to a weakly perfect
    disk surgery cobordism. 
\end{theorem}
\begin{proof}
We will count Morse flow trees.  
Let $\rho_1,\dots,\rho_n$ be an order on the index-$1$ Reeb chords of $\Lambda$ such that the action is non-decreasing. 
Let $\rho_{i,+}$ denotes the index-$2$ Reeb chord above the positive end and $\rho_{i,-}$ the index-$1$ Reeb chord above the negative end originating from $\rho_{i}$. 

Fix a pseudo-flat metric $g+dt^2$ on $\Lambda'$.  
Let $E_1, \dots E_n$ be the energy bounds from $g$ as in Proposition \ref{Prop:Energy bound}.  Let $a<b$ be two real numbers such that $e^b \ell(\rho_{j,+}) - e^a\ell(\rho_{i,-})-e^aE_j$ is negative if $i\leq j$. This is obviously true if $a=b$ as the order of the Reeb chords was energy non-decreasing. So we can find a pair $(a,b)$ such that this is true for all $i \leq j$.

By Construction \ref{Con:Basic cobordism} we find a cobordism of length less than $e^{b-a}$. WLOG the lower end is at $e^a=1$.
Any Morse flow graph with one positive puncture at 
$\rho_{j,+}$ and one negative puncture at $\rho_{i,-}$ for $i < j$ must leave a neighborhood of $\rho_{j,+}$ which we used to define the $E_j$, and thus have energy more than $E_j$. Recall that the energy of a Morse flow graph in such a cobordism is given by $e^{b'} \ell(\rho_{j,+}) - \ell(\rho_{i,-})$. Since $b'\leq b$ this value must also be smaller than $E_j$.  This is a contradiction, so there could not have been such a curve to begin with.  In addition, by Lemma \ref{Lem:Trivial tree} there is exactly one flow graph properly contained in such a neighborhood which coincides with the trivial cylinder. 

So any other flow graph must leave such a neighborhood, but again the energy constraint of Proposition \ref{Prop:Energy bound} rules out any such graph by a similar argument.  

Thus, the flow graph count obeys the following:
    \begin{enumerate}
        \item In the slice $\rho_{i} \times [a,b]$ there is exactly one flow graph with one positive puncture at $\rho_{i,+}$ and one negative puncture at $\rho_{i,-}$.  It is homotopic to to a trivial cylinder, and its moduli space is transversely cut out.  There are
        no other Morse flow graphs with these endpoints.
        \item If $j<i$ then there are no flow graphs with positive puncture at $\rho_{j,+}$ and negative puncture at $\rho_{i,-}$.
    \end{enumerate}
Now  Theorem \ref{thm: ehk flow tree count} asserts the 
    curves 
    for $L_{\sigma,\delta}$ with respect to $J_{\sigma,\delta}$ 
    correspond to the above Morse trees.  The above enumerated properties
    of said trees then translate to the assertion that $L_{\sigma,\delta}$
    is weakly perfect.
\end{proof}

\section{Perfect disk-surgery cobordisms for cubic planar graph Legendrians}

We turn to the study of cubic planar graph Legendrians, introduced and studied in \cite{Treumann-Zaslow, Casals-Zaslow, SSZ, Scharitzer-Shende}.  Given
an edge $e$ of a graph $\Gamma$, we write $\Gamma_e$ for the result of flipping $e$, and also preserve the notation $e$ for the new edge in $\Gamma_e$. 

Casals and Zaslow have shown that
$\Lambda_\Gamma$ and $\Lambda_{\Gamma_e}$
are related by Legendrian disk surgery \cite[Theorem 4.21]{Casals-Zaslow}.  For technical reasons, we cannot directly apply Theorem \ref{Thm: Basic Cobordism} to their construction.\footnote{Specifically, this is because the contact form on the Casals-Zaslow
Darboux chart describing the surgery on 
$\Lambda_\Gamma$
differs from the standard contact form on $J^1 S^2$.  We do not know whether the Casals-Zaslow Darboux chart extends to a chart appropriate for our flow tree counting methods above, nor whether it extends to a contact form ensuring that the Legendrians are Reeb positive.  
(By contrast, the Legendrians are self-evidently Reeb positive in the original contact form.)  The sensitivity of our discussions to contact forms is an artifact of the present absence of a fully developed `skein-valued SFT'.}
Instead, we use other results of the same article \cite{Casals-Zaslow} to construct a different cobordism, to which we can apply Morse flow technology.  Unlike the cobordism of Construction \ref{Con:Basic cobordism}, this new cobordism will not be arbitrarily short, and we will have to use more of its explicit structure to rule out unwanted Morse flow graphs.

\subsection{Cobordism}

\begin{construction} \label{CPG: Initial choices}
    Consider an unknotted Legendrian sphere $U\subset \R^5 = J^1 \R^2$, with the standard flying saucer front in $\R^3$.  We choose it so the base projection to $\R^2$ has image the unit disk $D_1$.

     Let $\Gamma$ be a cubic planar graph on $S^2$, fix $e$ a distinguished edge of $\Gamma$. Then there is an embedding of $\Lambda_\Gamma$ a standard neighborhood of  $U \subset \R^5$ which fulfills the following properties:

    \begin{enumerate}
        \item The base projection of the lift of the graph $\Gamma \subset \Lambda_\Gamma$ is contained within $D_{1-\epsilon}$ for the exception of a small neighborhood of the edge $e$ which lies outside of $D_{1-\epsilon}$.
        \item The index $2$ Reeb chords are above $D_{\epsilon}$ and the index $1$ Reeb chords are contained above $D_{1-2\epsilon}\setminus D_{\epsilon}$ near the upper hemisphere of $U$. 
        \item \label{pair of sheets gradient}  For any pair of sheets with one on the upper and one on the lower hemisphere, all local gradient differences  are pointing outward, and are transverse to $D_{\delta}$ for $1-\epsilon>\delta>\epsilon$.
    \end{enumerate}
\end{construction}

\begin{proof}
    The first $2$ conditions can be initially met by observing that the index $1$ Reeb chords are defined by the initial choice of function $f_\Gamma$ and then choosing an appropriate identification with the Legendrian unknot $S^2 \subset \R^5$.

    The second one is easily observed by seeing that if we rescale to some $\sigma \Lambda_\Gamma \subset J^1(S^2)$ for small $\sigma > 0$, the gradient differences between the upper and lower hemisphere eventually become dominated by the gradients of the upper and lower hemisphere of $S^2$ which are non-vanishing for $D_{\delta}$ with $1-\epsilon>\delta>\epsilon$. Thus statement \eqref{pair of sheets gradient} is true and thus for small $\sigma>0$ the statement is also true for $\sigma \Lambda_\Gamma$.
\end{proof}

We now construct a disk surgery cobordism.  Let us first note some differences our previous cobordism in Construction \ref{Con:Basic cobordism}.  Here we use a slightly different presentation of the Clifford torus $\T_{Cl}$, namely as the Legendrian associated to the tetrahedron graph. This allows us to use the Legendrian weave calculus established by Casals-Zaslow \cite[Section 4]{Casals-Zaslow}. Additionally, the non-cylindrical region of the cobordism will not be local to the disk; instead, it sits above the entire front in Figure \ref{Fig:CPG-Cobord-1}. 

\begin{construction} \label{CPG: Cobordism}
    Let $\Lambda_\Gamma$ be a cubic planar graph Legendrian associated to a graph $\Gamma$ whose front projection satisfies the properties of Construction \ref{CPG: Initial choices}. Then for each edge $e$ bounding the large face of $\Lambda_\Gamma$ there is an embedding of $\T_{Cl}$ such that there is a cobordism $\tilde{L}$ from $\Lambda_\Gamma \sqcup \T_{Cl}$ to $\Lambda_{\Gamma_e}$ where $\Gamma_e$ is the cubic planar graph obtained from $\Gamma$ by flipping the edge $e$, see Figure \ref{Fig:Edge-Flip}. In addition, this cobordism is trivial inside $J^1(\bar{D}_{1-\epsilon})$ and consists of Legendrian isotopies, an index $1$ and an index $0$-surgery.
\end{construction}

\begin{figure}
    \begin{picture}(50,175)
    \put(-75,0){\includegraphics[width=10cm, trim=0cm 0cm 0cm 0cm, clip]{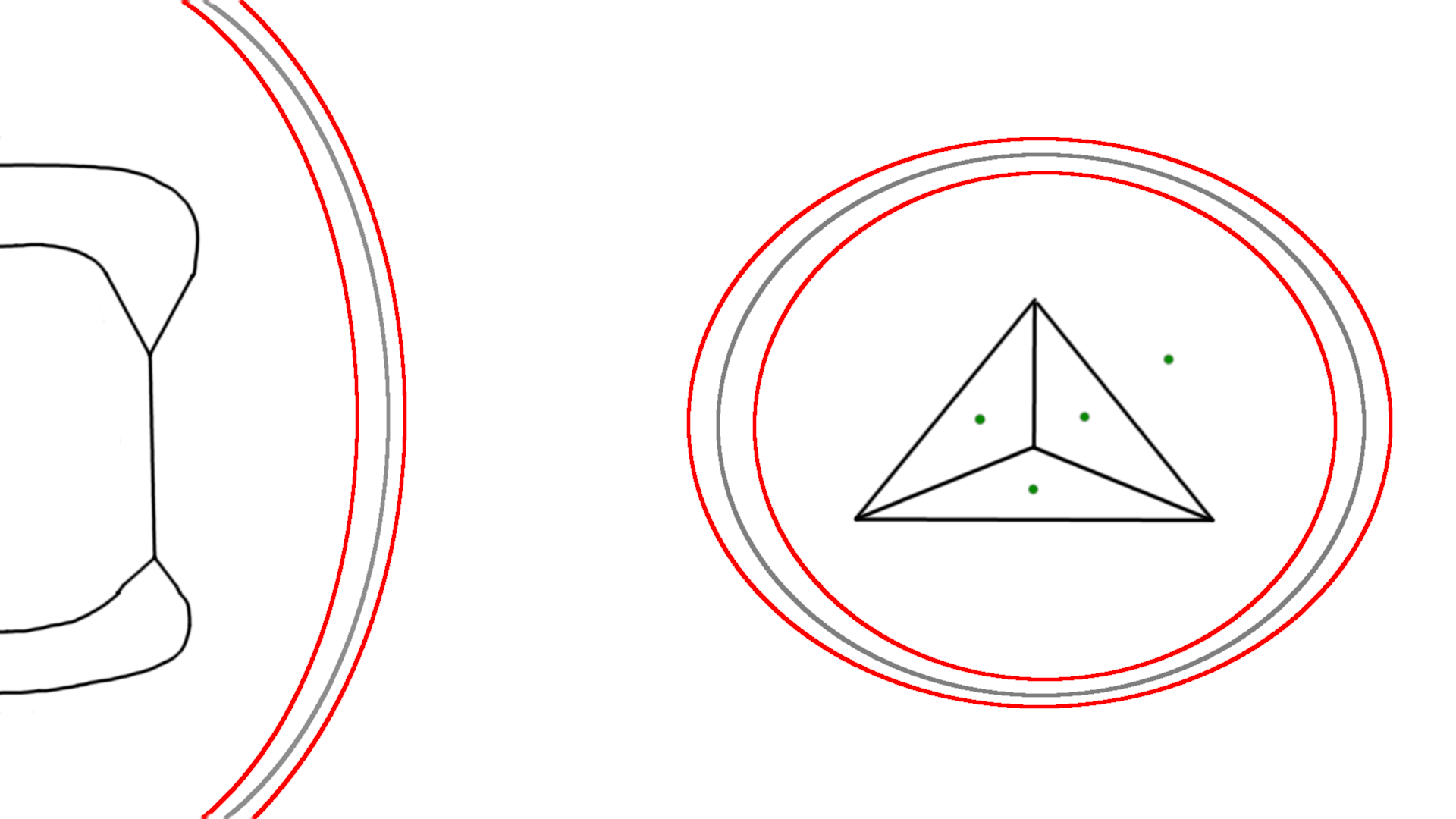}}
    
    \end{picture}
    \caption{The base projection of $\Lambda_\Gamma$ containing both the tetrahedron Legendrian (on the right) and a neigborhood of the the edge $e$. Depicted in black are the self-intersections of the front projection coming from $\Gamma \subset S^2$ while the grey lines are the self-intersections coming from the inclusion when satelliting $J^1(S^2)\rightarrow S^5$. Depicted in red are cusp-edges. The index-$1$ Reeb chords in this area are indicated in green. For a more detailed explanation, see \cite{Casals-Zaslow}.}  
    \label{Fig:CPG-Cobord-1}
\end{figure}

\begin{figure}
    \begin{picture}(50,125)
    \put(-150,0){\includegraphics[width=15cm, trim=0cm 4cm 0cm 0cm, clip]{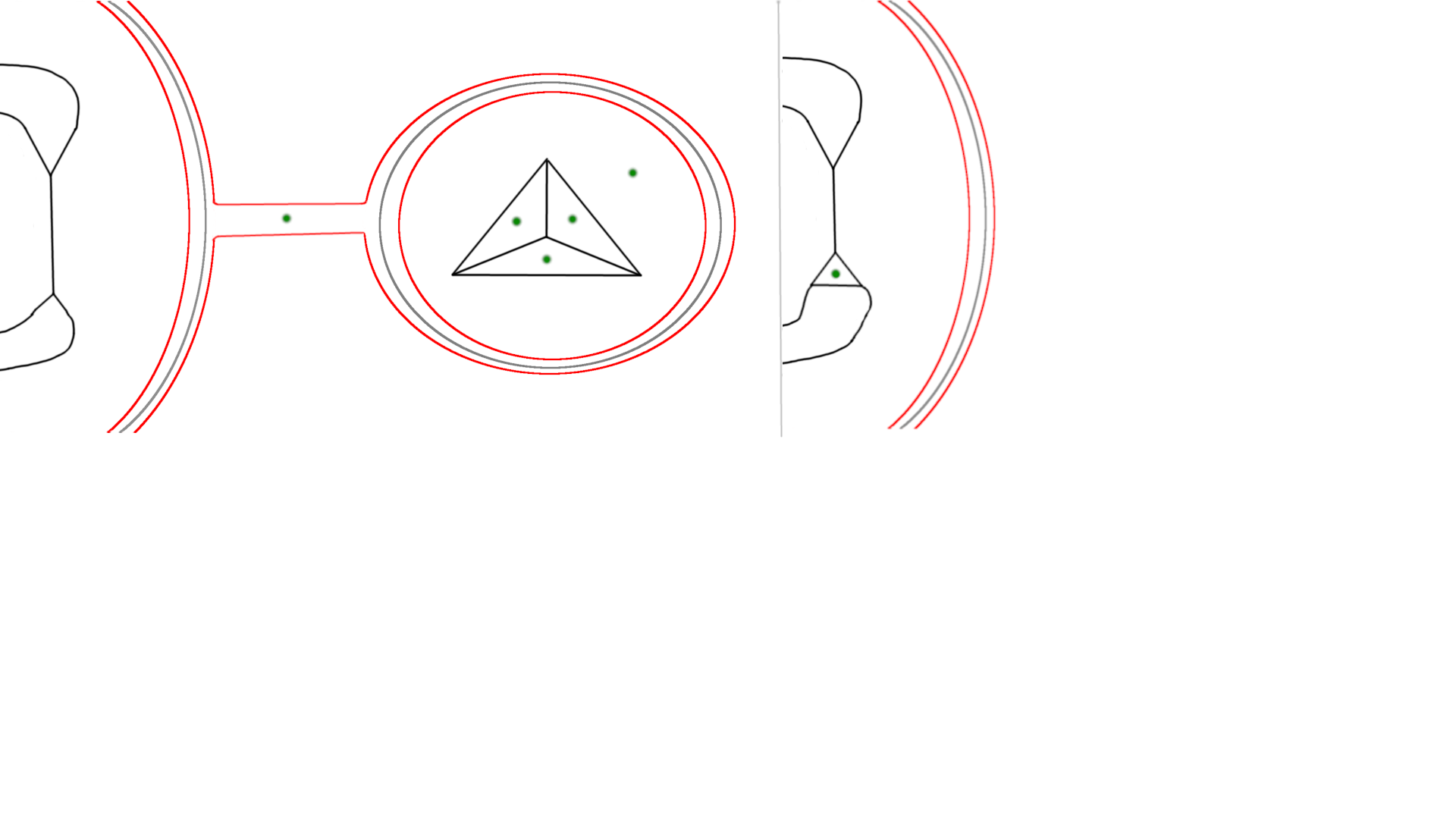}}
    
    \end{picture}
    \caption{On the left the resulting Legendrian after the index $0$-surgery which corresponds to a $1$-handle attachment. On the right the resulting simplified Legendrian. This Legendrian is obtained from the original Legendrian by replacing the vertex $v$ with a small triangle.}  
    \label{Fig:CPG-Cobord-2}
\end{figure}

\begin{proof}
    We construct a cobordism as a composition of pieces from \cite{Casals-Zaslow}.  

    \begin{enumerate}
        \item 
    Consider a neighborhood $N$ of $e$ as in Figure \ref{Fig:CPG-Cobord-1} which also contains the cusp-edges of $\Lambda_\Gamma$. Place $\T_{Cl}$ such that the base projection of $\T_{Cl}$ and $\Lambda_\Gamma$ are disjoint, see Figure \ref{Fig:CPG-Cobord-1}. Then we follow the same construction as in \cite[Theorem 4.10.4]{Casals-Zaslow}. Basically this is a Legendrian index $0$ surgery between the outer front cusp-edges, depicted in the Figure on the left of \ref{Fig:CPG-Cobord-2}. Followed by a Legendrian isotopy resulting in the graphs being       ``added" along a vertex each where the vertex on $\T_{Cl}$ is chosen arbitrarily and any vertex $v$ adjacent to $e$ will suffice. The two stages after the surgery and after the isotopy are depicted on the right in Figure \ref{Fig:CPG-Cobord-2}. The resulting Legendrian is associated to the cubic planar graph $\Gamma'$ is the same as replacing the chosen vertex $v$ with a triangle where each edge $e_i$ formerly attached to $v$ is now attached to a unique vertex of the triangle.
    \item 
    The second stage is an index $1$ surgery in the sense of \cite[Theorem 4.10.2)]{Casals-Zaslow}. In our case, we apply it to the edge $e'$ connecting the triangle to the former vertex $v'$ of $e$ which was not chosen in the previous stage. This move consists of an initial isotopy, an index $1$ surgery and then a final isotopy. We will not in detail describe this isotopy but just remark that it can be made in an arbitrarily small neighborhood of the edge $e'$. The resulting Legendrian is depicted on the left of Figure \ref{Fig:CPG-Cobord-3}.
    \item 
    Lastly, the Legendrian right now is not yet a Legendrian associated to a cubic planar graph: In the previous step we introduced an index $0$ Reeb chord on the face obtained from the triangular face and the "opposite" face before. However, as can be seen in Figure \ref{Fig:CPG-Cobord-3}. There is an additional index $1$ Reeb chord inherited from the triangular face. So a final isotopy is designed to merge this index $0$ and index $1$ Reeb chord. Again, this can be done outside $\bar{D}_{1-\epsilon}$. After this final isotopy, we have obtained a Legendrian which is associated to a cubic planar graph $\Gamma_e$ obtained from $\Gamma$ by flipping the edge $e$. This Legendrian is depicted on the right of Figure \ref{Fig:CPG-Cobord-3}.        
\end{enumerate}

\end{proof}

\begin{figure}
    \begin{picture}(50,150)
    \put(-60,0){\includegraphics[width=12cm, trim=0cm 4cm 5cm 0cm, clip]{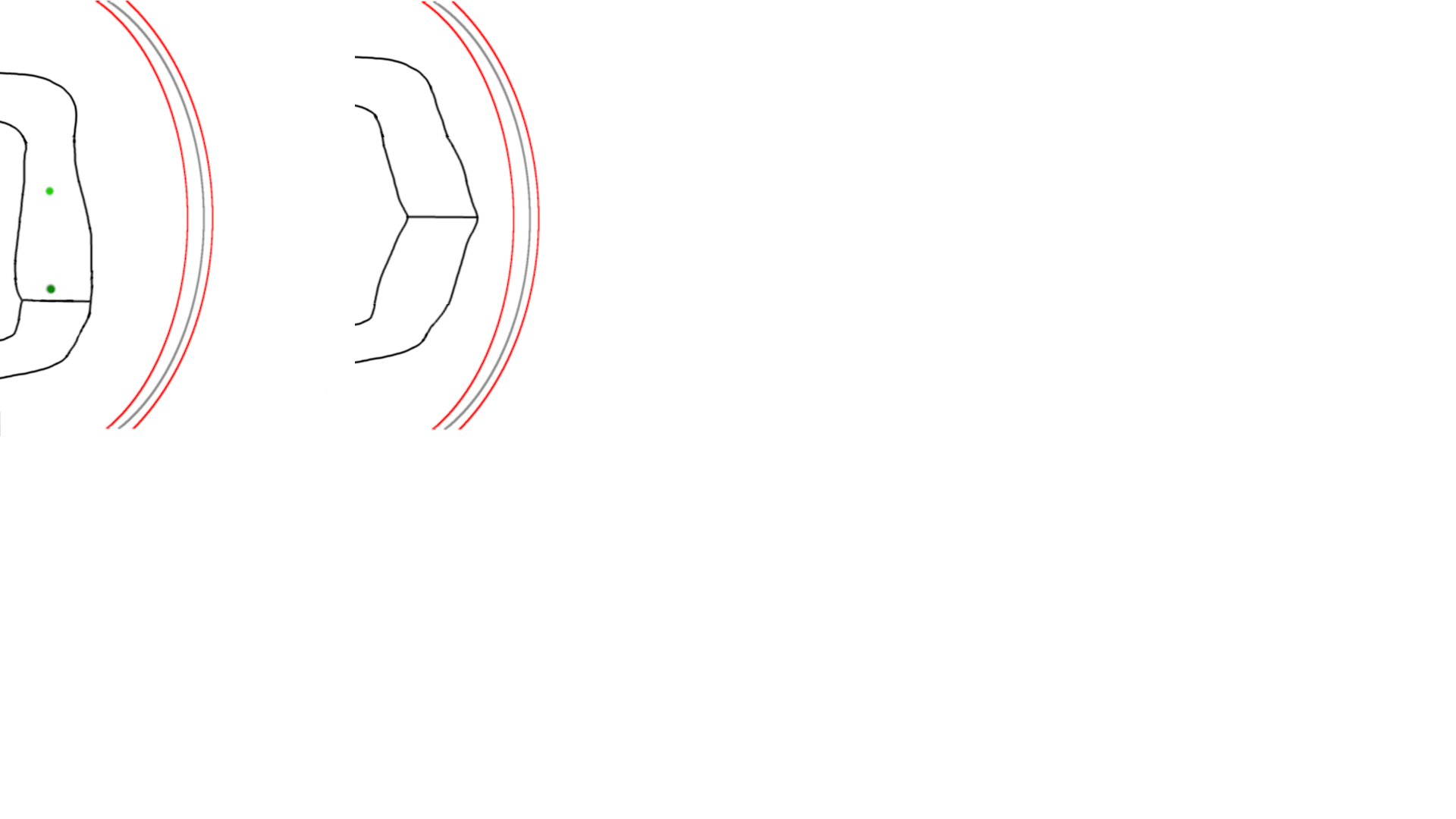}}
    
    \end{picture}
    \caption{On the left the Legendrian obtained from index $1$ surgery. The additional index $0$ Reeb chord is depicted in light green. On the right is the Legendrian after a small isotopy cancelling the Reeb chords and normalising the picture. This Legendrian has all the original Reeb chords.}  
    \label{Fig:CPG-Cobord-3}
\end{figure}

    The first stage of the above cobordism will be called the ``Clifford sum'' cobordism (following \cite[Theorem 4.10]{Casals-Zaslow}) at the vertex $v$ which was chosen.

\begin{figure}
    \begin{picture}(50,150)
    \put(-125,0){\includegraphics[width=12cm, trim=0cm 2cm 0cm 0cm, clip]{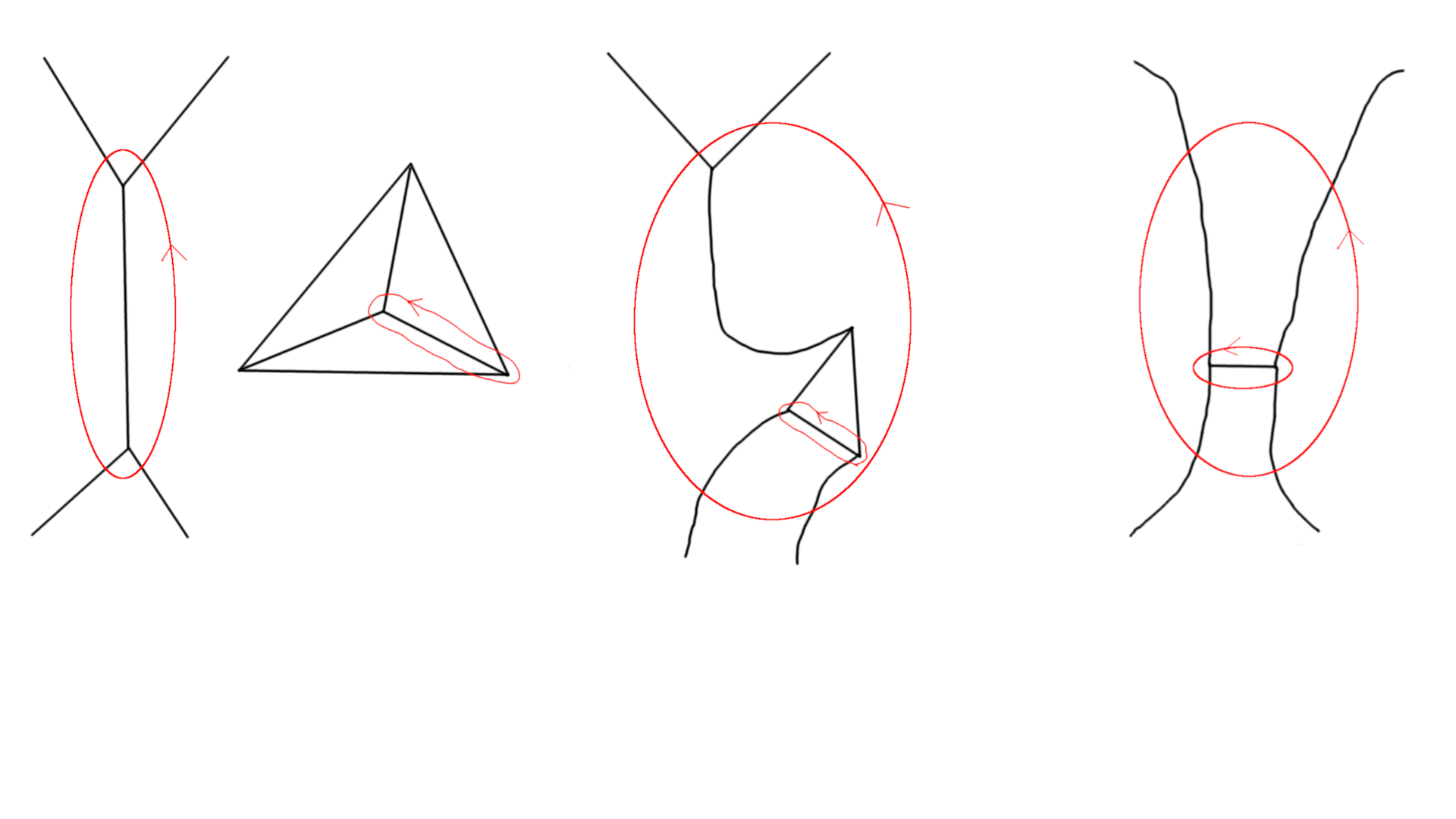}}
    
    \end{picture}
    \caption{Proof of Lemma \ref{fate of e}. Left: the negative end of the cobordism.  Middle: after index zero surgery.   Right: after index one surgery. Note that the two positively oriented lifts of the edges at the negative end of the edges have opposite orientation when oriented to the positive end.
    % On the left the situation at the base. We choose the usual orientation scheme that the lift of an edge is oriented by first pushing it onto the adjacent "upper" sheets and then orienting them with the "outward conormal first" orientation scheme. In the middle picture, we see where these edges end up after gluing the lower vertex of $e$ and the lower left vertex of the tetrahedron graph to one another. Finally, we propagate both of them to the top of the cobordism from Construction \ref{CPG: Cobordism}. We note the larger loop is now negatively oriented as on the pieces parallel to the edge it is on the lower sheets but still the orientation of the loop is as if we used "outward conormal first". In the same line, we note the other loop is still the same.
    }  
    \label{Fig:Propagate-1}
\end{figure}

\begin{lemma} \label{fate of e}
    If we depict the cobordism as a movie of base projections, we see one of the edges $e_T$ of the tetrahedron graph get 
    connect summed with the edge $e$ of $\Gamma$.  
  
    Then the associated cycles $[-e_T] \in H_1(\mathbb{T}_{Cl})$ and 
    $[e] \in H_1(\Lambda)$
    have the same image in
    $$H_1(\mathbb{T}_{Cl}) \to H_1(\widetilde{L}) \leftarrow H_1(\Lambda)$$ 
\end{lemma}
\begin{proof}
    Note the cycles associated to opposite edges of the tetrahedron graph agree up to a shift in orientation.  Now the proof is given in Figure \ref{Fig:Propagate-1}. 
\end{proof}

\subsection{Perfectness}

We now show the  Morse cobordism $L^{Mo}$ associated to $\tilde{L}$ is perfect. 

We fix notation. The positive end of $L^{Mo}$ is $\Lambda_{\Gamma_e}$ and the negative end is $\Lambda_\Gamma \sqcup \T_{Cl}$. By construction the set $J^1(D_{1-\epsilon})$ contains all Reeb chords of $\Lambda_\Gamma$ and of $\Lambda_{\Gamma_e}$ (contained within the initially chosen Darboux chart $\R^5 \subset S^5$) and the cobordism is cylindrical above this set. Thus the set of Reeb chords are in $1:1$ correspondence. Each gives rise to two Reeb chords in $L^{Mo}$: one above the positive end and one above the negative end. For the Reeb chord above the negative end the index coincides with the index as in $\Lambda_\Gamma$ while the index for the positive end is raised by $1$. These will be denoted by $\rho_+$ and $\rho_-$, respectively. 

\begin{theorem} \label{CPG: Curve Count}
    The disk surgery cobordism of Construction \ref{CPG: Cobordism} is isotopic to a perfect disk surgery cobordism. 
\end{theorem}
\begin{proof}
  Let $L^{Mo}$ be the Morse cobordism obtained from the eventually cylindrical cobordism $\tilde{L}$ from the construction. Let $\rho_i$ ($i=1,2$) be two index $1$ Reeb chords of $\Lambda_\Gamma$. Endow $T^*(\R \times [a,b])$ with the Euclidean metric.

    We will show that the space of Morse flow graphs with one positive puncture at $\rho_{1,+}$ and a negative puncture at $\rho_{2,-}$ is empty unless $\rho_{1} = \rho_2$, in which case it contains only the trivial Morse flow tree. In particular, these spaces of Morse flow graphs are transversally cut-out.  Thus, we obtain a perfect disk surgery cobordism via Theorem \ref{thm: ehk flow tree count}.

  By Construction \ref{CPG: Initial choices} $\Lambda_\Gamma$ is basic around all its Reeb chords and by Construction \ref{CPG: Cobordism} the cobordism $\tilde{L}$ is cylindrical above $\bar{D}_{1-\epsilon} \times \R$. So in particular $\tilde{L}$ is basic around all Reeb chords of $\Lambda_\Gamma$. So, we may use Lemma \ref{Lem:Trivial tree} with respect to the euclidean metric on $M \times [e^a,e^b]$ and obtain that the space of Morse flow graphs with unique positive puncture at $\rho_+$ and unique negative puncture at $\rho_-$ contains the trivial Morse flow tree for any Reeb chord $\rho$ of $\Lambda_\Gamma$ and they are transversely cut-out. To obtain the theorem, we only have to rule out rigid Morse flow graphs which have a unique positive puncture at a Reeb chord $\rho_{i,+}$ and a negative puncture at a negative Reeb chord $\rho_{j,-}$ which is not the same as the trivial tree. We will show this by proving that a Morse flow graph having a negative puncture at $\rho_{j,-}$ must have a positive puncture at $\rho'_+$ coming from an index $2$ Reeb chord of $\Lambda_\Gamma$.

   Let $G$ be an interior Morse flow graph with a negative puncture at $\rho_{j,-}$. 
   From Lemma \ref{Lem:Flow graph restrictions} we obtain that the flow graph close to a puncture above $\rho_{j,-}$ must look like a multiple cover of a leaf which is oriented towards the puncture. Either all covers of that edge continue up the trivial flow line ($\subset \rho \times [e^a,e^b]$) to $\rho_{j,+}$ and take it as a positive puncture multiple times or there is an internal vertex $v$ which has an adjacent edge whose $1$-jet lift includes one of the sheets of the lower hemisphere. Recall that by the definition of a Morse flow graph there must be a cyclic order on the edges adjacent to a (non-puncture) vertex $v$. This obeys the rule that an inward pointing edge must be followed (or preceded) by an edge which lifts to the same sheet but is outward pointing. Thus we claim that there must be some edge $e$ pointing to $v$ which has one $1$-jet lift on the upper hemisphere and one lift on the lower hemisphere (See the next paragraph). By Construction \ref{CPG: Initial choices} these gradients are all transverse to $\partial D_{\delta}$ for some $\delta>\epsilon$. Thus as we follow $e$ up its radius will shrink. Each time we meet a vertex of $G$ we find another edge $e'$ which fulfills the same condition. Until we are inside $D_\epsilon$ which by assumption a mixed gradient flow line cannot enter. Now denote by $T$ the maximum value of the projection to $[e^a,e^b]$ any gradient flow line of $\mathcal{G}'$ achieves inside $D_\delta$. This cannot be achieved on the interior of an edge and if it were to be on a non-puncture vertex then we could find another flow line which we can follow upwards and achieve a value $T'>T$ in its interior. This is only possible if $T=e^b$, but by the argument in the proof of Lemma \ref{Lem:Flow graph restrictions},
   this cannot happen for a non-puncture vertex. Thus we have successfully ruled out Morse flow graphs which are not of the desired type.

   It remains to prove the claim that if we have a non-puncture vertex $v$ with an outgoing edge $e$ which is supported by one sheet of the upper hemisphere and one of the lower hemisphere then there must be an ingoing edge $e'$ fulfilling the same restriction (with possibly different upper or lower sheets). Above any point $p$ in $D_{\delta}$ $0<\delta< 1-\epsilon$ we have $4$ sheets which we will denote by $U^1,U^2,L^1,L^2$ where $U^i$ are the two sheets of the upper hemisphere and $L^i$ are the two sheets of the lower hemisphere. Furthermore denote by $U^i_{\pm}$ and $S^i_{\pm}$ the number of incoming directions ($U^i_+/S^i_+$) and the number of outgoing directions ($U^i_-/S^i_-$) at $v$. The restriction that there is a cyclical order on the edges adjacent to any vertex implies that an incoming lift of an edge $e_i$ on some sheet must be followed (or preceded) by an outgoing lift on the same sheet, implies $U^i_+=U^i_-$ and $S^i_+=S^i_-$ for $i=1,2$. Now consider the weighted sum $U_w=U^1_++U^2_+-U^1_--U^2_-$ and similarly $S_w=S^1_++S^2_+-S^1_--S^2_-$ both of these numbers must be $0$. Note however, that edges supported either by both upper or lower sheets doesn't change the sums. However, if we have a mixed outgoing edge this contributes $-1$ to $S_w$ and $1$ to $U_w$. So a mixed incoming edge is necessary.
\end{proof}

\subsection{Filling the Clifford torus}

% In this section, we will reap the benefits of our work from the last section and show that we can prove a skein-theoretic version of the $5$-term relation geometrically. Theoretically, we could use the results of \cite{Scharitzer-Shende} find a Reeb-positive Legendrian $\Lambda$ such that for any filling $L$ we know exactly the skein-element defined by $A(\rho)\Psi_L=0$ and then find a filling such that there is an obvious decomposition of $\Psi_L$ into its components. However, we will find our desired solutions simply by topological and algebraic rather than geometric methods. First, we have to be a bit more precise and fix a local model where we do our calculations:

Consider the Legendrian $\T_{Cl} \subset S^5$ associated
to the tetrahedron graph.  
Each edge of the graph determines a cycle on $\mathbb{T}_{Cl}$; in fact opposite edges give the same cycle.  Let us denote the resulting cycles $\ell, m, n$; one has $\ell + m + n = 0$, and any two give a basis of $H_1(\mathbb{T}_{Cl})$. 

There are three topologically distinct families of 
Lagrangian fillings of $\T_{Cl}$, coming from different smoothings of the 
Harvey-Lawson cone. 
Each of the three fillings contracts exactly one of these cycles.  If a given filling contracts $\ell$,
then in that filling $m = -n$, and one of these is positive with respect to the primitive of the symplectic form.  Moreover, the other two fillings render $\ell$  respectively positive or negative with respect to the primitive of the symplectic form on $\C^3$. 

These fillings may be glued {\em topologically} to the Clifford torus end of a cobordism coming from Construction \ref{CPG: Cobordism}.  We refer to this construction as a {\em topological Harvey-Lawson filling}.

\begin{lemma}  \label{Lem: CPG topological filling}
    Consider 
    the three possible topological Harvey-Lawson fillings of the 
    Clifford torus.  These each contract a different one of the three cycles  $\{\ell, m, n\} \in H_1(\T_{Cl})$; denote the contracted 
    cycle $\gamma$.  

    Let $e_T$ be the edge from Lemma \ref{fate of e}.
    Then the intersection pairing 
    $(\gamma, e_T)$ takes the three values $+1, 0, -1$, one on each filling.  We refer to the corresponding fillings
    as positive, forbidden, and negative, respectively.  

    Moreover, on the positive and negative fillings, the topological result of gluing in the Harvey-Lawson filling 
    is a trivial cylinder; this being false for the forbidden filling.  
\end{lemma}
\begin{proof}
    Proven similarly to Lemma \ref{Topological filling possible}. In that proof we require that the index $2$ Morse singularity of the cobordism $L_e$ and the index $1$ Morse singularity in $\mathbf{T}_{Cl}$ are cancellable. This is only possible if (after some choice of metric) their stable and unstable manifolds have one algebraic intersection. As the unstable manifold of the index $2$ manifold is homotopic to the $e+e_T$ cycle and the stable manifold is homotopic to the contractible cycle. This is impossible if $e_T$ is the contractible cycle. If $e_T$ is not the contractible cycle it is some generator of homology and thus has one algebraic intersection with the contractible cycle.
\end{proof}

\subsection{Graph mutation and skein-valued cluster transformations}

\begin{theorem} \label{cubic planar graph cluster}
    Let $\Gamma$ be a cubic planar graph, and $e$ a non-loop edge
    of $\Gamma$.  We write $e_+$ for the corresponding oriented simple closed curve, and $e_-$ for the curve with opposite orientation.  Then we have the following equations, holding, respectively,
    in completions of the skein in the directions of 
    $e_+$ and $e_-$: 
    $$ \mathbf{A}(\Lambda_{\Gamma_e}) \mathbf{E}(e_+) = 
    \mathbf{E}(e_+)  \mathbf{A}(\Lambda_\Gamma) $$
    $$ \mathbf{A}(\Lambda_{\Gamma_e}) \mathbf{E}(e_-) = 
    \mathbf{E}(e_-)  \mathbf{A}(\Lambda_\Gamma) $$
\end{theorem}
\begin{proof}
    In Theorem \ref{CPG: Curve Count}, we have constructed perfect cobordisms 
    for the Legendrian disk surgeries 
    associated to mutations of cubic planar graphs. 
    Lemma 7.7 provides each with two standard trivializations, one inducing each orientation of the loop corresponding to $e$.  Now the result follows from Theorem \ref{cluster transformation from perfect cobordism}.
\end{proof}

\section{The necklace graph}

The ``necklace graph'' $\Lambda_{Neck, g}$ 
(see Fig. \ref{Fig:Necklace}) has the virtue that its associated 
Legendrian has an exact filling which is topologically the genus $g$ handlebody
 \cite[Theorem 4.10.1, Theorem 4.17]{Casals-Zaslow}.  

Several of the main results of 
\cite{SSZ} concern properties of sequences of q-cluster transformations associated to 
sequences of graph mutations beginning with the necklace graph, and having a certain positivity property.  
We will give geometric counterparts of said results here, 
and correspondingly prove that 
certain explicit skein-valued lifts of the formulas written in \cite{SSZ} are in fact skein-valued curve counts.

%constructed as a composition of a disk filling the unknot, Legendrian isotopies, and a sequence of $g$-many 0-surgeries \cite[Theorem 4.10.1, Theorem 4.17]{Casals-Zaslow}, 
%and is topologically the genus $g$ handlebody.  

\begin{figure}
    \begin{picture}(50,125)
    \put(-150,0){\includegraphics[width=15cm, trim=0cm 4cm 0cm 0cm, clip]{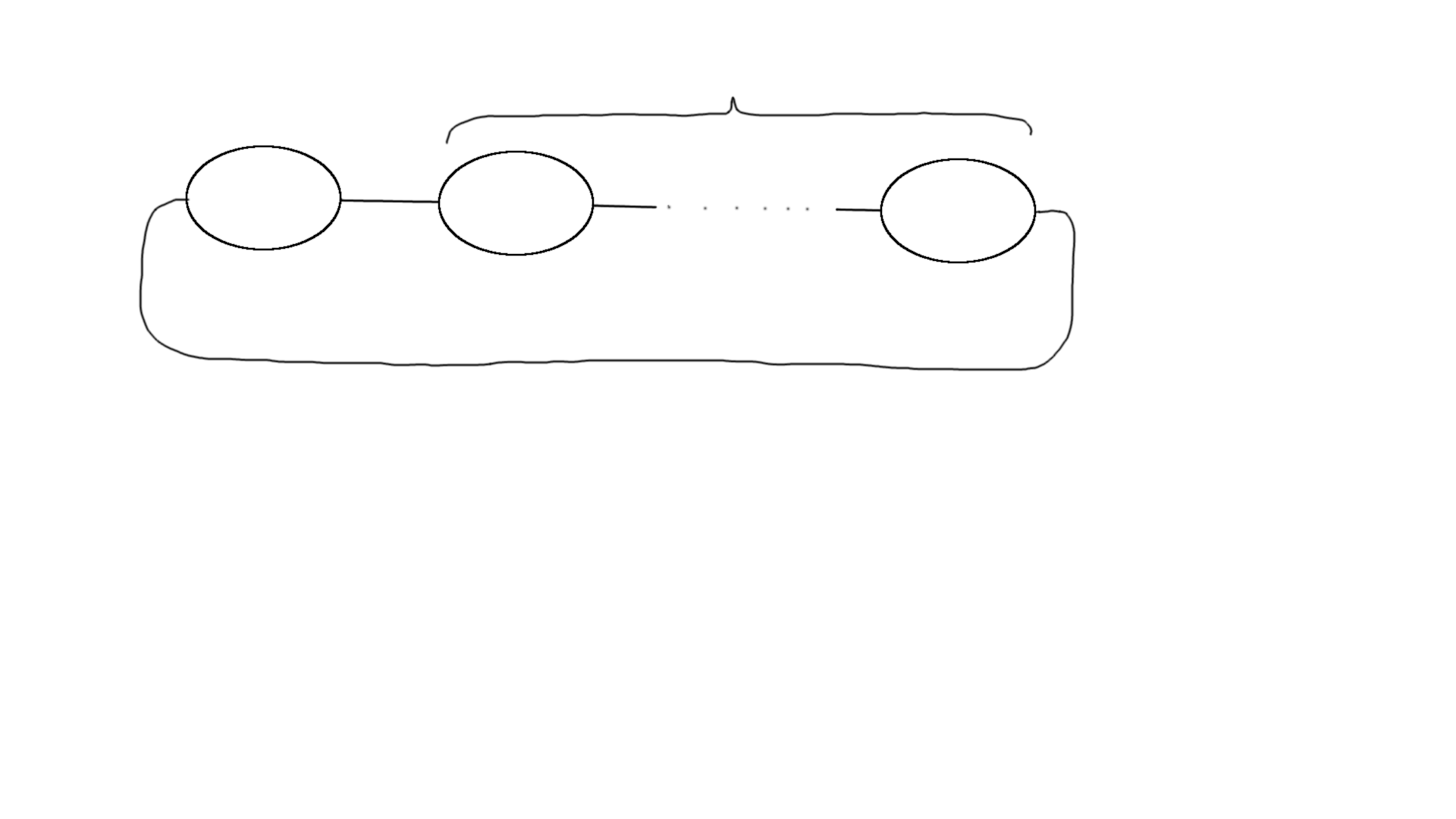}}
    \put(63,112){$g$}
    
    \end{picture}
    \caption{The necklace graph of genus $g$.}   
    \label{Fig:Necklace}  
\end{figure}

For the remainder of the article, we invert the following quantities associated to pairs $(\lambda,\mu)$ of Young diagrams:
\begin{equation} \label{clambdamu} c_{\lambda,\mu} = a^{-1}(q^{-1/2} - q^{1/2})\sum_{\square \in \lambda}q^{-cn(\square)} - a(q^{-1/2} - q^{1/2})\sum_{\square \in \mu} q^{cn(\square)} \quad \quad (\lambda,\mu) \neq (\emptyset,\emptyset)
\end{equation}
The latter are the eigenvalues of the difference of the diagonal operator and the operator given by taking disjoint union with an unknot, see \cite{Gilmer-Zhong-connectsum}, and we will need to divide by them in the proof of Corollary \ref{cor injective}.

\subsection{Lemma on skeins of handlebodies}

Let $L_g$ be a handlebody of genus $g$, with boundary 
$\Sigma_g$.  
Fix a disk $D \subset L_g$.
Denote by $Sk_k(L_g)$ the submodule of $Sk(L_g)$ generated by framed links whose geometric intersection number with $D$ is $\leq k$.  
Evidently $[\bigcirc] - [\partial D]$ preserves $Sk_k(L_g)$.  
We have the following result of Gilmer and Zhong: 

\begin{lemma} \cite[Lemma 2.1]{Gilmer-Zhong-connectsum} \label{Gilmer-Zhong-lemma}
    Any  $L \in Sk_k$ can be expanded in a finite sum $\Psi = \sum_{i=1}^n a_i [L_i]$, where the $[L_i] \in Sk_k$ satisfy 
    $([\bigcirc] - [\partial D])[L_i] = c_i [L_i] \pmod {Sk_{k-1}}$, and
    the $c_i$ are among the $c_{\lambda,\mu}$ of \eqref{clambdamu}. 
\end{lemma}

\begin{corollary} \label{cor injective}
    The endomorphism $([\bigcirc] - [\partial D])$ is injective on
    $Sk_k/Sk_{k-1}$. 
\end{corollary}
\begin{proof}
    Fix $L \in Sk_k \setminus Sk_{k-1}$; expand $L$ as per Lemma \ref{Gilmer-Zhong-lemma}.  
        Let $M \subset Sk_k / Sk_{k-1}$ be the submodule generated by the $[L_i]$
    over the coefficient ring.     Then $([\bigcirc] - [\partial D])$ is a surjective endomorphism of $M$: indeed, for an element $\sum a_i[L_i]$, a preimage is $\sum \frac{a_i}{c_i}[L_i]$.  (Recall we have inverted the $c_{\lambda, \mu}$.)  We recall that 
    any surjective endomorphism of a finitely generated module is an isomorphism \cite{Vasconcelos}.  It follows that $L$ is not in the kernel
    of $([\bigcirc] - [\partial D])$. 
\end{proof}

\begin{corollary}
    \label{skein meridians cut}
    Fix a disk $D \subset L_g$.  If $L \in Sk(L_g)$ satisfies
    $([\bigcirc] - [\partial D]) L = 0$, then $L$ is the image of
    an element of $Sk(L_g \setminus D)$.
\end{corollary} 
\begin{proof}
    Take $k$ minimal so that $L \in Sk_k$.  
     $([\bigcirc] - [\partial D]) \Psi = 0$ implies $\Psi=0 \in Sk_k/Sk_{k-1}$ or in other words $\Psi \in Sk_{k-1}$, which is a contradiction (unless $k=0$). 
\end{proof}
\begin{remark}
    The above corollaries are essentially due to Gilmer and Zhong, who 
    however only state them in the case that the disk is separating 
    \cite[Lemma 2.2]{Gilmer-Zhong-connectsum}. The fact that
    their arguments work in general was earlier observed in \cite[Lemma 7.2]{HSZ}.
\end{remark}

Fix $\lambda \in H^1(L_g)$.  We write $Sk^{\ge 0}(L_g)$
for the elements on which $\lambda \ge 0$, and $Sk^{>0}(L_g)$ for the strictly positive elements.  Recall that, a priori, curve counts land in a completion 
$\widehat{Sk}^{\geq 0}_\lambda (L_g)$.  This completion has, by definition, 
the property that the contribution of curves of bounded $\int \lambda$ to the coefficient of some fixed $q^a$ can be written using only finitely many knots.  (The reason curve counts land here is Gromov compactness.)  

We do not know how to generalize the above corollaries to this completion.  
However, we can generalize them to the smaller subset 
$\widetilde{Sk}^{\geq 0}_\lambda(L_g)\subset \widehat{Sk}^{\geq 0}_\lambda (L_g)$ with the property that the contribution of any fixed $H_1(L_g)$ degree can be written using only
finitely many knots:

\begin{corollary} \label{skein meridians cut completion}
 If $\Psi \in 1 + \widetilde{Sk}^{> 0}_\lambda(L_g)$ and $([\bigcirc]-\partial D) \Psi = 0$ then $\Psi \in 1 + \widetilde{Sk}^{> 0}_\lambda(L_g \setminus D)$.
\end{corollary}

\begin{proof}
    Decompose $\Psi = \sum_{c \in H_1} \Psi_c$ into homology classes and apply Lemma \ref{skein meridians cut} to each.
\end{proof}

We do not know a priori that curve counts live in
$\widetilde{Sk}^{\geq 0}_\lambda(L_g)$: it could be that curves of increasing genera have boundaries in the same $H^1(L_g)$ class, but tracing out ever more complicated knots. 
The counts of interest here do however have this property: 

\begin{lemma} \label{Lift lives in subcompletion}
    Any finite product of elements $ \mathbf{E}(e_i)$, where $\lambda(e_i)>0$, is in $1+ \widetilde{Sk}^{> 0}_\lambda(L_g)$. 
\end{lemma}

\begin{proof}
    Let us check first that this holds for a single $\mathbf{E}(e)$.  By \cite[Theorem 1.1]{unknot}, $\mathbf{E}(e)$ can be written in the following form $\sum_\lambda c_\lambda W_\lambda $ where $c_\lambda$ is some element of $\widetilde{R}$ and $W_\lambda$ is a link whose homology class is given by $|\lambda|[e]$. Thus, only finitely many elements of the above sum belong to any given homology class. 
    
    It is not quite immediate to pass to products, because $\widetilde{Sk}^{\ge 0}_\lambda(L_g)$ is not closed under products.  However, the further subset
    where $H_1(L_g)$ classes are constrained to any strictly convex cone inside $\lambda \ge 0$ is closed under products.  For any finite product of the $\mathbf{E}(e_i)$, we may find such a convex cone containing all $e_i$ and compute there. 
\end{proof}

\begin{corollary}
    If $\Psi = \mathbf{E}(e_1) \ldots \mathbf{E}(e_n)$ for $\lambda(e_i) > 0$, and $([\bigcirc]-\partial D)\Psi=0$ for some $D$, then $\Psi \in 1+ \widetilde{Sk}^{> 0}_\lambda(L_g \setminus D)$.
\end{corollary}

In the other direction, we have the following much simpler result:  

\begin{lemma} \label{annihilators must bound}
    Let $c$ be a simple closed curve on $\Sigma_g$.  Let $\Psi \in 1+ \widehat{Sk}_\lambda^{>0}(L_g)$.  
    If $(\bigcirc - c) \Psi = 0$, then $c$ must bound in $L_g$. 
\end{lemma}
\begin{proof}
    If $\lambda(c) < 0$, then the leading term of $(\bigcirc - c) \Psi$ is $c \ne 0$.  If $\lambda(c) > 0$ then the leading term is $\bigcirc \ne 0$.  Thus 
    $\lambda(c) =0$, and the leading term is $\bigcirc - c$.  As this must vanish  in $Sk(L_g)$, we conclude $c$ bounds.  
\end{proof}

\subsection{Curve counts and the necklace graph Legendrian}

\begin{figure}
    \begin{picture}(20,100)
    \put(-150,-70){\includegraphics[width=14cm, trim=0cm 3cm 4cm 0.5cm, clip]{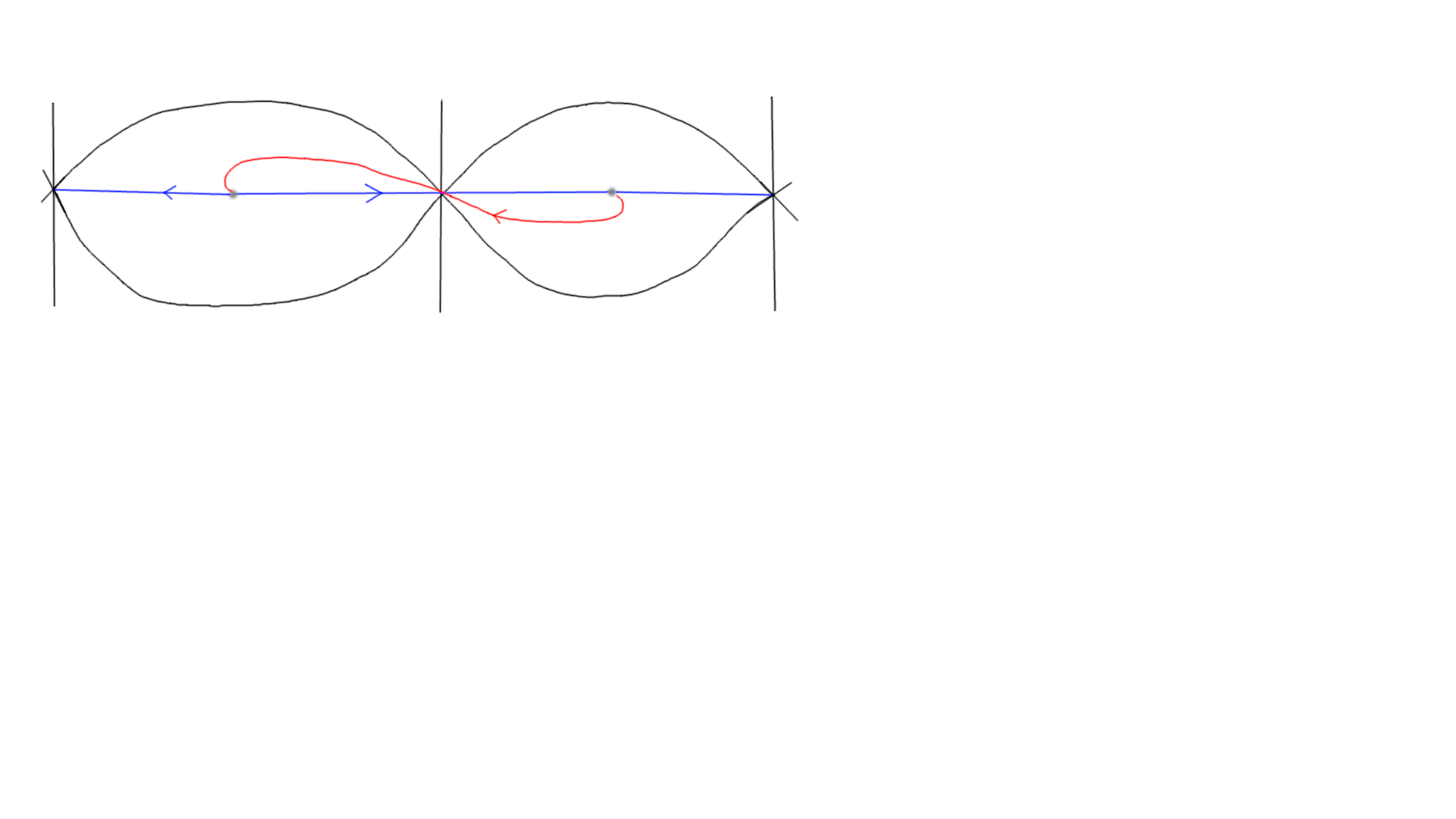}}

    \end{picture}
    \caption{We describe the element $\mathbf{A}(f)$,
    when $f$ is a bigon face of a graph $\Gamma$,
    by specializing \cite[Figure 7]{Scharitzer-Shende}. 
    What is drawn is the region of $\Lambda_\Gamma$ above a neighborhood of $f$; the outer two $6$-valent vertices are identified with one another. In black are the lifts of the edges of the graph $\Gamma$, in blue the tangles of $\mathbf{A}(f)$,   in red a choice of capping path we used in \cite[Figure 7]{Scharitzer-Shende}.}  
    \label{Fig:Pearl}
\end{figure}

We recall a very special case of the main theorem
of our previous article \cite{Scharitzer-Shende}. For the Reeb chord $\rho$ associated to a bigon of $\Lambda_{Neck, g}$, 
and $E_i$ either edge of the face, one has
\begin{equation} \label{necklace A formula}
    \mathbf{A}(\rho) = a^{-1} \bigcirc+[E_i].
\end{equation}
See Figure \ref{Fig:Pearl}.

\begin{remark} \label{Rem: -a^{-1} contribution unknot}
    The explicit formula 
    for $\mathbf{A}(\rho)$
    may be confusing: on the 
    one hand, 
    the unknot and $[E_i]$
    become isotopic in the filling, and on the other, 
    we have claimed the image
    of $\mathbf{A}(\rho)$ is zero.  The point is that we are working in a skein with sign/linking lines.  What necessarily happens in the filling is that  $[E_i]$ must cross some such lines  as it is contracted, introducing an appropriate power of $-a$, in this case $-a^{-1}$. 
\end{remark}

\begin{figure}
    \begin{picture}(20,100)
    \put(-175,0){\includegraphics[width=12cm, trim=0cm 3cm 0cm 1cm, clip]{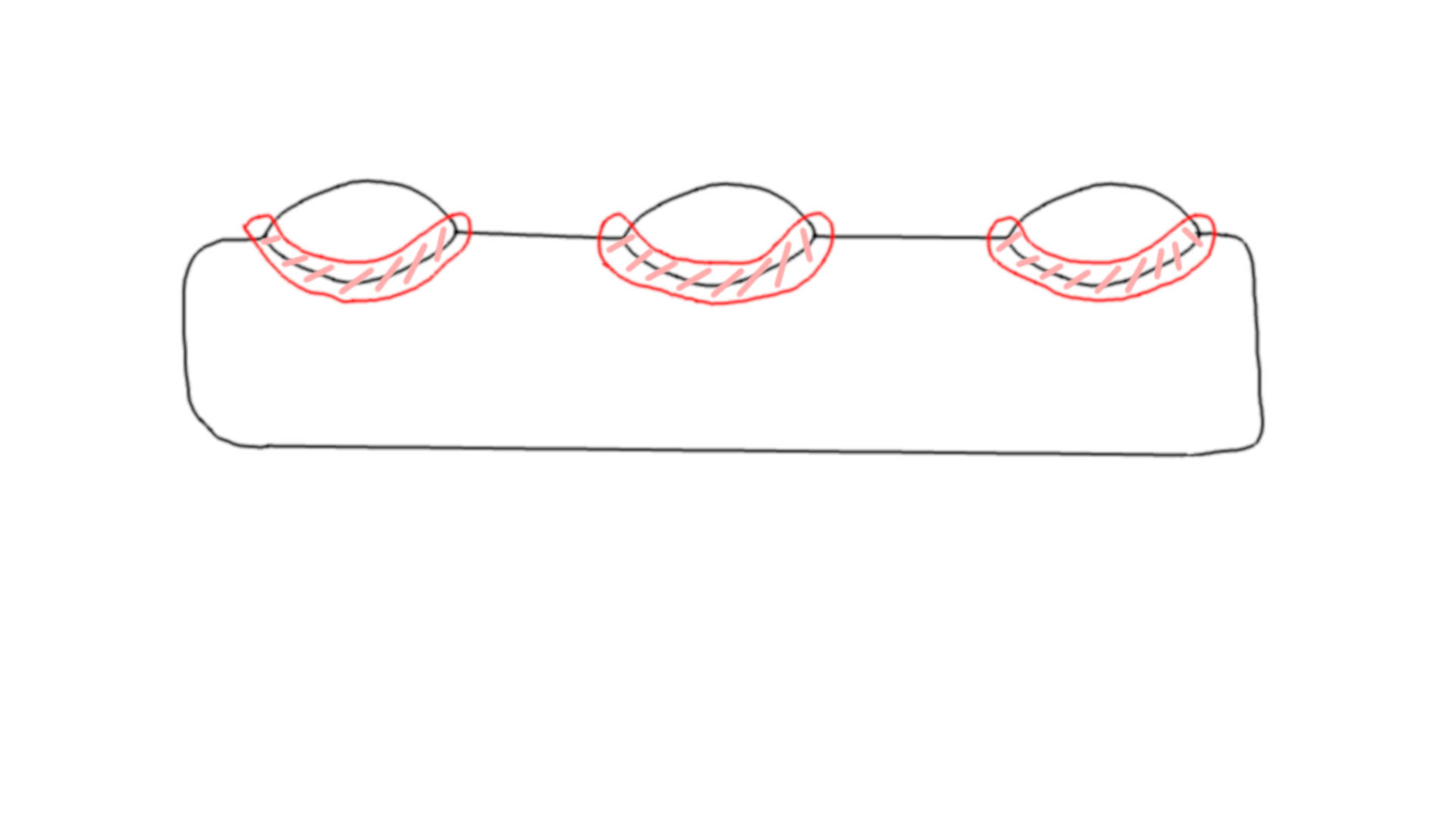}}

    \end{picture}
    \caption{The necklace graph $\Lambda_{Neck, 2}$, and a collection of disks whose lifts to $\Lambda_{\Gamma_{Neck, 2}}$ give disconnecting cylinders. %Note that the red lines cross the black lines an even number of times, so they lift to two different closed curves. These two curves bound an annulus which is the branched double cover of the disk bound by the base projection of the curve. On the right after removing those annuli the double cover is an unbranched double cover over the sphere minus these disks.
    }  
    \label{Fig:Pearl-Removal}
\end{figure}

\begin{lemma} \label{necklace separate}
    Choose one edge of the bigons of the necklace graph.  The corresponding cycles
    separate $\Lambda_{\Gamma_{Neck, g}}$ into two punctured disks.
\end{lemma}
\begin{proof}
    A small thickening of these cycles branch-cover a small neighborhood of the corresponding edges.  In the complement of these neighborhoods, the map $\Lambda_{\Gamma_{Neck, g}} \to S^2$ is an unbranched cover, which one sees by inspection is disconnected. 
    See Figure \ref{Fig:Pearl-Removal}.  
\end{proof}

\begin{corollary} \label{5-term:Necklace uniqueness}
    Let $L$ be any genus $g$ handlebody filling $\Lambda_{Neck, g}$.  Fix any 
    $\lambda \in H_1(L)$.  
    If $\Psi \in 1 + \widetilde{Sk}_\lambda^{> 0}(L)$ and $\mathbf{A}(\Lambda_{Neck, g}) \Psi = 0$, then $\Psi = 1$.   
\end{corollary} 
\begin{proof}
We pull the sign/linking lines past the $\mathbf{A}(\Lambda_{Neck, g})$ so that the expression \eqref{necklace A formula}  for the generating elements becomes the more familiar $\bigcirc - [E_i]$.  By Lemma \ref{annihilators must bound}, the $E_i$  appearing in the formula 
\eqref{necklace A formula} bound disks
in the filling.  
Since, per Lemma \ref{necklace separate}, these cycles separate $\Lambda_{Neck,g}$ into two punctured disks, we may ensure the disks in the filling correspondingly separate the filling into two balls.   

Iteratively applying Corollary \ref{skein meridians cut}
implies that image of $\Psi$ in $L_g$ is in the skein of a ball, hence is a scalar.  
Finally, the leading term of $\Psi$ is the scalar $1$, so we are done.
\end{proof}

\begin{theorem} \label{uniqueness} 
    Let $S = (C_1, \ldots, C_n)$ be a composable sequence of cobordisms, with positive end  $\Lambda$ and negative end $\Lambda_{Neck, g}$.  
    Let $L$ be the topological filling of $\Lambda$ obtained from the standard filling $L_g$ of $\Lambda_{Neck, g}$ by the smooth identification 
    $\Lambda \cong \Lambda_{Neck, g}$ using the trivialization of $S$. 
    Assume the controlling $\lambda \in H_1(\Lambda)$ vanishes on meridians of $\Lambda_{Neck, g}$, so extends to $L$. 

    Then  $\mathcal{A}(\Lambda) \Psi = 0 \in \widehat{Sk}_\lambda(L)$
    has a unique solution in the subcompletion $\Psi \in 1 + \widetilde{Sk}_\lambda^{>0}(L)$, namely the image of $\Psi_S$ in $\widehat{Sk}_\lambda(L)$.

    %In particular, in case $(L, \lambda)$ is diffeo rel boundary to some geometric filling $L_{geom} \subset \R^6$, and $\lambda$ is the restriction of the symplectic primitive, then 
    %$\Psi_S$ is the skein-valued curve count of curves ending on $L_{geom}$.
\end{theorem}
\begin{proof}
    By Corollary \ref{necklace annihilates}, $\Psi_S$ is indeed a solution.  Suppose given some other solution, $\Psi \in 1 + \widetilde{Sk}_\lambda^{>0}(L)$. 
    We multiply $\mathbf{A}(\Lambda)\Psi = 0$ by
    $\Psi_S^{-1}$ (note the inverse exists as usual for something in a completion with invertible leading term).  Then we have 
    $$0 = \Psi_S^{-1} \mathbf{A}(\Lambda)\Psi = \mathbf{A}(\Lambda_{Neck, g}) \Psi_S^{-1} \Psi $$
    Now $\Psi$ has leading term $1$ by hypothesis, and $\Psi_S$ is the product of elements with leading term $1$, so also 
    $\Psi_S^{-1} \Psi$ has leading term $1$.  By 
    Lemma \ref{5-term:Necklace uniqueness}, $\Psi_S^{-1} \Psi = 1$.\footnote{Note the product means: compose the operations of pushing in the given element from the boundary, not: push both operations in from the boundary in different copies of $L$ and then multiply using the product structure on the skein of a handlebody.  We never use the product structure on the skein of a handlebody.}
    Applying $\Psi_S$ on both sides, we have $\Psi_S = \Psi$. 
\end{proof}

\section{Composable cobordisms from admissible mutations}
In order to apply our results on composable cobordisms
(Corollary \ref{composable cobordism formula}, Theorem \ref{uniqueness}), we need to check when sequences of cobordisms are composable.  

In particular, suppose given some cubic planar graph $\Gamma = \Gamma_0$,  putative controlling element $\lambda \in H_1(\Lambda_\Gamma)$, sequence of edges $e_1, e_2, \ldots, e_n$ and sequence of signs $\epsilon_1, \epsilon_2, \ldots, \epsilon_n \in \{\pm\}$.  
We write $\Gamma_i$ for the graph obtained by mutating $\Gamma_{i-1}$ at $e_i$.  We write $C_i$ for the corresponding disk surgery cobordism of sign $\epsilon_i$ (recall the sign prescription in Lemma \ref{Lem: CPG topological filling}).  

We would like an algorithmic way of deciding
when the sequence $(C_1, C_2, \ldots, C_n)$ is a composable sequence controlled by $\lambda$.  In other words, we should check when the class of the circle $[e_{i-1}] \in H_1(\Gamma_i)$ are all positive for $\lambda$ under the identifications $ H_1(\Gamma) = H_1(\Gamma_i)$ induced by the cobordism trivializations.  

The main difficulty in doing so is to give an expression for how the homology transforms under the cobordism trivializations.  Fortunately, this was already done in \cite{SSZ}.  We recall their result (which however was proven for a somewhat different presentation of the cobordisms, so we provide a proof in our context).

\begin{lemma}\cite{SSZ}
\label{homology mutation}
    Let $\Lambda_\Gamma$ be a CPG Legendrian and $e \in \Gamma$ an edge. Fix a positive topological Harvey-Lawson filling of the cobordism $L$ associated to the edge $e$. Then the isomorphism 
    $H_1(\Lambda_{\Gamma_e}) \xrightarrow{\sim}  H_1(\Lambda_\Gamma) $ is given by the following formulas
    for homology classes of loops associated to edges of the respective graphs, where $\tilde{e}'$ denotes the edge corresponding to $e'$ after the edge flip at $e$:
    \begin{align*}
        \tilde{e}' \mapsto & e' & \text{$e'$ not adjacent to $e$} \\
        \tilde{e} \mapsto & -e
    \end{align*}
    Let $e_1,\dots e_4$ be the edges adjacent to $e$ such that, in the counterclockwise cyclic order at vertices, $e_2$ and $e_4$ precede $e$ and $e_1,e_3$ come after $e$, see Figure \ref{Fig:Homology maps}. Then the induced map satisfies also:
    \begin{align*}
        \tilde{e}_i \mapsto &e_i+e && i=1,3 \\
        \tilde{e}_i \mapsto &e_i && i=2,4
    \end{align*}
    The corresponding map for negative mutation is:
    \begin{align*}
        \tilde{e}_i \mapsto &e_i && i=1,3 \\
        \tilde{e}_i \mapsto &e_i+e && i=2,4
    \end{align*}
    
    Figure \ref{Fig:Homology maps} also depicts this map for the edges around $e$.
\end{lemma}

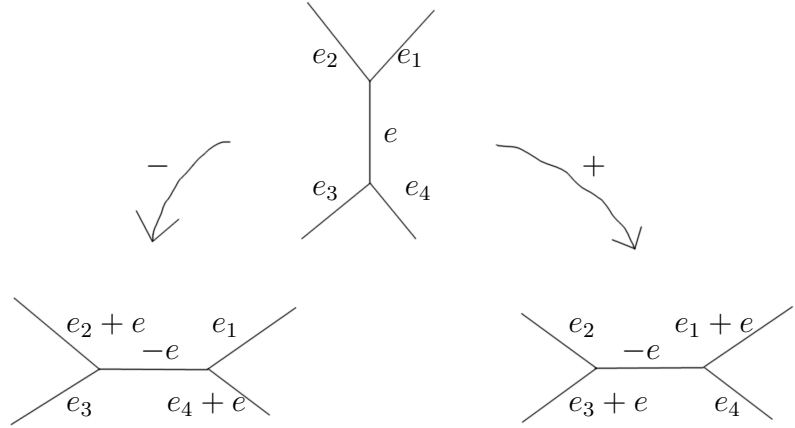
\begin{figure}
    \begin{picture}(20,185)
    \put(-150,0){\includegraphics[width=12cm, trim=0cm 0cm 0cm 0cm, clip]{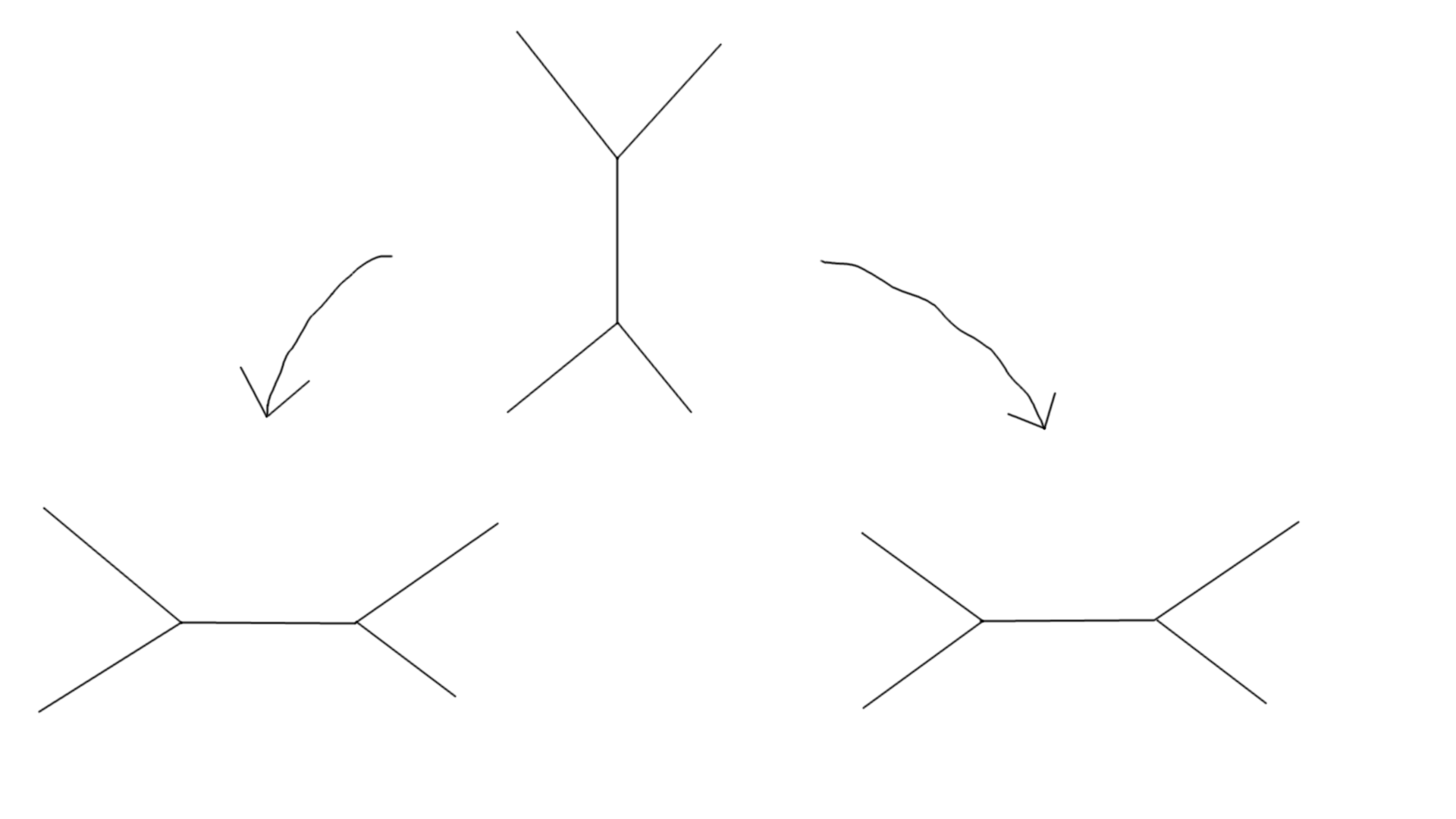}}
    \put(-90,120){$-$}
    \put(75,120){$+$}
    
    \put(0,132){$e$}
    \put(5,162){$e_1$}
    \put(-27,162){$e_2$}
    \put(-27,112){$e_3$}
    \put(8,112){$e_4$}

    \put(-92,50){$-e$}
    \put(-65,60){$e_1$}
    \put(-120,60){$e_2+e$}
    \put(-120,30){$e_3$}
    \put(-82,30){$e_4+e$}
    
    \put(90,50){$-e$}
    \put(110,60){$e_1+e$}
    \put(70,60){$e_2$}
    \put(70,30){$e_3+e$}
    \put(125,30){$e_4$}

    \end{picture}
    \caption{
    %The homology classes of the edges of the graph $\Lambda_{\Gamma_e}$ expressed in the homology of $\Lambda_\Gamma$ depending on whether one mutates positively or negatively.
    Positive and negative mutation of edge labels, per \cite{SSZ}. 
    }  
    \label{Fig:Homology maps}
\end{figure}

\begin{figure}
    \begin{picture}(20,155)
    \put(-150,-30){\includegraphics[width=12cm, trim=0cm 0cm 0cm 0cm, clip]{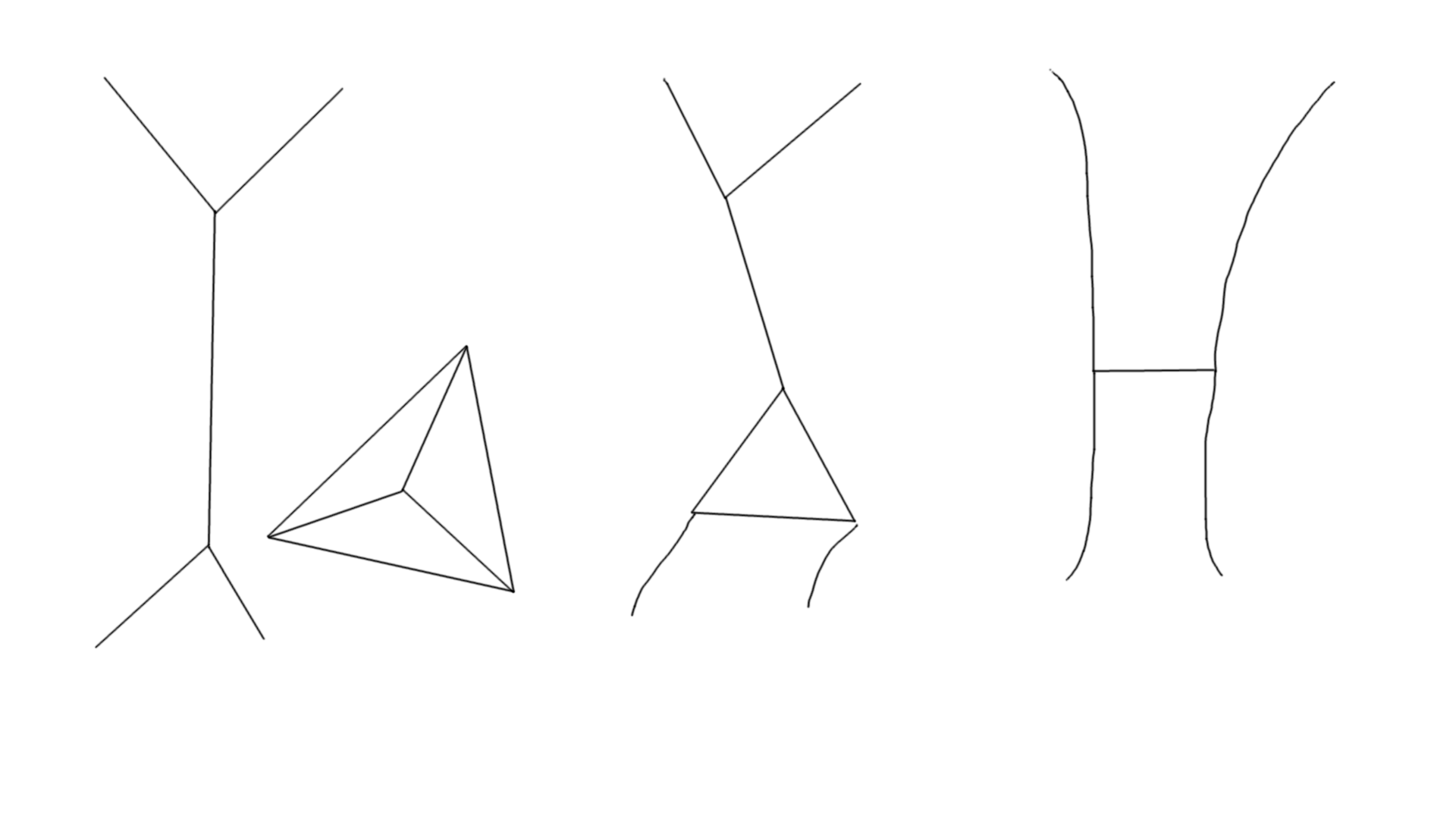}}

    \put(-95,70){$e$}
    \put(-88,120){$e_1$}
    \put(-120,120){$e_2$}
    \put(-125,25){$e_3$}
    \put(-90,25){$e_4$}

    \put(-70,33){$0$}
    \put(-75,22){$U$}
    \put(-85,55){$-U$}
    \put(-30,45){$0$}
    \put(-48,55){$U$}
    \put(-50,40){$-U$}
    
    \put(30,90){$e-U$}
    \put(35,120){$e_1$}
    \put(0,120){$e_2$}
    \put(-15,35){$e_3$}
    \put(50,30){$e_4+U$}
    \put(15,55){$U$}
    \put(42,55){$0$}
    \put(20,30){$-U$}
    
    \put(110,55){$-e$}
    \put(155,120){$e_1$}
    \put(70,120){$e_2+e$}
    \put(85,30){$e_3$}
    \put(145,30){$e_4+e$}

    \end{picture}
    \caption{On the left a neighborhood of the edge $e$ and a tetrahedron graph with a choice of basis coming from the chosen smoothing such that $V=0$. In the middle after identifying the lower edge of $\Gamma$ and the lower left edge of the tetrahedron graph. On the right the result after gluing in the disk which sets $e-U=0$.}  
    \label{Fig:Homology-map-1}
\end{figure}
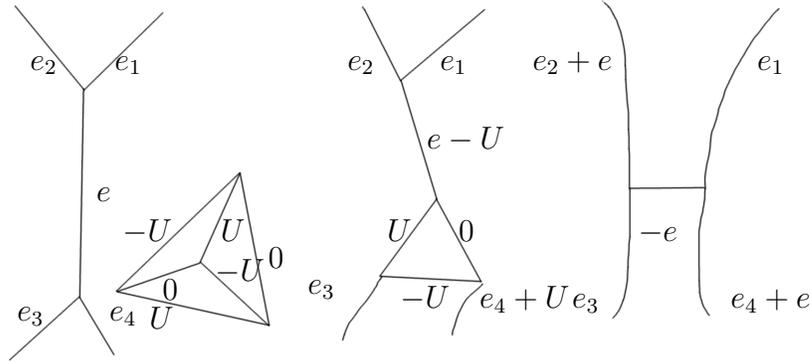

\begin{proof}

    We will do the calculation in the case of negative mutation. The case of positive mutation is proven similarly.
    
    Consider the first stage of the Cobordism from Construction \ref{CPG: Cobordism}, i.e. the  ``Clifford sum'' cobordism. The cobordism is a cylinder in a neighborhood of all edges not adjacent to the two vertices which merge. As we chose negative mutation the gluing can be done as indicated in the middle of Figure \ref{Fig:Homology-map-1}. The other cases of negative mutation are proven similarly. As this part of the cobordism only concerns a neighborhood of the two merging vertices, we record that this cobordism maps all other edges by the identity. The same is true for the remaining edges of the tetrahedron graph Legendrian. Now, recall that each face must fulfill the condition that the product over all edges adjacent to a face must vanish. So we obtain the equation for the remaining $3$ edges:

    \begin{align*}
        (\Sigma_{e' \in F,e'\neq e_{i},e_j} e')+e_i+e_j & =0 & i,j \in \{0,3,4\} &,  i \neq j \\
        (\Sigma_{e' \in F,e'\neq e_{i},e_j} e' )+\phi(e_i)+\phi(e_j)+e_{i,j} & =0 & i,j \in \{0,3,4\} &,  i \neq j 
    \end{align*}

where $e=:e_0$, $e_{3,4}=-U$, $e_{0,3}=U$ and $e_{0,4}=0$. $i,j$ are considered to be unordered. Setting the corresponding pairs equal, we obtain the following equation:

\begin{align*}
    \phi(e_0)+\phi(e_3)+U &= e_0+e_3 \\ 
    \phi(e_3)+\phi(e_4)-U &= e_3+e_4 \\
    \phi(e_4)+\phi(e_0) &= e_4+e_0
\end{align*}

A straightforward verification shows:

\begin{align*}
    \phi(e_0)&=e_0-U \\
    \phi(e_3)&=e_3 \\
    \phi(e_4)&=e_4+U
\end{align*}

Next, we consider the result of the $1$-surgery which corresponds to gluing in a $2$-disk. Again all edges not adjacent to it just propagate upwards, so the homology map maps them by the identity. Furthermore, it kills the $\phi(e)$ edge which enforces on the homology level the equation $e_0-U=0$. So the edge labeled $-U$ from $\Lambda_{\Gamma_v}$ can now be labeled $-e$. It remains to label the two remaining edges. To do so, we consider $e_1+\phi(e_0)+U$ and $e_2+\phi(e_0)+U$ which is represented by a loop which goes from one end of the triangle around $\phi(e_0)$ to the end of either $e_1$ or $e_2$. As can be seen in Figure \ref{Fig:Homology-map-1}. These avoid the edge which vanishes by the surgery. So finally $\phi$ can be modified by: $\phi(e_1)=e_1$,  $\phi(e_2) = e_2 + e$ and $\phi(e_0)=-e_0$.
\end{proof}

For comparison with \cite{SSZ}, let us recall their definition of `admissible sequence of mutations'.  

\begin{definition}(\cite{SSZ})
    Let $(U_j, V_j)$ be the basis for $H_1(\Gamma_{Neck})$ in which the 
    classes of the cycles associated to the edges
    are as depicted in Figure \ref{Fig:Necklace-basis}.
    We will call this the standard basis for the necklace graph Legendrian. 
    
    Let $(e_1,\dots e_n)$ be a sequence of edges and $(\epsilon_1,\dots,\epsilon_n)$ a sequence in $\{1,-1\}$.  Consider the sequence $\Gamma_0=\Gamma_{Neck,g}$,  $\Gamma_i=(\Gamma_{i-1})_{e_i}$.  This is said to be {\em admissible}
    if the homology class of $\epsilon_i e_i$ written in the basis $(U_j,V_j)$ has no negative coefficients in the $U_j$, and at least one of said coefficients is positive.  
\end{definition}

\begin{remark}
The `seeds' of \cite{SSZ} are q-lifts of the choice of homology basis, and their transformation relation \cite[Figure 4.2.2., page 20/21]{SSZ}
is compatible with Lemma \ref{homology mutation}.     
\end{remark}

We translate this as follows: 

\begin{lemma} \label{admissible implies composable}
     Let $\lambda \in H^1(\Lambda_{\Gamma_{Neck}})$
     vanish on the $V_j$ and be positive on the $U_i$. 
    If a sequence $(e_1,\dots,e_n)$ and $(\epsilon_1,\dots,\epsilon_n)$ defines an admissible sequence of mutations then
    the corresponding sequence of disk surgery cobordisms is composable and controlled by $\lambda$. 
    Conversely, if the sequence is composable for all such $\lambda$, then it is admissible. 
\end{lemma}
\begin{proof}
     Consider the induced $1$-form $\lambda'$ on $\mathbf{T}$ which is given by $\lambda(e_i + e_T)=0$ and thus $\lambda'(e_T)=-\lambda(e_i)$ where $e_T$ is the edge from Lemma \ref{fate of e}. Note that in the case of positive mutation, see Figure \ref{Fig:Homology-map-1}, the edge $e_T$ succeeds (in the anticlockwise cyclical order) the contractible edge $\gamma$. Thus by the intersection pairing for their lifts $(\gamma,e_T)=-1$ and thus $(\gamma,\cdot)=-\lambda'$ which by Lemma \ref{Lem: CPG topological filling} implies that we can positively fill the associated cobordism. So if $\lambda$ evaluates positively on $e_i$ we can compose the positive mutation cobordism. One proves the case for negative mutation similarly. 
\end{proof}

\begin{corollary} \label{admissible sequences determined}
        Let $\Gamma$ be a cubic planar graph which can be obtained from $\Gamma_{Neck, g}$ by a sequence of admissible mutations.  Let $L$ be the handlebody filling of $\Gamma$ obtained by composing the disk surgery cobordisms with the standard filling of 
    $\Gamma_{Neck, g}$.  Then for appropriate $\lambda$,  
    $$\mathcal{A}(\Lambda_\Gamma) \cdot \Psi = 0 \in \widehat{Sk}_\lambda(L_g)$$ 
    has a unique solution $\Psi = 1 + \cdots$, namely the appropriate product of $\mathbf{E}(\partial D)$ for the mutation disks.  
    
    In particular, given any Lagrangian $L' \subset \C^3$
    filling $\Sigma_\Gamma$ which has the same topology (i.e. contracts the same cycles) as $L$, the skein-valued count of curves on $L'$ is this $\Psi$. 
\end{corollary}
\begin{proof}
    Follows from Theorem \ref{uniqueness} and Lemma \ref{admissible implies composable}. 
\end{proof}

\begin{remark} \label{composable more general}
    We see from Lemma \ref{admissible implies composable} that the notion of composable cobordism is much more general than that of admissible mutation, even in the context of cubic planar graphs. 
\end{remark}

The following is useful in calculations: 

\begin{lemma} \label{check homology only}
    Let $(e_1, \ldots, e_n)$ and $(\epsilon_1, \ldots, \epsilon_n)$ be the edges and signs defining a composable sequence $S$ of mutations beginning at the necklace graph, and let $(e'_1, \ldots, e'_m)$ and $(\epsilon'_1, \ldots, \epsilon'_m)$ be another, call it $S'$.

    Assume that the two mutation sequences induce the same map in homology.  Then $\Psi_S = \Psi_{S'} \in \widehat{Sk}(L)$, where $L$ is the standard filling of the necklace graph Legendrian. 
\end{lemma}
\begin{proof}
    The force of the lemma is that we {\em do not} assume that the trivializations induce the same map on simple closed curve classes, so we can't simply invoke the uniqueness result of Theorem \ref{uniqueness} directly.  
    
    Instead, we reason as follows.  Since the maps are the same in homology, it follows that in fact
    $(e_1, \ldots, e_n, \tilde e_m, \ldots, \tilde e_1)$
    and $(\epsilon_1, \ldots, \epsilon_n, - \epsilon'_m, \ldots, -\epsilon'_1)$ define a composable sequence,  now from the necklace graph to itself.  We write $\overline{S}'$ for the second half of this sequence.  
    
    By Corollary \ref{relation for admissible filling}, we find that the image of  $\Psi_{\overline{S}'} \Psi_S$ in $\widehat{Sk}(L)$ is annihilated by $\mathbf{A}(\Lambda_{Neck})$.  Now by Theorem  \ref{uniqueness} and Lemma \ref{Lift lives in subcompletion}, we may conclude $\Psi_{\overline{S}'} \Psi_S = 1 \in \widehat{Sk}(L)$.  On the other hand, we may apply exactly the same reasoning to conclude 
    $\Psi_{\overline{S}'} \Psi_{S'} = 1 \in \widehat{Sk}(L)$.
    Cancelling
    on the left gives the desired result.
\end{proof}

\begin{figure}
    \begin{picture}(50,125)
    \put(-150,0){\includegraphics[width=15cm, trim=0cm 4cm 0cm 0cm, clip]{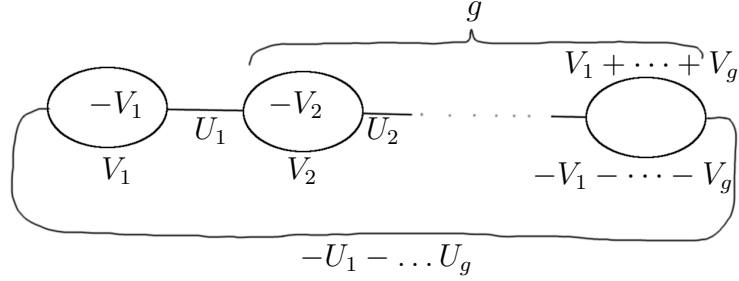}}
    \put(63,112){$g$}
    \put(-75,50){$V_1$}
    \put(-80,75){$-V_1$}

    \put(-5,50){$V_2$}
    \put(-12,75){$-V_2$}
    
    \put(-40,65){$U_1$}
    \put(25,65){$U_2$}

    \put(100,90){$V_1+\dots +V_g$}
    \put(88,48){$-V_1-\dots-V_g$}

    \put(0,17){$-U_1-\dots U_g$}
    
    \end{picture}
    \caption{The necklace graph of genus $g$ with a special choice of basis. (The quantum torus version of this basis appears in \cite[page 54]{SSZ}.) }   
    \label{Fig:Necklace-basis}  
\end{figure}

We will now study how the diffeomorphism
$\Lambda_\Gamma \cong \Lambda_{\Gamma_e}$ act on isotopy classes of simple closed curves, rather than just on homology.

\begin{lemma}\label{dehn twist}
    Fix an edge $e$ of a CPG $\Gamma$.  We consider the mutation at $e$.  As always, denote by $[\mathrm{edge}]$ the oriented closed curve associated to the edge.  
    We preserve the notations $\tilde e$, $e_i$, $\tilde e_i$ from Lemma \ref{homology mutation}.  

    Under the diffeomorphism $\Lambda_\Gamma \cong \Lambda_{\Gamma_e}$, we find that $[\tilde e]$ is the same curve as $[e]$ but with reversed orientation. 
    
    If the mutation was positive then for $i=2,4$ the closed curves $[\tilde e_i]$ and $[e_i]$ are homotopic while if $i=1,3$ $[\tilde e_i]$ is 
    the Dehn twist of $[e_i]$ around $[e]$. If the mutation is negative then the two cases are reversed.
\end{lemma}

\begin{figure}
    \begin{picture}(20,155)
    \put(-100,-30){\includegraphics[width=12cm, trim=0cm 3cm 5cm 0cm, clip]{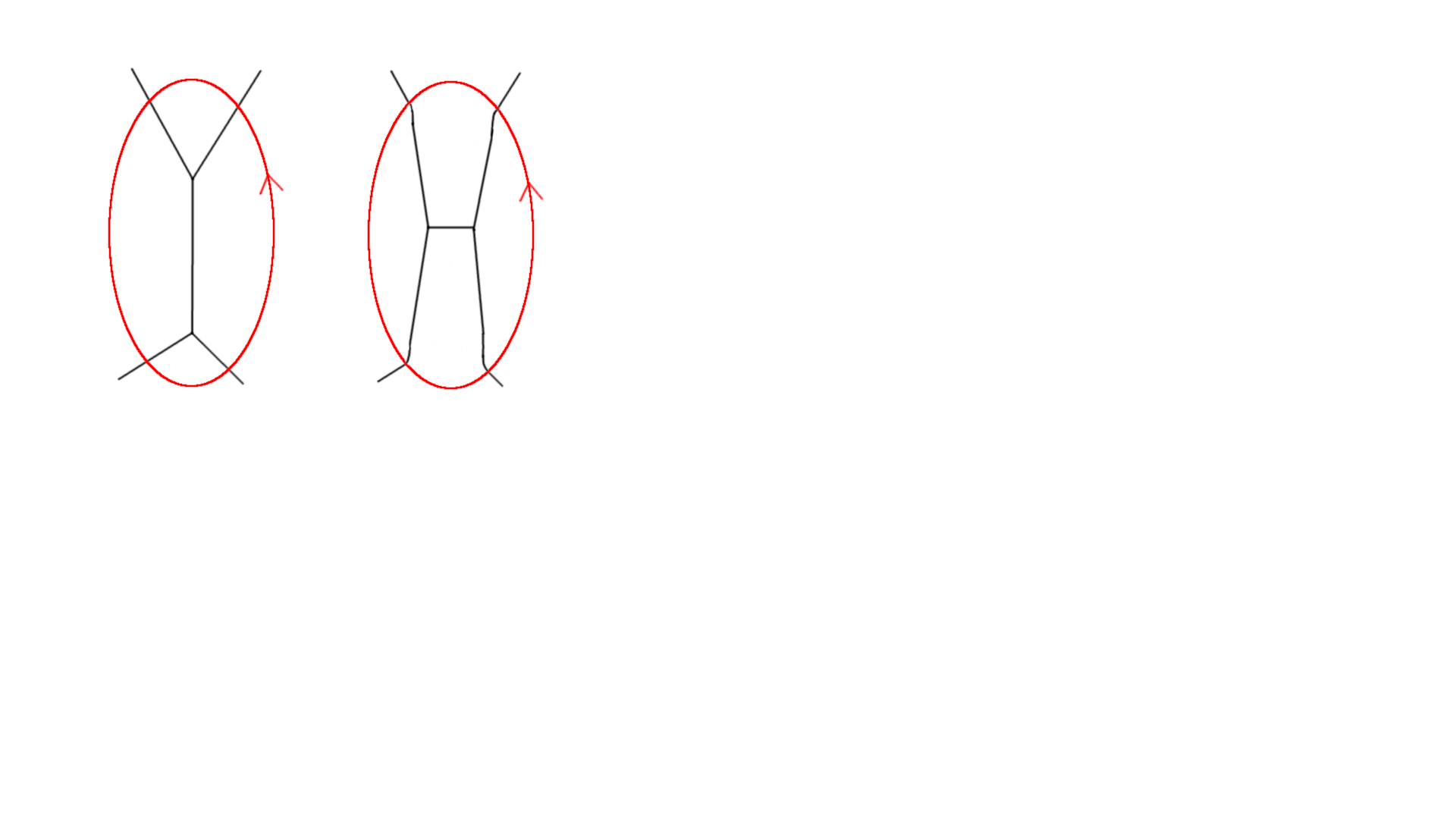}}

    \end{picture}
    \caption{Neighborhoods of the edges $e$ and $\tilde{e}$. Indicated in red is the curve which has two lifts both of which lift to a neighborhood of the boundary. Note that the curve on the left is the positively oriented lift of $e$ while the curve on the right is the negatively oriented lift of $\tilde{e}=-e$.  }  
    \label{Fig:Dehn-Twist-Annulus}
\end{figure}

\begin{figure}
    \begin{picture}(20,250)
    \put(-170,0){\includegraphics[width=15cm, trim=0cm 0cm 0cm 0cm, clip]{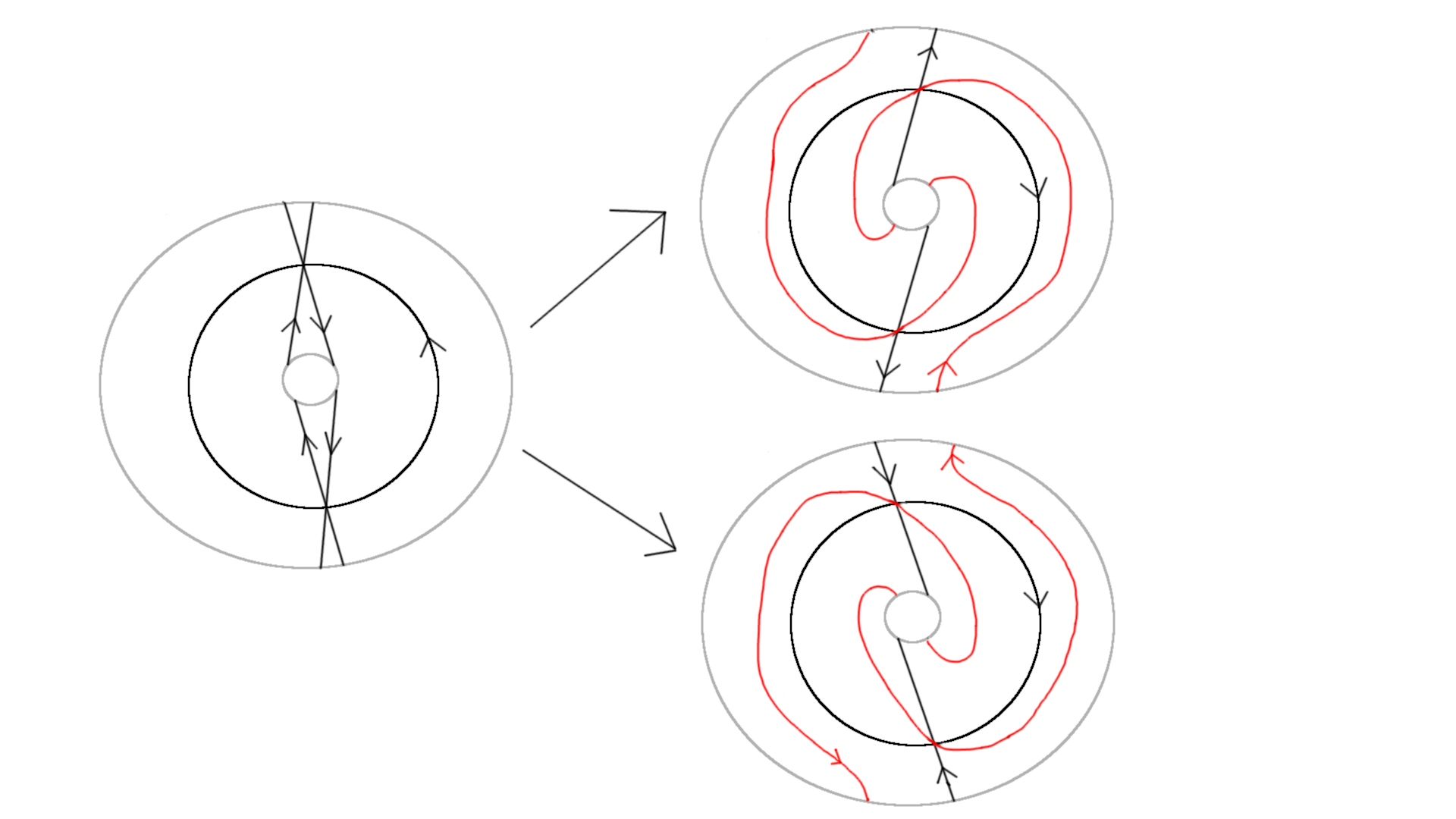}}

    \end{picture}
    \caption{On the left the annulus $A$ from Figure \ref{Fig:Dehn-Twist-Annulus} when viewed on the surface underlying $\Lambda_\Gamma$ and $\Lambda_{\Gamma_e}$. Indicated in black and red are the double lifts of the edges within this neighborhood. On the left pre-surgery and on the right after negative (above), resp. positive (below) mutation. In red are the curves which transform by Dehn twist while in black are the ones which do not.}  
    \label{Fig:Half-Dehn}
\end{figure}

\begin{proof}
    Basically the point is that the transformation is confined to an annulus, and in an annulus the action on relative homology determines the action on the space of simple paths. 

    In more detail, the statement about $[e]$ and $[\tilde{e}]$ is proven in Figure \ref{Fig:Propagate-1}. To observe the other case, let us fix positive mutation and a topological trivialisation of $L_e \# \mathbf{T} \cong \Lambda_\Gamma \times \R$ as in Lemma \ref{Lem: CPG topological filling}. Outside the annulus $A$ indicated in Figure \ref{Fig:Dehn-Twist-Annulus} this trivialisation coincides with the cylindrical structure of $(\Lambda_\Gamma \setminus A) \times \R$ and with $\Lambda_\Gamma \times (-\infty, -R]$ and $\Lambda_{\Gamma_e} \times [R,\infty)$ for $R$ large enough. Outside of the cylinder $A\times \R$ the double lifts of $e_1,\dots,e_4$ and $\tilde{e}_1,\dots \tilde{e}_4$ coincide. 
    
    So, we fix a chart of $A \times \{R\}$ such that the double lifts edges $e_1,\dots,e_4,e$ in $A$ look as on the left of Figure \ref{Fig:Half-Dehn}, compare \cite[Section 4.6]{Treumann-Zaslow}. We extend this chart to $A \times [-R,R]$ (using the trivialisation) and draw the edges $\tilde{e}_1,\dots \tilde{e}_4$ above $A$ in this chart. Then by the condition that the homology classes of $[\tilde{e}_i]$ and $[e_i]$ are related as in Lemma \ref{homology mutation}. There is up to homotopy only the possibility for the $[\tilde{e}_i]$ to transform by Dehn twist around $[e]$ if $i=1,3$ and (up to homotopy) coincide with the $[e_i]$ if $i=2,4$. As the $[\tilde{e}_i]$ are inserted in a single layer $\Sigma \times \{R\}$ they are framed with the blackboard framing and thus they do not link with themselves, so the $[\tilde{e}_i]$ are completely determined. The resulting lifts are depicted in the top right of Figure \ref{Fig:Half-Dehn}. One proves the case of negative mutation similarly and this situation is depicted in Figure \ref{Fig:Half-Dehn} on the bottom right.
\end{proof}

\section{Towards a skein-valued pentagon relation} \label{sec: 5 terms} \label{sec: Full 5-term}

\begin{figure}
   \includegraphics{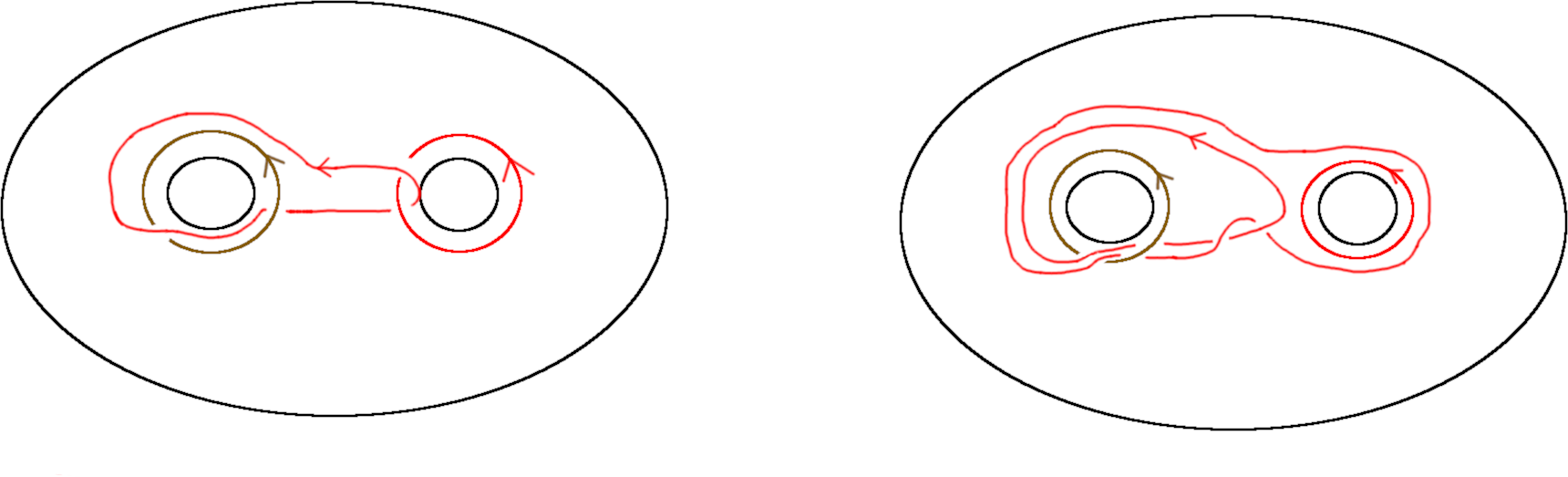}
    \caption{The links of equation \eqref{not five term}.  The $e_0$ terms are drawn in brown.}  
    \label{Fig:Pentagon}
\end{figure}

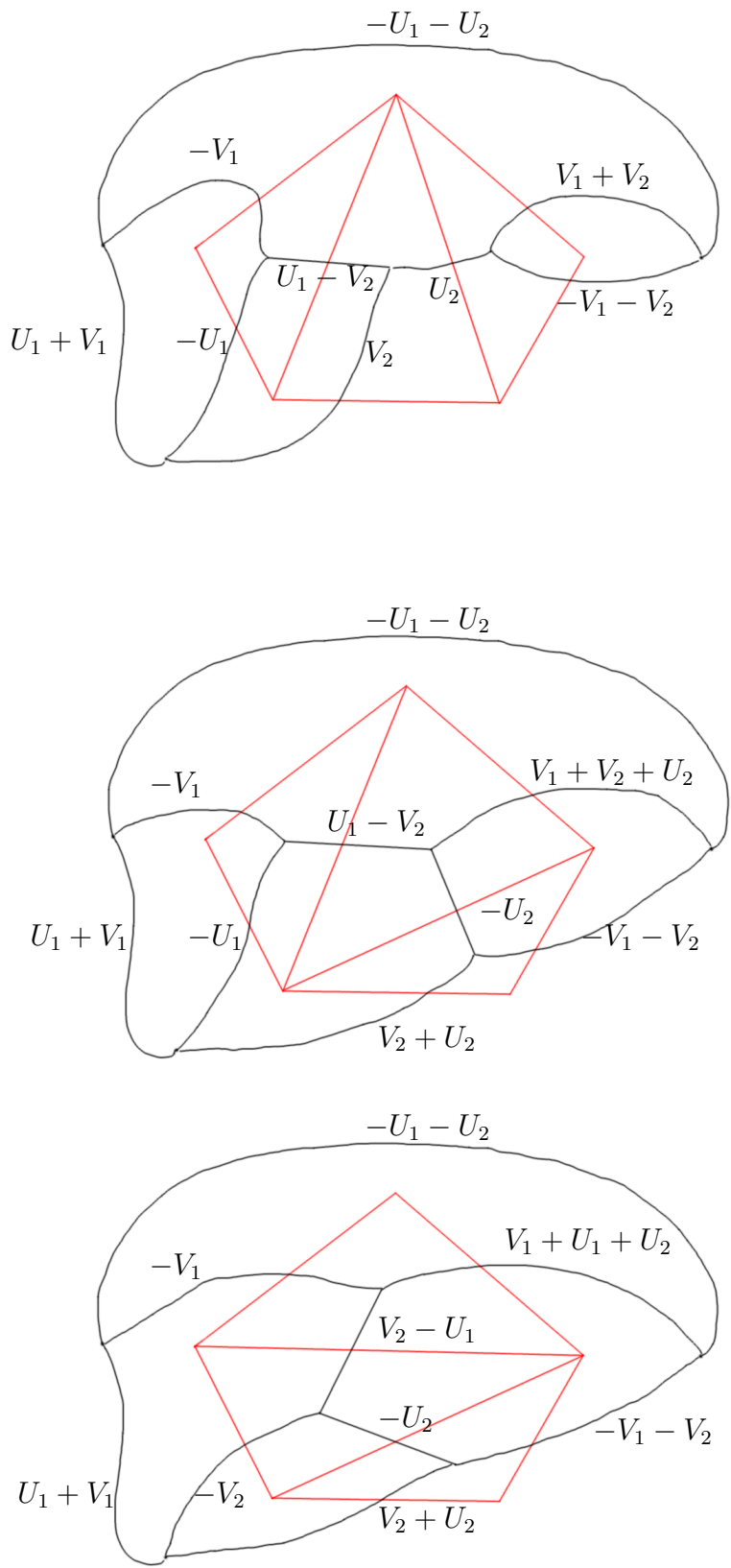
\begin{figure}
    \begin{picture}(20,600)
    \put(-150,-30){\includegraphics[width=10cm, trim=0cm 5cm 0cm 0cm, clip]{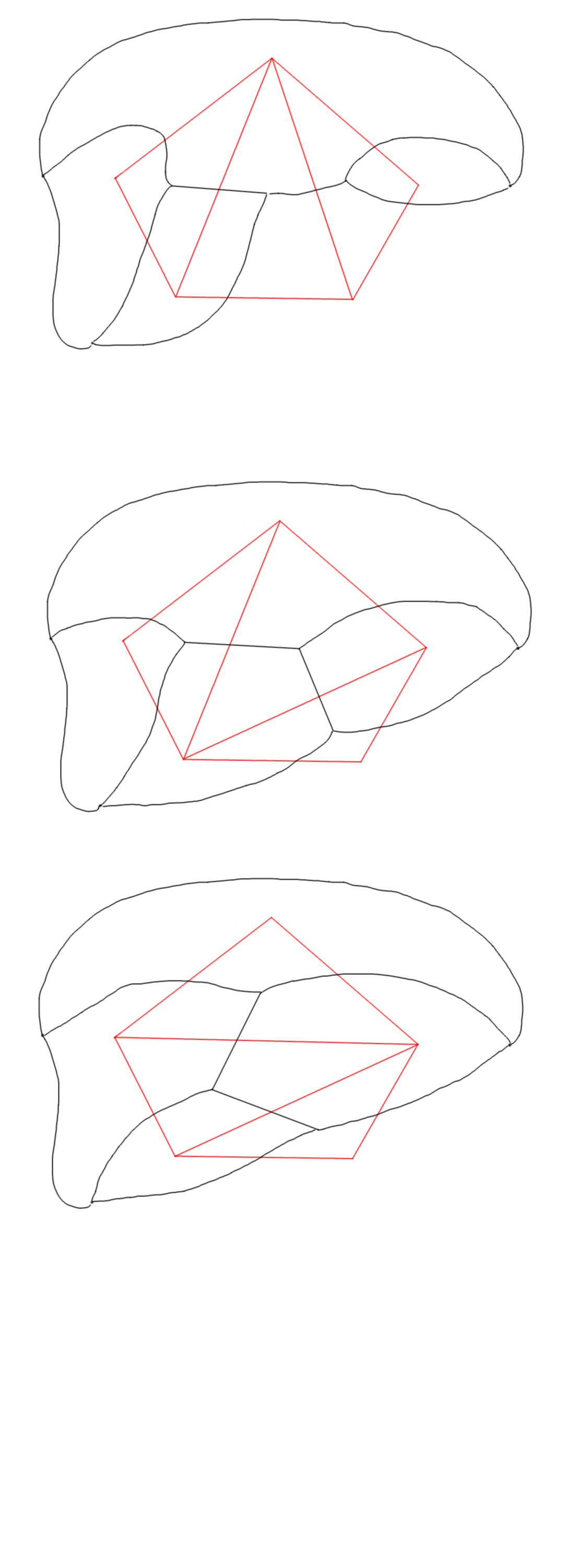}}

    \put(-60,500){$U_1-V_2$}
    \put(0,495){$U_2$}
    \put(50,540){$V_1+V_2$}
    \put(50,490){$-V_1-V_2$}
    \put(-25,600){$-U_1-U_2$}
    \put(-95,550){$-V_1$}
    \put(-165,475){$U_1+V_1$}
    \put(-100,475){$-U_1$}
    \put(-25,470){$V_2$}

    \put(-40,285){$U_1-V_2$}
    \put(20,250){$-U_2$}
    \put(40,305){$V_1+V_2+U_2$}
    \put(60,240){$-V_1-V_2$}
    \put(-25,365){$-U_1-U_2$}
    \put(-110,300){$-V_1$}
    \put(-157,240){$U_1+V_1$}
    \put(-95,240){$-U_1$}
    \put(-20,200){$V_2+U_2$}

    \put(-20,85){$V_2-U_1$}
    \put(-20,50){$-U_2$}
    \put(30,120){$V_1+U_1+U_2$}
    \put(65,45){$-V_1-V_2$}
    \put(-25,165){$-U_1-U_2$}
    \put(-110,110){$-V_1 $}
    \put(-162,20){$U_1+V_1$}
    \put(-93,20){$-V_2$}
    \put(-20,10){$V_2+U_2$}

    \end{picture}
    \caption{At the top the  graph $\Gamma$ obtained after one mutation of the necklace graph.  Middle,  the graph obtained after negatively mutating at the edge of $\Gamma$ labelled $U_2$.  Bottom, the graph obtained by mutating the middle graph at the edge labelled $U_1-V_2$ label.}  
    \label{Fig:5-term-1}
\end{figure}

\begin{figure}
    \begin{picture}(20,600)
    \put(-150,-30){\includegraphics[width=10cm, trim=0cm 6cm 0cm 0cm, clip]{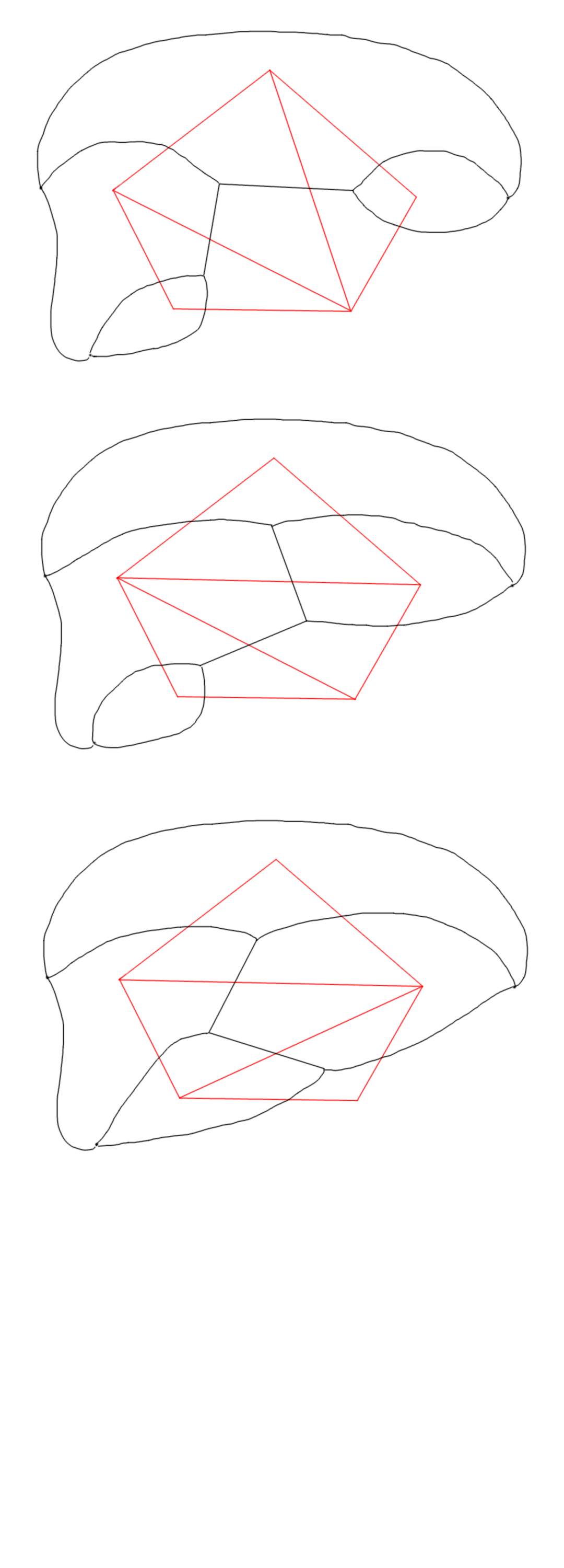}}

    \put(-40,455){$V_2-U_1$}
    \put(-32,485){$U_1+U_2-V_2$}
    \put(50,500){$V_1+V_2$}
    \put(50,448){$-V_1-V_2$}
    \put(-25,560){$-U_1-U_2$}
    \put(-110,505){$-V_1$}
    \put(-165,425){$U_1+V_1$}
    \put(-110,425){$-V_2$}
    \put(-50,400){$V_2$}

    \put(-30,255){$U_2$}
    \put(0,285){$V_2-U_1-U_2$}
    \put(41,314){$V_1+U_1+U_2$}
    \put(50,250){$-V_1-V_2$}
    \put(-25,365){$-U_1-U_2$}
    \put(-110,310){$-V_1$}
    \put(-160,240){$U_1+V_1$}
    \put(-105,240){$-V_2$}
    \put(-50,200){$V_2$}

    \put(-20,85){$V_2-U_1$}
    \put(-20,50){$-U_2$}
    \put(30,120){$V_1+U_1+U_2$}
    \put(65,45){$-V_1-V_2$}
    \put(-25,165){$-U_1-U_2$}
    \put(-110,110){$-V_1 $}
    \put(-162,20){$U_1+V_1$}
    \put(-110,20){$-V_2$}
    \put(-10,8){$V_2+U_2$}

    \end{picture}
    \caption{The other path of mutations. At the top the graph obtained from $\Gamma$ by negatively mutating at the edge labelled by $U_1-V_2$. The middle graph is obtained from negative mutation at the edge labelled $U_1+U_2-V_1$. The bottom graph is obtained by negative mutation at the edge labelled by $U_2$.}  
    \label{Fig:5-term-2}
\end{figure}

Consider the  graph $\Gamma$ described in Figure \ref{Fig:5-term-1}.  It results from a negative mutation of the necklace graph at some edge $e_0$.  Figures \ref{Fig:5-term-1} and \ref{Fig:5-term-2} describe sequences of negative mutations continuing from this graph. 
We denote the mutation edges appearing in the longer sequence as $e_1,e_2,e_3$, and those in the shorter sequence as  $e'_1,e'_2$.  

\begin{lemma} \label{same graphs and homology classes}
    These two sequences of mutations end in the same final graph, with the same final homology classes.
\end{lemma}
\begin{proof}
    The calculation is done in the Figures \ref{Fig:5-term-1} and \ref{Fig:5-term-2} .
\end{proof}

By Lemma \ref{check homology only}, 
we conclude
\begin{equation} \label{not five term}
  \mathbf{E}([e_3])\mathbf{E}([e_2])\mathbf{E}([e_1])\mathbf{E}([e_0]) =  \mathbf{E}([e'_2])\mathbf{E}([e'_1])\mathbf{E}([e_0]) \in \widehat{Sk}(L_2)  
  \end{equation}
We draw the corresponding knots in Figure \ref{Fig:Pentagon}.

As it appears on the right, we cannot directly cancel the $\mathbf{E}([e_0])$
to obtain the more natural analogue of the pentagon identity
\begin{equation}\label{skeinpentagon}
    \mathbf{E}([e_3])\mathbf{E}([e_2])\mathbf{E}([e_1]) \stackrel{?}{=}  \mathbf{E}([e'_2]) \mathbf{E}([e'_1])
\end{equation}

\begin{remark}
After the appearance of this article as a preprint, Equation \eqref{skeinpentagon} was proven by \cite{Nakamura} and \cite{Hu},
and generalized in \cite{Scharitzer-An}. 
\end{remark}

\appendix

\section{$4$-chains for the disk surgery cobordisms}

In this Appendix, we explain how to fix various framing factors left ambiguous in the body of the paper by appropriate explicit choices of 4-chain on the cobordisms. 

\subsection{$4$-chain for the cobordism of Construction \ref{Con:Basic cobordism}}

\begin{definition} \label{4-chain triviality}
    Let $L \subset J^1(M) \times \R \subset Y \times \R$ be a disk surgery cobordism between $\Lambda$ and $\Lambda_D$ which coincides with a trivial cobordism outside the cylinder above some ball $B^5 \subset J^1(M)$. Let $K$ be a $4$-chain bounding $2L$ as described at the beginning of Section \ref{what perfect cobordisms do}. We say that $K$ induces trivial framing lines, if it fulfills the following properties:
    \begin{enumerate}
        \item $K^\circ \cap L \cap (B^5 \times \R)^c = \mathfrak{p} \times \R$ for some set of oriented points $\mathfrak{p} \subset \Lambda$;
        \item $\mathfrak{q}= K^\circ \cap L \cap (B^5 \times \R)$ is the boundary of an oriented $2$-chain inside $L$;
        \item $K$ does not intersect any Reeb trajectory of $\Lambda \subset J^1(M)$.
    \end{enumerate}
\end{definition}

\begin{construction} \label{Con: 4-chain of Basic cobordism}
    Let $\Lambda$ be a Legendrian in $J^1(M) \subset Y $ and $D$ a surgery locus such that $\Lambda_D$ is Reeb-compatible with $\Lambda$. Assume furthermore, that all Reeb chords of $\Lambda \subset J^1(M)$ are pairwise disjoint. Let $L$ be the disk surgery cobordism from Construction \ref{Con:Basic cobordism}. Then there is a $4$-chain $K$ with boundary $2L \sqcup -2L'$ where $L$ and $L'$ are isotopic, $L \cap L' = \emptyset$ and the framing lines induced by the intersection $K^\circ \cap L$ fulfills the properties of Definition \ref{4-chain triviality}.
\end{construction}

\begin{proof}
    First, we recall that for a trivial cobordism $\Lambda \times \R \subset J^1(M) \times \R$ there is a $4$-chain $K' \times \R$ \cite[Proof of Construction 5.3]{Scharitzer-Shende}, fulfilling the following properties:

    \begin{enumerate}
        \item $K' \cap J^1(M)$ is arbitrarily close to the positive and negative Reeb push-offs of $L$ after some small time.
        \item The intersections $K'^\circ \cap \Lambda$ are in $1:1$ correspondence with the Reeb chord endpoints and lie arbitrarily close to them.
        \item $K'^\circ$ is disjoint from the Reeb chords of $\Lambda$.
        \item The vectorfield $v$ in the neighborhood $\Lambda \times \R$ is close to the Reeb vectorfield which is dual to $\frac{d}{dt}$, the vectorfield of the $\R$ direction. 
    \end{enumerate}

    So, we only need to discuss how to extend this $4$-chain into the non-cylindrical pieces of $L$. In Construction \ref{Con:Basic cobordism} there appear two different types of non-cylindrical pieces. One are the ambient surgeries and the other one are the Lagrangian cobordisms coming from Legendrian isotopies. In particular, we need to show that all pieces of the framing lines $1$-chain $\mathfrak{p}$ appearing here are homologically trivial. By construction, there are $4$ pieces of the framing lines before the surgeries ambient surgeries, see the left of Figure \ref{Fig:Surgery-1}. We arbitrarily extend the $4$-chain into the handle of the index $1$ surgery such that outside a small ball $B_1$ around the handle it again agrees with the standard model. The topology of the handle joins two $D^3$ disks with one another, so relative to the boundary any components of $\mathfrak{p}$ appearing there are homologically trivial and we know by the right of Figure \ref{Fig:Surgery-1} that only two additional components of the framing line leave the boundary of $(B_1 \times \R) \cap L$. We proceed in the same way around the index $0$ surgery. Here there is a small ball $B_0$ such that $(B_0 \times \R )\cap L$ topologically coincides with an annulus with a $2$-handle attached along its central ring. So again, all different ways to extend the $4$-chain into the interior lead to homologically trivial components of $\mathfrak{p}$. Again, as can be seen on the left of Figure \ref{Fig:Surgery-2} two additional components of the framing lines appear. Finally, the last piece of the cobordism is a Lagrangian cobordism coming from Legendrian isotopy, as described in Construction \ref{Con:Basic cobordism} this cobordism can be chosen arbitrarily close to the trivial cobordism. So we may, extend the $4$-chain in such a way that the framing lines follow the paths traced out by the Reeb chord endpoints arbitrarily close. We see that the index $2$ Reeb chord stays within a ball $B_1'$ which is slightly larger than $B_1$ but $(B_1' \times \R )\cap L$ is topologically the same as $B_1 \cap L$. As seen on the right of Figure \ref{Fig:Surgery-2} the framing lines cancel each other. Since  $(B_1' \times \R)\cap L$ is topologically a ball $B_1'$ all framing lines are bounded by a $2$-chain relative to the boundary. Since both endpoints live in the same connected component of $((B_1' \times \R)\cap L)_{-\infty}$ we can choose a bounding $2$-chain which lives entirely in $((B_1' \times \R)\cap L)_{-\infty}$. A similar argument shows that the remaining components of the framing lines are also bounded by a $2$-chain relative to the negative end of $L$.
\end{proof}

\begin{lemma}
    Let $K$ be the $4$-chain from Construction \ref{4-chain triviality}. If $2L'$ bounds a $4$-chain in the complement of $L \sqcup \bigsqcup C_\rho \subset Y \times \R$ where $C_\rho$ is the cylinder above a Reeb chord $\rho$ of $\Lambda$. Then the equation from Theorem \ref{cluster transformation from perfect cobordism} specialises to:
    $$\mathbf{A}(\Lambda_D,\rho_+) \mathbf{E}(\partial D) = \mathbf{E}(\partial D)  \mathbf{A}(\Lambda,\rho_-)$$
    which is an equation in the appropriate completion of the skein-module $Sk(\Lambda \times \R, \mathfrak{p} \times \R)$ where $\mathfrak{p}$ is the set of oriented endpoints of Reeb chords.
\end{lemma}

\begin{proof}
    The $4$-chain from Construction \ref{Con: 4-chain of Basic cobordism} is disjoint from all cylinders above Reeb chords of $\Lambda$, so if we can close the remaining boundary of the $4$-chain in the complement of $L$ and these cylinders then we do not introduce new sign/framing lines from intersections with $L$ nor change the framing of the trivial cylinders. Finally, as the framing lines above the surgery locus are homologically trivial, we may choose some $2$-chain to cancel them. Note, that in this way, we may change the $\gamma$ appearing in the equation $(1-m-\gamma l)\mathbf{E}=0$, however this is equivalent to a change of framing of $\partial D$.
\end{proof}

\subsection{4-chain for the cobordism of Construction \ref{CPG: Cobordism}}
\begin{construction} \label{Con: 4-chain of CPG corbordism}
    Let $L$ be the disk surgery cobordism from Construction \ref{CPG: Cobordism}. Then there is a $4$-chain $K$ bounding $2L$ which induces trivial framing lines. 
\end{construction}

\begin{proof}
    We proceed exactly, as in the proof of Construction \ref{Con: 4-chain of Basic cobordism}. We only need to modify the argument that the framing and sign lines $K^\circ \cap L$ inside the surgery locus $B^5 \times \R$ are bounded by a $2$-chain. As in Construction \ref{Con: 4-chain of Basic cobordism}, we will not argue about the framing and sign lines close to surgery handles since all these choices are contractible relative to the boundary of a small ball containing the handle. In fact, we will only focus on the paths traced out by Reeb chord endpoints.

    Recall that the first stage of Construction \ref{CPG: Cobordism} was an ambient index $1$ surgery followed by a Lagrangian cobordism coming from a Legendrian isotopy. Consider Figure \ref{Fig:CPG-Cobord-1} and take $B \subset S^5$ to be a ball containing some vertex $v$ adjacent to $e$ and the complement of one of the triangles of $T_{Cl}$. Then the Legendrian isotopy can be chosen so that none of the Reeb chords outside this triangle pass through it. On the right of Figure \ref{Fig:CPG-Cobord-2}, we observe that all Reeb chords save the one inside the chosen triangle have cancelled with one another. As all Reeb chords outside the triangle have cancelled with one another and the intersection $B \cap L$ is topologically a ball, we again see that there is a $2$-chain bounding the framing lines relative to the boundary of this ball. Next, we turn to the second and third stage of Construction \ref{CPG: Cobordism}. Here we consider the following subset: Consider the triangle on the right of Figure \ref{Fig:CPG-Cobord-2}. Then let $B_f$ be a ball containing the Reeb chord endpoints and the upper vertex of the triangular face. Let $B_f \times (a,b)$ be an appropriate extension of this set into the previous stages so that $B_f \times (a,b)$ is topologically a ball. Recall that $B_f$ is indeed a disk as it is the branched double cover of a disk with a single branching point. Furthermore, let $B'_e$ be a neighborhood of the edge $e$ on the right of Figure \ref{Fig:CPG-Cobord-2} and let $B'_e \times [b,c]$ be an appropriate extension so that it contains all framing lines. Topologically $B'_e \cap L$ is an annulus but $B'_e \times [b,c] \cap L$ is an annulus with a $2$-handle attached, so again it is topologically a 3-ball. Taking the union with $B_f \times (a,b) \cap L$ doesn't change the topology as the sets can be chosen such that the intersection $B_e \times B_f$ is a disk. So finally, again $\mathfrak{p}$ inside this $3$-ball is bounded by a $2$-chain relative to the boundary of this $3$-ball. All endpoints live in the same connected component of $L_{-\infty}$, so in fact these $2$-chains can be chosen relative to $L_{-\infty}$.

    Finally, we  argue that the boundary component $2L'$ is bounded by a $4$-chain $K'$ which intersects neither $L$ nor any of the cylinders above the Reeb chords of $\Lambda$. To do so, take a large ball $B^5 \subset S^5$ which contains all slices $L \cap S^5 \times \{d\}$ for $d$ arbitrary. Then $(S^5 \setminus B^5) \times \R$ is contractible and contains $2L'$. Thus we can find a $4$-chain $K'$ which doesn't contribute new sign lines or intersections to the curves bounding $L$. Now $K + K'$ is the desired $4$-chain.
\end{proof}

\begin{proposition} (Fixing the framing of Theorem \ref{cubic planar graph cluster}) \label{Fixed framing CPG cluster}
    Let $\Gamma$ be a cubic planar graph Legendrian and $e$ a non-loop edge. We write $e^+$ for the corresponding oriented simple closed curve and $e^-$ for the curve with opposite orientation. Then for each face $f$ of $\Gamma$, we have the following equations, holding, respectively in completions of the skein in the directions of $e_+$ and $e_-$: 

        $$ \mathbf{A}(\Lambda_{\Gamma_e},f) \mathbf{E}(e_+) = 
    \mathbf{E}(e_+)  \mathbf{A}(\Lambda_\Gamma,f) $$
    $$ \mathbf{A}(\Lambda_{\Gamma_e},f) \mathbf{E}(e_-) = 
    \mathbf{E}(e_-)  \mathbf{A}(\Lambda_{\Gamma},f) $$

    where both $e_+$ and $e_-$ are given the blackboard framing and $\mathbf{E}(e_+)$ is the skein-valued lift of the $q$-dilogarithm with scaling $\gamma=q^\frac{1}{2}$ where $\mathbf{E}(e_-)$ has the scalar $\gamma=-q^\frac{1}{2}$, see Section \ref{skein dilogarithm}. 
\end{proposition}

\begin{proof}
    For definiteness, we will show the equation for $e_+$. The equation for $e_-$ is shown similarly. Since the equation is true for all face equations of the graph. By Construction \ref{Con: 4-chain of CPG corbordism}, we can choose the framing factor of Theorem \ref{cluster transformation from perfect cobordism} to be $\gamma=1$. It is sufficient to check what the framing of $e_+$ is for one of the equations. Let $f$ be a face such that $e$ is a boundary of $f$ in $\Gamma$. Then we obtain the following two equations from \cite[Theorem 7.1.]{Scharitzer-Shende}:

    \begin{eqnarray}
        \mathbf{A}(\Lambda_{\Gamma_e},f) &=& \sum_{v \in f \in \Gamma_e} [l_v] \\
        \mathbf{A}(\Lambda_\Gamma,f) &=& \sum_{v\in f \in \Gamma} [l_v]
    \end{eqnarray}
    The difference between here and the cited Theorem, only lies in the fact that we have not chosen capping paths. So the tangles $l_v$ are the double cover of a path from a chosen midpoint of $f$ to the vertex $v$ and where $e_{+,\mathfrak{f}}$ is the same curve as $e_+$ but the framing $\mathfrak{f}$. We observe directly, that all such tangles commute with the element $\mathbf{E}(e_{+,\mathfrak{f}})$ if $v$ is not one of the vertices adjacent to the edge $e$. So we may reduce the equation:

    $$ \mathbf{A}(\Lambda_{\Gamma_e}) \mathbf{E}(e_{+,\mathfrak{f}}) = 
    \mathbf{E}(e_{+,\mathfrak{f}})  \mathbf{A}(\Lambda_\Gamma) $$

    To:

    $$ ([l_{v}]) \mathbf{E}(e_{+,\mathfrak{f}}) = 
    \mathbf{E}(e_{+,\mathfrak{f}})  ([l_{v_1}] + [l_{v_2}]))$$

    For convenience, we also multiply by the inverse of $\mathbf{E}(e_{+,\mathfrak{f}})$ from the left:

    $$ \mathbf{E}(e_{+,\mathfrak{f}})^{-1} ([l_{v}]) \mathbf{E}(e_{+,\mathfrak{f}}) = ([l_{v_1}] + [l_{v_2}]))$$

    where $v_1,v_2$ are the vertices bounding e such that $v_2$ follows $v_1$ in clockwise order and $v$ is the unique vertex bounding $f$ after mutation. Choosing the trivialisation given in Lemma \ref{dehn twist}, we can identify the tangles $[l_v]$ and $[l_{v_1}]$ with one another. So, these represent the same homology class $\delta$ relative to the midpoints of faces. Whereas the underlying homology class of $l_{v_2}$ is $\delta+\epsilon$ where $\epsilon$ is the underlying homology class of $e_+$. In this relative homology filtration, we can split the equation into equations for each $\delta + k \epsilon$ where $k$ is some non-negative integer. We focus on the equation on degree $k=1$.

    $$ -\frac{\gamma}{q^{\frac{1}{2}}-q^{-\frac{1}{2}}} [e_{+,\mathfrak{f}}] [l_{v}]  +  \frac{\gamma}{q^{\frac{1}{2}}-q^{-\frac{1}{2}}} [l_{v}] [e_{+,\mathfrak{f}}] = [l_{v_2}]$$

    where we refer to \cite[Proposition 3.2.]{unknot} for the expansion of the first terms of the skein-valued quantum dilogarithm. \footnote{Note that their lift is shifted by $q^\frac{1}{2}$ relative to our convention.} Once we substitute $[l_v]=[l_{v_1}]$ and multiply by $q^{\frac{1}{2}}-q^{-\frac{1}{2}}$,we obtain from the HOMFLY-identities exactly the statement that $l_{v_2}$ is obtained from $l_{v_1}$ by Dehn twisting along $e_+$. All of these tangles are given the blackboard framing thus, for the above equation to be true, $\mathfrak{f}$ must be the blackboard framing along $e_+$.
\end{proof}

Similarly, we can formulate the above equations with capping paths:

\begin{corollary}
    Let $\Gamma$ be a cubic planar graph Legendrian and $\lambda$ a $1$-form. Let $(e_1,\dots,e_n)$ and $\epsilon_1,\dots,\epsilon_n)$ be the edges and signs defining a composable sequence of mutations beginning at $\Gamma$. Assume furthermore that $v$ is a vertex of a face $f$ which is not the vertex of any of the $e_i$. Then 

    $$ \mathbf{A}(\Lambda_{\Gamma_n},f,v) \Psi = 
    \Psi  \mathbf{A}(\Lambda_\Gamma,f,v) $$

    where $\mathbf{A}(\Lambda_{\Gamma_n},f,v)$ and $\mathbf{A}(\Lambda_\Gamma,f,v)$ are the skein elements defined in \cite[Theorem 7.1.]{Scharitzer-Shende} and $\Psi$ is the product of $\mathbb{E}(e_i^\pm)$ where each $e_i$ is endowed with the blackboard framing and each $\mathbb{E}(e_i^\pm)$ has the appropriate shifting scalar, as in Proposition \ref{Fixed framing CPG cluster}.
\end{corollary}

\begin{proof}
    The expression for $\mathbf{A}(\Lambda_{\Gamma_n},f)$ and $\mathbf{A}(\Lambda_\Gamma,f)$ follows immediately from Proposition \ref{Fixed framing CPG cluster}. To obtain this version of the statement, we only need to note that to obtain $\mathbf{A}(\Lambda_{\Gamma_n},f,v)$ and $\mathbf{A}(\Lambda_\Gamma,f)$ from these tangles, we chose the capping paths as in \cite[Theorem 7.1.]{Scharitzer-Shende}. Attaching the capping paths immediately yields the left-hand side of the equation. However, the right hand side, now is no longer a proper product of elements as the capped tangles now have components above and below $\Psi$. However, as we chose a capping path going through the vertex $v$ which is disjoint from any element of $\Psi$, we may commute this component of the capped tangles below all elements of $\Psi$ and thus reobtain the desired product structure on the right hand side of the equation.
\end{proof}

\section{Relation to cluster algebra
in 4d symplectic geometry}

In this article, we showed that skein-valued cluster transformations arise naturally when considering how curve counts transform for disk surgery of 2d Legendrians in a 5d contact manifold. 
Previous works have investigated cluster transformations which appear in the context of disk surgery of 2d Lagrangians in a 4d exact symplectic manifold.  In this appendix, we explain the relationship.

First let us recall the 4d setup.  
Consider a Weinstein symplectic 4-manifold $X$, and exact Lagrangians $\Lambda, \Lambda'$.   The moduli space of rank one brane structures on $\Lambda$ 
is the torus $\mathcal{M}_1(\Lambda) \cong H^1(\Lambda, \mathbb{C}^*)$; it carries 
a symplectic structure induced from  the Poincar\'e pairing on $\Lambda$.  When $\mathcal{M}_1(\Lambda) \cap \mathcal{M}_1(\Lambda') \ne \emptyset$,\footnote{The intersection could e.g. be taken in the moduli of pseudoperfect modules for the Fukaya category.  Said moduli space is defined finitely presented categories in \cite{Toen-Vaquie}, and finite presentation for the Fukaya category of a Weinstein manifold follows from \cite{GPS3} and standard facts about constructible sheaves.} the intersection defines a rational isomorphism of Poisson tori.  
When $\Lambda$ and $\Lambda'$ are related by Lagrangian disk surgery, the isomorphism can be identified with the 
simplest cluster transformation (see e.g. \cite{Auroux-T, STW, Pascaleff-Tokonog}).   
This (4d, exact) setting already contains many interesting cluster structures previously constructed 
combinatorially, in particular those associated to bicolored graphs on surfaces \cite{STWZ}.  

Now let $W$ 
be a 6d Weinstein symplectic Calabi-Yau manifold, whose ideal contact boundary $\partial_\infty W$ contains
the above $X$ as a Liouville hypersurface.  The examples we considered in the present article have $W = \R^6$.  
We may consider the partially wrapped Fukaya category  $Fuk(W, X)$, and we write 
$\mathcal{M}(W, X)$ for its moduli of pseudoperfect objects. 
For exact $\Lambda \subset X$, we denote also by $\Lambda$ a Legendrian lift to the contactization $X \times \mathbb{R}$.  
We write similarly
$Fuk(W,  \Lambda)$ and $\mathcal{M}(W,  \Lambda)$.

One can see from \cite{GPS1, GPS2} that there are pushout diagrams of categories
and correspondingly pullback diagrams of moduli:  
\begin{equation} \label{pullback diagram}
\begin{tikzcd}
Fuk(W, \Lambda)   & Fuk(T^*\Lambda) \ar{l} & & \mathcal{M}(W,   \Lambda) \ar{r} \ar{d} & \mathcal{M}(\Lambda) \ar{d} \\ 
Fuk(W, X)  \ar{u}  & Fuk(X)  \ar{l} \ar{u} & & \mathcal{M}(W, X) \ar{r} & \mathcal{M}(X) 
\end{tikzcd}
\end{equation}
At present we only know how to prove \eqref{pullback diagram}  if $\Lambda \subset X$ is regular in the sense of \cite{Eliashberg-Ganatra-Lazarev}; 
regularity  enters
in e.g. the construction via \cite{GPS2} of the Viterbo functor $Fuk(X) \to Fuk(T^*\Lambda)$.  For the purpose of studying disk surgery $\Lambda \to \Lambda'$, we may replace
$X$ with the Weinstein manifold whose core is the union of $\Lambda$ and the Lagrangian disk determining the surgery; inside
this latter manifold, $\Lambda$ is evidently regular.

Let us write 
$\mathcal{M}_1(W,  \Lambda) = \mathcal{M}(W,  \Lambda) \times_{ \mathcal{M} (\Lambda)}  \mathcal{M}_1(\Lambda)$
and $\mathcal{A}(\Lambda)$ for the ideal cutting out the image of $\mathcal{M}_1(W, \Lambda)$ in $\mathcal{M}_1(\Lambda)$. 
When $W$ is subcritical, it follows from the punctured surgery formula \cite{Ekholm-Lekili} that $\mathcal{A}(\Lambda)$ 
cuts out what is called the `augmentation variety' in Legendrian contact homology.  Said ideal is generated by counts of holomorphic disks ending on the Reeb chords of index one; since $\mathbf{A}(\Lambda, \rho)$ is correspondingly a count of all genus curves, it follows from the definitions that 
$\mathcal{A}(\Lambda)$ is generated by the dequantizations $\mathbf{A}(\Lambda, \rho)|_{Lk, q=1}$. 

Suppose now $\Lambda, \Lambda' \subset X$ are Lagrangians related by disk surgery. 
A Lagrangian surgery disk lifts to a Legendrian surgery disk, so the Legendrians $\Lambda, \Lambda'$ are related
by Legendrian disk surgery.  Then 
$\mathcal{M}_1(\Lambda)$ and $\mathcal{M}_1(\Lambda')$ are related by cluster transformation, and 
it  follows formally
from the above discussion that said cluster transformation carries 
$\mathcal{A}(\Lambda)$ to $\mathcal{A}(\Lambda')$. 
This is the dequantization of the assertion \eqref{main formula} that the skein-valued counts $\mathbf{A}(\Lambda, \rho)$ transform by skein-valued cluster transformation. 

 %Questions closely related to relationship between $\mathcal{A}(\Lambda)$ and $\mathcal{A}(\Lambda')$ were studied from the point of view of Legendrian contact homology in \cite{Rizell-Surgery}, some results of which will be very important in the present article. 

Finally, we explain how to put the cubic planar graph Legendrians 
in the context of Diagram (\ref{pullback diagram}).  Take $X$ to be the sphere plumbing whose skeleton is obtained from $S^2 = \partial \R^3$ by attaching
unknotted spheres at each vertex of $\Gamma$; this embeds in a standard neighborhood of the $S^2$.  
Indeed, the cluster structure relevant to \cite{Treumann-Zaslow} 
is that associated by \cite{Fock-Goncharov-Teichmuller} to rank two local systems on a sphere 
with unipotent monodromy around vertices $V_\Gamma$ of $\Gamma$.  
From \cite{STWZ} we learn that this variety arises as a moduli space of $Fuk(X')$, where $X'$ is constructed by taking the Liouville manifold
given by the cotangent bundle of $S^2 \setminus V_\Gamma$, and attaching Weinstein handles at the Legendrian lifts of a pair of concentric circles
around each puncture.   It is easy to see $X = X'$.  

We leave the reader to check that the cluster chart Lagrangians of \cite{STWZ}
are sent to the cluster chart Legendrians of \cite{Treumann-Zaslow}.  From this, \cite{STWZ, STW}, and Diagram \eqref{pullback diagram}, one recovers the fact (checked by hand in \cite{Treumann-Zaslow}) that $\mathcal{A}(L_\Gamma)$ and $\mathcal{A}(L_{\Gamma_e})$
are related by cluster transformation. 
The ``Chromatic Lagrangian'' of \cite{SSZ}, there presented locally by charts, 
is then described globally as the appropriate component of $\mathcal{M}(\C^3, X) \to \mathcal{M}(X)$. 

Finally, let us explain why Figure \ref{Fig:Half-Dehn} matches the half Dehn twist introduced in \cite{STW}.  Our above description of $X$ used the \cite{STWZ} presentation.  We could instead present it along the lines of \cite{STW} by taking one of the Lagrangians which give cluster charts, and attaching certain disks.  The attaching loci would be the $[e_i]$ of Figure \ref{Fig:Half-Dehn}.  The half Dehn twist of \cite{STW}  charaterized how such attaching loci transform under disk surgery.

\bibliographystyle{hplain}
\bibliography{skeinrefs}

\end{document}